\documentclass[11pt,a4paper]{article}
\usepackage{graphicx}
\usepackage{color}
\usepackage{xcolor}  
\usepackage[utf8]{inputenc}
\usepackage[normalem]{ulem}
\usepackage{amsmath,amsthm,amsfonts} 
\usepackage{hyperref}

\newcommand\cyl{T}

\newcommand\infcyl{\Theta}

\newcommand{\contbeta}{{{\beta}}}
\newcommand{\conteta}{{{\eta}}}

\newcommand{\mixedgrowth}{\Phi_p}
\newcommand{\mixedgrowthnop}{\Phi}
\newcommand\Elb{E^\mathrm{lb}}
\newcommand\Elast{E^\mathrm{elast}}
\newcommand\Eub{E^\mathrm{ub}}
\newcommand\subcr{\mathrm{subcr}}
\newcommand\core{\mathrm{core}}
\newcommand\moll{\mathrm{moll}}

\renewcommand\ln{\log}

\newcommand\R{\mathbb{R}}
\newcommand\C{\mathbb{C}}
\newcommand\Z{\mathbb{Z}}
\newcommand\N{\mathbb{N}}

\newcommand\calH{\mathcal{H}}

\newcommand\calB{\mathcal{B}}

\newcommand\calM{\mathcal{M}}

\newcommand\calL{\mathcal{L}}
\newcommand\calA{\mathcal{A}}

\newcommand\weakstarto{{\displaystyle\mathop{\rightharpoonup}^*}}
\newcommand\weakto{\mathop{\rightharpoonup}}
\newcommand\e{\varepsilon}

\newcommand\supp{\mathop{\mathrm{supp}}}

\newcommand\Div{\mathop{\mathrm{div}}}
\newcommand\Curl{\mathop{\mathrm{curl}}}
\newcommand\CCurl{\mathop{\mathrm{Curl}}}
\newcommand\curl{\mathop{\mathrm{curl}}}
\newcommand\cof{\mathop{\mathrm{cof}}}
\newcommand\dist{{\mathrm{dist}}}
\newcommand\rel{{\mathrm{rel}}}

\newcommand\Id{{\mathrm{Id}}}
\newcommand\SO{{\mathrm{SO}}}
\newcommand\sym{{\mathrm{sym}}}
\newcommand\skw{{\mathrm{skew}}}

\newcommand\loc{{\mathrm{loc}}}
\newcommand\eps{{\varepsilon}}

\newcommand\dx{{\,dx}}

\newcommand{\LM}[1]{\hbox{\vrule width.2pt \vbox to#1pt{\vfill \hrule width#1pt height.2pt}}}

\newcommand\MBCC{{\mathcal M}^1_{\mathcal B}}

\newcommand{\LL}{{\mathchoice
{\,\LM7\,}{\,\LM7\,}{\,\LM5\,}{\,\LM{3.35}\,}}}
\newtheorem{theorem}{Theorem}[section]
\newtheorem{lemma}[theorem]{Lemma}
\newtheorem{definition}[theorem]{Definition}
\newtheorem{remark}[theorem]{Remark}
\newtheorem{corollary}[theorem]{Corollary}
\newtheorem{proposition}[theorem]{Proposition}
\numberwithin{equation}{section}

\newcommand\psiC{\psi_\C}

\newcommand\HWLin{$H^\mathrm{W}_\mathrm{Lin}$}
\newcommand\HWFinite{$H^\mathrm{W}_\mathrm{Finite}$}

\newcommand{\Flin}{F^{\mathrm{Lin}}}
\newcommand{\Ffinite}{F^{\mathrm{Finite}}}

\newcommand\Cyl{{T}}

\begin{document}
\begin{center}
 {\LARGE 
Line-tension limits for 
line singularities and application to the mixed-growth case
 } \\[5mm]
July 4, 2022\\[5mm]
Sergio Conti$^{1}$, Adriana Garroni$^{2}$, and Roberta Marziani$^{3}$
\\[2mm]
{\em $^1$ Institut f\"ur Angewandte Mathematik,
Universit\"at Bonn\\ 53115 Bonn, Germany }\\
{\em $^{2}$ Dipartimento di Matematica, Sapienza, Universit\`a di Roma\\
00185 Roma, Italy}\\
 {\em $^{3}$
     Angewandte Mathematik, Universit\"at Münster\\
     48149 Münster, Germany
}\\[4mm]
\begin{minipage}[c]{0.8\textwidth}
We study variational models for dislocations in three dimensions in the line-tension scaling. We present a unified approach which allows to treat energies with subquadratic growth at infinity and other regularizations of the singularity near the dislocation lines. We show that the asymptotics
via Gamma convergence is independent of the specific choice of the energy and of the regularization procedure.
\end{minipage}
\end{center}

\tableofcontents

\section{Introduction}

Variational models depending on fields that can have  topological singularities are of particular interest in materials science. 
The asymptotic analysis 
of such models
is often based on tools from geometric measure theory and permits to derive effective energies concentrated on sets of lower dimension.
Important examples range from the study of vortices in superconductors, to grain boundaries, fractures and other interfaces in solids, as well as line defects and disclinations in liquid crystals.
Here we consider the codimension-two case, and in particular
integral energies in three dimensions for fields that are curl-free away from a one-dimensional set, having in mind the important application of the study of dislocations in crystals.

Dislocations are the main mechanism for plastic deformation in metals. They are one-dimensional singularities of the strain field, 
with discrete multiplicity arising from the structure of the underlying crystal lattice.
We treat dislocation models of the type
\begin{equation}\label{eqenergyintro}
 \int_{\Omega} W(\beta) dx,
\end{equation}
where $\Omega\subseteq\R^3$, 
$W:\R^{3\times 3}\to[0,\infty)$ is an elastic energy density, 
$\beta:\Omega\to\R^{3\times 3}$ 
characterizes the strain field and obeys $\curl \beta=\mu$.
{In turn, $\mu$ is} the density of dislocations; which is
a 
divergence-free finite measure of the form $\eps \theta\otimes \tau\calH^1\LL\gamma$, with $\gamma\subset\Omega$ a one-rectifiable curve, $\theta\in L^1(\gamma,\calH^1;\calB)$
the discrete, vectorial multiplicity, $\tau\in L^\infty(\gamma, \calH^1; S^{2})$ tangent to $\gamma$. Here $\calB$ is a 
{three dimensional} 
Bravais lattice which describes the underlying crystalline structure and 
the small parameter $\eps$ represents the ratio between the atomic distance and the size of the sample.
In a finite setting we understand $\Omega$ as the deformed configuration, and $\beta$ as the inverse of the elastic strain; the energy density $W$ is then invariant 
under the action of $\SO(3)$ on the right.
{In the entire paper $\curl f$, for $f$ a distribution taking values in $\R^{3\times 3}$, denotes the rowwise distributional curl, which is also a distribution taking values in $\R^{3\times 3}$.}

If $W$ has quadratic growth, the energy \eqref{eqenergyintro} is infinite whenever $\mu\ne0$. 
In reality, the discrete nature of crystals shows that the appropriate model would be
discrete.
The divergence is a nonphysical result due to the fact that one uses a continuum approximation {which} in the region close to the singularity is not appropriate.
Therefore a number of regularizations {of the continuum model} have been proposed. They include a subquadratic growth of $W$ at infinity, the replacement of $\mu$ by a mollified version {at scale $\eps$}, and
the elimination of a core region from the integration domain. 
For example, the latter corresponds to 
\begin{equation}
 \int_{\Omega\setminus (\gamma)_\eps} W(\beta) dx
\end{equation}
where $(\gamma)_\eps$ denotes an $\eps$-neighbourhood of the curve $\gamma$.
The above models are all considered semi-discrete models, {in the sense that they are continuum models that 
still contain the discrete parameter $\eps$ and some discrete effects, namely, the quantization of the measure and the regularization, both at scale $\eps$.}
In turn, the energy density $W$ can be treated in a geometrically linear setting, or with finite kinematics. Ultimately, all these variants  have  the same leading-order behaviour.

In this paper, we provide a unified treatment of the three-dimensional continuum models which covers many different regularizations, as for example the one with mixed growth conditions,
and clarifies 
the equivalence of all these possible approximations and corresponding regularizations.
Our general approach produces also a simpler and more direct proof of some results from the
literature \cite{ContiGarroniOrtiz2015,garroni2020nonlinear}; we expect that this unified approach will 
prove helpful in further future generalizations.
We assume that dislocations are dilute, in the sense that the curve $\gamma$ on which $\mu$ is supported has some regularity, which may however degenerate in the limit. 
Our mechanical model includes frame indifference and is formulated in the deformed (spatial) configuration $\Omega$. As a measure of elastic strain at a point $x\in\Omega$, we use $\beta(x)\in\R^{3\times 3}$, which maps directions in the spatial configuration to directions in the lattice configuration; in the language of continuum mechanics this is the inverse elastic strain seen as a spatial field, $\beta=(F_e)^{-1}_s$, as discussed in Section~\ref{secdislofinite} below
{(the subscript $s$ indicates that the field depends on the spatial coordinates)}. 

The topological singularity leads to a logarithmic divergence of the energy in the regularization parameter $\eps$. After rescaling by $\log\frac1\eps$, we show $\Gamma$-convergence to an energy 
which depends on a constant rotation $Q$, 
a measure $\mu=\theta\otimes \tau \calH^1\LL\gamma$ 
representing the dislocation density, and
a curl-free field 
 $\conteta$ 
representing the elastic strain.
The limiting energy takes the form
\begin{equation}
 \int_\Omega\frac 12\C_Q \conteta\cdot \conteta \, dx + \int_{\gamma\cap\Omega} \psiC^\rel(\theta,Q\tau)d\calH^1.
\end{equation}
The second term represents  the line-tension energy, which 
can be obtained as the relaxation of 
\begin{equation}\label{eqintropsicnonrl}
\int_{\gamma\cap\Omega} \psiC(\theta,Q\tau)d\calH^1
\end{equation}
where $\psiC$ is the line-tension energy per unit length of infinite straight dislocations and can be computed from the matrix of elastic constants $\C$ by a one-dimensional variational problem (see \cite{ContiGarroniOrtiz2015} and Section \ref{preliminaries} below for details). The first term involves the elasticity matrix $\C$, rotated by $Q$.

We remark that our kinematic treatment of dislocations in finite kinematics is different from the one used in recent related mathematical literature, as for example \cite{ScardiaZeppieri2012,garroni2020nonlinear}. Indeed, 
we work in the deformed configuration and 
our limiting energy has the integrand $\psiC(\theta,Q\tau)$, whereas the cited literature works in the reference configuration and obtains a functional containing (in our notation) $\psiC(Q^T\theta',\tau')$. 
The two can be made to coincide 
if $\theta'=Q\theta$, $\tau'=Q\tau$; as $\theta$ takes value in the lattice $\calB$ this suggests that 
the expression $\psiC(Q^T\theta',\tau')$ needs
the multiplicity $\theta'$ to take values in the rotated lattice $Q\calB$. 
We remark that  our approach closely matches the one presented in 
\cite[(7.6)--(7.11)]{MuellerScardiaZeppieri2015}.

In Section \ref{sec-main} we state the main $\Gamma$-convergence result for the energies with mixed growth, we discuss the other models for alternative regularizations in the nonlinear and linear setting, 
which have the same $\Gamma$-limit,
and we give a sketch of the main ideas of our approach. The proofs of the $\Gamma$-convergence results are then collected in Section \ref{sec-mainproofs}.

In Section \ref{seccellpb} we show that for a single straight dislocation the leading-order term of the energy can be characterized by a cell problem under general assumptions for the energy density $W$. After rescaling, this term leads to \eqref{eqintropsicnonrl}. Therefore the linearization of a nonlinear elastic energy arises naturally in the limit. In particular this provides (Section~\ref{seccomp}) the lower bound for the $\Gamma$-limit of a nonlinear elastic model with energy densities of mixed growth, i.e., that behave as the minimum between $\dist^2(\cdot, \SO(3))$ and $\dist^p(\cdot, \SO(3))$ for some $p\in (1,2)$. {For this model we also obtain a compactness statement, that asserts that sequences $\beta_\eps$ with uniformly bounded energy have a converging subsequence in the relevant topology.}

Moreover in Section \ref{secupperbound} we show that the upper bound can be obtained by a pointwise limit. This proves that there is indeed a complete separation of scales between the relaxation, which happens at the line-tension level and only involves the line-tension energies $\psiC$ and $\psiC^\rel$, and the concentration, which relates the elastic energy $W$ integrated on a three-dimensional volume to the line-tension energy $\psiC$ integrated on one-dimensional sets. This was already apparent in the two-dimensional Nabarro-Peierls setting of \cite{ContiGarroniMueller2011}.

The general approach presented here will also allow, in a forthcoming paper \cite{ContiGarroniOrtizDiscrete}, to treat a discrete model, confirming that the semi-discrete models are a good description of discrete ones outside  the core region.

In two dimensions dislocations are point singularities, and a similar analysis for dilute configurations was performed for linear models in \cite{CermelliLeoni,GarroniLeoniPonsiglione} and for nonlinear models in 
\cite{ScardiaZeppieri2012,MuellerScardiaZeppieri2015}, both in the energy regimes scaling as $\log\frac1\eps$ and as $\log^2\frac1\eps$. These results  were extended in \cite{DelucaGarroniPonsiglione2012,Ginster1,Ginster2} to general configurations, exploiting the connection between models for dislocations and Ginzburg-Landau models for vortices \cite{AlicandroCicalesePonsiglione} and refining the ball construction due to Sandier and Jerrard \cite{Sandier1998,Jerrard1999,SandierSerfaty2007} to the context of elasticity in order to obtain compactness and optimal lower bounds. Similar techniques have been used 
first in \cite{Ponsiglione2007} and then 
in \cite{AlicandroDelucaGarroniPonsiglione2014} in order to study a discrete two-dimensional model for screw dislocations, see also \cite{hudson2015analysis}. 
Similar problems arise in the context of spin systems, see for instance \cite{BadalCicaleseDelucaPonsiglione}.
The two scales mentioned above give rise to mesoscopic limits where dislocations are identified with points (with energy scaling as $\log\frac1\eps$) or with densities ($\log^2\frac1\eps$). 
Results which obtain
a characterization of grain boundaries in polycrystals exhibit a different scaling.  
Lauteri and Luckhaus \cite{LauteriLuckhaus} characterized
the optimal scaling for the elastic energies of small-angle grain boundaries in nonlinear kinematics and gave a rigorous derivation of the Read-Shockley formula (see \cite{GarroniSpadaro} for the proof of the corresponding $\Gamma$-convergence result and  \cite{FanzonPalombaroPonsiglione} for the result in the context of linear elasticity with diluteness assumptions).

The extension from two dimensions to three dimensions in the context of Ginzburg-Landau models  is based on a slicing argument \cite{AlbertiBaldoOrlandi2005,BaldoJerrardOrlandiSoner2013}
which cannot be directly used for 
the anisotropic vectorial problem of 
elasticity, also because of the  degeneracy  arising from frame indifference and the relaxation of the line-tension energy. Therefore the question  whether our asymptotic analysis can also be obtained without the diluteness condition in dimension three may be  very difficult.
An intermediate step in this direction is the case in which the kinematics is restricted to a plane. Here the model reduces to a nonlocal phase-field model and was studied without diluteness assumptions in both regimes in 
\cite{GarroniMueller2006,CacaceGarroni2009,ContiGarroniMueller2011,ContiGarroniMueller2016,ContiGarroniMueller2021}. 

\section{Line dislocations in linear elasticity}\label{preliminaries}

We first collect some  results concerning the continuum study of dislocations in linear elasticity. Precisely, we recall the definition of the strain field in the presence of defects and its properties, then we give a characterization of the strain field for a straight infinite dislocation and of the associated line tension per unit length.

As usual, we assume that   $\C\in\R^{3\times 3\times 3\times 3}_\sym$  is degenerate on linearized rotations and is coercive on symmetric matrices,
in the sense that
\begin{equation}\label{eqdefC}
 \C=\C^T\,,\hskip3mm \C A\cdot A \ge c_0|A+A^T|^2 \text{ and } \C(A-A^T)=0 \text{ for all } A\in\R^{3\times 3}\,.
\end{equation}
Working in $\R^3$, dislocations are measures concentrated along one-dimensional sets. More precisely a distribution of dislocations in $\Omega\subset\R^3$ is a measure $\mu$ in the set $ \calM^1(\Omega)$ defined below.

\begin{definition}\label{defm1}
 We denote by $\calM(\Omega;V)$, with $V$ a finite-dimensional vector space, the set of finite $V$-valued Radon measures on the open set $\Omega\subseteq\R^3$. We write $\calM^1(\Omega)$ for the set of $\mu\in \calM(\Omega;\R^{3\times 3})$  which obey
$\Div\mu=0$ distributionally and have the form
\begin{equation}\label{eq-defmu}
 \mu=\theta\otimes \tau \calH^1\LL\gamma
\end{equation}
for some 1-rectifiable set $\gamma\subseteq\Omega$, $\theta:\gamma\to\R^3$, $\tau:\gamma\to S^2$. 
\end{definition}
The condition $\Div \mu=0$ automatically implies that $\tau$ is tangent to $\gamma$ $\mu$-almost everywhere. We refer to the multiplicity $\theta$ in \eqref{eq-defmu} as the Burgers vector of the dislocation line $\gamma$.

\subsection{Construction of the strain field}

Let  $\mu\in\calM(\R^3;\R^{3\times 3})$ with $\Div\mu=0$. Within linear elasticity, the corresponding equilibrium strain is a distributional solution 
$\beta\in L^1_\loc(\R^3;\R^{3\times 3})$ to
\begin{equation}\label{eqsystembetamu2}
 \begin{cases}
  \curl\contbeta=\mu,\\
  \Div \C\contbeta=0,
 \end{cases}
\end{equation}
with $\C$ as in \eqref{eqdefC}.
Existence and uniqueness of solutions to \eqref{eqsystembetamu2} were proven in
\cite[Theorem~4.1]{ContiGarroniOrtiz2015}. We recall here this result in $\R^3$.
\begin{theorem}\label{theoremsolr3}
Let $\C$ be as in (\ref{eqdefC}).
For any  bounded measure  $\mu\in\calM(\R^3;\R^{3\times 3})$
with $\Div\mu=0$ the following holds.
\begin{enumerate}
 \item \label{lemmasolr3exist}
 There is a unique $\contbeta\in L^{3/2}(\R^3;\R^{3\times 3})$ such that
  \begin{equation}\label{eqdivbmur3}
 \Div \C\contbeta=0 \quad\text{ and }\quad \Curl \contbeta=\mu
  \end{equation}
  distributionally.  The solution $\contbeta$ satisfies
  \begin{equation*}
 \|\contbeta\|_{L^{3/2}(\R^3)} \le c |\mu|(\R^3)\,.
  \end{equation*}
  \item\label{lemmasolr3intrep}
There is a function $N\in C^\infty(\R^3\setminus\{0\};\R^{3\times 3\times 3\times 3})$ which depends only on $\C$, is positively $-2$ homogeneous, in the sense that
\begin{equation}\label{eqNmenoduehom}
 N(\lambda x)=\lambda^{-2} N(x) \text{ for all } x\in\R^3,\lambda\in(0,\infty), 
\end{equation}
and such that the unique solution $\contbeta$  satisfies
  \begin{equation}\label{eqbetaNmu}
 \contbeta_{ij}(x)=\sum_{k,l=1}^3\int_{\R^3} N_{ijkl}(x-y) d\mu_{kl}(y)
\end{equation}
(for $x\not\in\supp\mu$).
\item \label{lemmasolr3lp}
 If  additionally $\mu\in \calM^1(\R^3)$ and $\mu=\sum_i b_i \otimes t_i\calH^1\LL\gamma_i$ for countably many segments $\gamma_i$, then for $x\not\in\supp\mu$
  \begin{equation}\label{eqDcontbetadist1}
 |\contbeta(x)|\le c \sum_i \frac{|b_i|}{\dist(x, \gamma_i)}.
  \end{equation}
If the number of segments is finite then
$\contbeta\in L^p(\R^3;\R^{3\times 3})$ for all $p\in [3/2,2)$.
\end{enumerate}
The constant $c$ depends only on $\C$.
\end{theorem}

We remark that in point \ref{lemmasolr3lp}  of
\cite[Theorem~4.1]{ContiGarroniOrtiz2015} 
the vectors $b_i$ were required to be in a lattice. This assumption was never used in the proof.
Assertion \ref{lemmasolr3intrep} is proven in (4.7) and the following equations of 
\cite{ContiGarroniOrtiz2015}.

\subsection{Characterization of the strain for straight dislocations}

From the representation formula \eqref{eqbetaNmu} one also obtains a representation for the strain field of an infinite straight dislocation 
of Burgers vector $b$
along the line $\R t$, for some $t\in S^2$ and $b\in\R^3$, which is then a distributional solution of 
 \begin{equation}\label{eq-infinitestrain}
 \begin{cases}
  \curl\contbeta=b\otimes t \calH^1 \LL (\R t), \\
  \Div\C\contbeta=0.
 \end{cases}
\end{equation}
In the entire paper we denote by $B_r(x)$ and $B'_r(x)$ the balls of radius $r>0$ centered at $x$ in $\R^3$ and $\R^2$ respectively. For $x=0$ we simply write $B_r$ and $B'_r$.
\begin{lemma}\label{lemmabetabtkernel}
 Let  $N\in C^{\infty}(\R^3\setminus\{0\};\R^{3\times 3\times 3\times 3})$ be as in Theorem \ref{theoremsolr3}\ref{lemmasolr3intrep}.
 For any $b\in\R^3$ and $t\in S^2$  the function 
 \begin{equation}\label{eqbetamuline}
  (\hat\beta(x))_{ij}:=
  \sum_{k,l=1}^3\int_{\R t} N_{ijkl}(x-y)b_k t_l d\calH^1(y)
 \end{equation}
 (for $x\not\in \R t$) 
 {is in $L^1_\loc(\R^3;\R^{3\times 3})$ and}
 satisfies \eqref{eq-infinitestrain} in the sense of distributions.
\end{lemma}
\begin{proof}
 For $R>0$ we define $\mu_R$ by
 \begin{equation}
  \mu_R:= b\otimes \tau\calH^1 \LL \gamma_R
 \end{equation}
 where $\gamma_R$ is the union of the segment $[-Rt, Rt]$ with a half-circle which has $[-R t,R t]$ as diameter and $\tau:\gamma_R\to S^2$ is a tangent vector to $\gamma_R$, oriented so that it coincides with $t$ in $[-Rt, Rt]$. One easily checks that $\mu_R\in \calM^1(\R^3)$  and $\mu_R\weakstarto\mu$ as $R\to\infty$.
 
 We define $\contbeta_R:=N \ast \mu_R$, where $\ast$ denotes convolution as in 
 \eqref{eqbetaNmu}. 
 By Theorem~\ref{theoremsolr3}, it satisfies
 \begin{equation}\label{eqELbetaR}
  \begin{cases}
   \curl\contbeta_R=\mu_R,\\
   \Div\C\contbeta_R=0.
  \end{cases}
 \end{equation}
 Let $\hat\contbeta:=N\ast (b\otimes t \calH^1\LL (\R t))$ be defined as in \eqref{eqbetamuline}. 
 We remark that the integral exists for all $x\not\in\R t$ since the integrand is continuous and decays as $|y|^{-2}$ at infinity.
 
 We show that $\contbeta_R\to\hat\contbeta$ in $L^1_\loc(\R^3;\R^{3\times 3})$. 
 Fix $r>0$, and assume that $R\ge 2r$. Then for almost every $x\in B_r\setminus \R t$ we have
 \begin{equation*}
  \hat\contbeta(x)-\contbeta_R(x)=
  \int_{(\R t)\setminus [-R t, Rt]} N(x-y)b\otimes t d\calH^1(y)
  -\int_{\gamma_R\setminus [-R t, Rt]} N(x-y)b\otimes \tau d\calH^1(y)
 \end{equation*}
 which implies, since $|x|\le \frac12 R$ and $|y|\ge R$,
 \begin{equation*}
  |\hat\contbeta-\contbeta_R|(x)\le
  2 |b| \|N\|_{L^\infty(S^2)} \int_R^\infty \frac{4}{s^2} ds 
  + \pi R |b| \frac{4\|N\|_{L^\infty(S^2)} }{R^2}
  \le \frac{c |b| \|N\|_{L^\infty(S^2)} }{R}.
 \end{equation*}
 Therefore $\hat\contbeta-\contbeta_R$ converges to zero  uniformly on compact sets, as $R\to\infty$, and hence in $L^1_\loc(\R^3;\R^{3\times 3})$.
 By \eqref{eqELbetaR}, we obtain \eqref{eq-infinitestrain}
 distributionally in $\R^3$.
\end{proof}

The strain field $\hat\beta$  in \eqref{eqbetamuline} can be characterized with a variational argument and it is given by a two-dimensional profile $\contbeta_{b,t}$.
In order to do this, as in \cite{ContiGarroniOrtiz2015},
for any $t\in S^2$
we fix a rotation $Q_t\in \SO(3)$ such that $Q_te_3=t$, define the map
\begin{equation}\label{eqdefPhit}
 \Phi_t(r,\theta,z):=Q_t(r \cos\theta, r\sin\theta, z) ,
\end{equation}
and the orthonormal basis
\begin{equation}
 e_r(\theta):=(\cos\theta,\sin\theta,0),\hskip3mm
 e_\theta(\theta):=(-\sin\theta,\cos\theta,0),\hskip3mm
 e_z:=(0,0,1).
\end{equation}

We recall that in \cite[Lemma~5.1]{ContiGarroniOrtiz2015}, for any 
$b\in\R^3$ and $t\in S^2$ we defined $\contbeta_{b,t}\in L^1_\loc(\R^3;\R^3)$ as the unique function of the form
\begin{equation}\label{eqvarprobbetzabt1}
 \begin{split}
 \tilde\beta(\Phi_t(r,\theta,z))= \frac1r (f(\theta)\otimes Q_t e_\theta + g\otimes Q_t e_r)
 \end{split}
\end{equation}
for some $f\in L^2((0,2\pi);\R^3)$ with $\int_0^{2\pi} f(\theta) d\theta=b$ and $g\in\R^3$ which minimizes the variational problem
\begin{equation}\label{eqvarprobbetzabt}
 \begin{split}
  \psiC(b,t):=&\inf_{ \tilde\beta}\int_0^{2\pi} \frac12\C{ \tilde\beta}(\Phi_t(1,\theta,0))\cdot { \tilde\beta}(\Phi_t(1,\theta,0)) d\theta
  \\
  &= \int_0^{2\pi} \frac12\C\contbeta_{b,t}(\Phi_t(1,\theta,0))\cdot \contbeta_{b,t}(\Phi_t(1,\theta,0)) d\theta\,.
 \end{split}
\end{equation}

\begin{lemma}\label{lemmabetabtkernel1}
 Let  $N\in C^\infty(\R^3\setminus\{0\};\R^{3\times 3\times 3\times 3})$ be as in Theorem \ref{theoremsolr3}\ref{lemmasolr3intrep}.
 For any $b\in\R^3$ and $t\in S^2$ the solution $\contbeta_{b,t}$ of the one-dimensional problem in (\ref{eqvarprobbetzabt1}-\ref{eqvarprobbetzabt}) is
 \begin{equation}\label{eqbetamuline2}
  (\contbeta_{b,t}(x))_{ij}=
  \sum_{k,l=1}^3\int_{\R t} N_{ijkl}(x-y)b_k t_l d\calH^1(y)
 \end{equation}
 (for $x\not\in \R t$). On the same set it obeys 
 \begin{equation}\label{eqdecaybetabt}
  |\beta_{b,t}|(x)\le \frac{c|b|}{\dist(x,\R t)}\,,
 \end{equation}
 where $c$ depends only on $\C$. 
 \end{lemma}
 Equation \eqref{eqbetamuline2} in particular shows that $\beta_{b,t}$ does not depend on the choice of the matrix $Q_t$.
\begin{proof}
We need to prove that $\hat\contbeta$ as defined in \eqref{eqbetamuline} coincides with $\contbeta_{b,t}$. This involves two steps: we first show that $\hat\contbeta$ has the same structure as $\contbeta_{b,t}$,  and then that it obeys the Euler-Lagrange equations of the variational problem that defines $\contbeta_{b,t}$,  which has a unique solution by 
 \cite[Lemma~5.1]{ContiGarroniOrtiz2015}.
 
 We claim that 
 the definition of $\hat \contbeta$ implies that there is $G\in C ^1_\text{per}([0,2\pi];\R^{3\times 3})$ such that for any $r>0$, $\theta$ and $z$, with $\Phi_t$ as in \eqref{eqdefPhit},
 \begin{equation}\label{eqhatbetaphiG}
  \hat\contbeta(\Phi_t(r,\theta,z))= \frac1r G(\theta).
 \end{equation}
 To see this, we observe that by definition $\Phi_t(r,\theta,z)=rQ_te_r(\theta)+zt$, so that the definition of $\hat \contbeta$ yields
 \begin{equation*}
  \hat\contbeta(\Phi_t(r,\theta,z))= \int_{\R} N(rQ_te_r(\theta) + (z-s)t) (b\otimes t)ds
  = \frac{1}{r} \int_{\R} N(Q_te_r(\theta) -s't) (b\otimes t)ds',
 \end{equation*}
 where we used the change of variables $s':=(s-z)/r$ and  \eqref{eqNmenoduehom}. As the integral depends only on $\theta$, and the dependence is smooth, \eqref{eqhatbetaphiG} is proven.

 Next we show that the condition that $\hat\contbeta$ is locally curl free away from $\R t$ implies that $G$ has a special structure.
 We compute
 \begin{equation}\label{eqderPhit}
  \partial_r\Phi_t=Q_te_r,\hskip3mm
  \partial_\theta\Phi_t=rQ_te_\theta,\hskip3mm
  \partial_z\Phi_t=Q_te_3.
 \end{equation}
 Therefore
 \begin{equation*}
  \partial_r (\hat\contbeta\circ\Phi_t)= D\hat\contbeta Q_t e_r\,,\hskip3mm
  \partial_\theta (\hat\contbeta\circ\Phi_t)= r D\hat\contbeta Q_t e_\theta\,,\hskip3mm
  \partial_z (\hat\contbeta\circ\Phi_t)= D\hat\contbeta Q_t e_3.
 \end{equation*}
By \eqref{eqhatbetaphiG}, $\hat\contbeta\circ\Phi_t=G(\theta)/r$, and 
 \begin{equation*}
  D\hat\contbeta Q_t e_r=-\frac1{r^2} G\,,\hskip3mm
  D\hat\contbeta Q_t e_\theta=\frac1{r^2} G'\,,\hskip3mm
  D\hat\contbeta Q_t e_3=0.
 \end{equation*}
 Away from $\R t$ we have  $\curl\hat\contbeta=0$ and therefore $D\hat\contbeta=(D\hat \contbeta)^T$,
 in the sense that 
 $\partial_j\hat\beta_{ki}=\partial_i\hat\beta_{kj}$.
 This implies
 \begin{equation*}
  0=(D\hat\contbeta Q_te_r) Q_te_3-(D\hat\contbeta Q_te_3) Q_te_r
  =-\frac1{r^2}GQ_te_3
 \end{equation*}
 which gives $G(\theta)Q_te_3=0$ for all $\theta$.
 Further,
 \begin{equation}\label{eqGrteta}
  0=(D\hat\contbeta Q_te_\theta) Q_te_r-(D\hat\contbeta Q_te_r) Q_te_\theta
  =\frac1{r^2} G'(\theta) Q_t e_r + \frac{1}{r^2} G(\theta) Q_te_\theta.
 \end{equation}
 We define $f,g,h:(0,2\pi)\to\R^3$ by 
 \begin{equation*}
  G(\theta)= f(\theta)\otimes Q_te_\theta + g(\theta)\otimes Q_te_r + h(\theta)\otimes Q_t e_3.
 \end{equation*}
 The condition $GQ_te_3=0$ implies $h=0$. We compute, using $\partial_\theta e_\theta=-e_r$ and $\partial_\theta e_r=e_\theta$,
 \begin{equation*}
  G'(\theta)=f'(\theta)\otimes Q_t e_\theta - f(\theta)\otimes Q_t e_r
  +g'(\theta)\otimes Q_t e_r + g(\theta)\otimes Q_t e_\theta.
 \end{equation*}
Inserting in \eqref{eqGrteta} leads to
 \begin{equation*}
  -f(\theta)+g'(\theta)+f(\theta)=0 
 \end{equation*}
 which implies that $g$ is constant.
 We conclude that
 \begin{equation}
  \hat\contbeta(\Phi_t(r,\theta,z))= \frac1r (f(\theta)\otimes Q_t e_\theta + g\otimes Q_t e_r).
 \end{equation}
 By the condition $\curl\hat\contbeta=b\otimes t\calH^1\LL(\R t)$ we obtain $\int_0^{2\pi} f(\theta) d\theta=b$.

 It remains to show that $(f,g)$ are the unique minimizers of the  variational problem in \eqref{eqvarprobbetzabt}.
 Let $\tilde f\in L^\infty((0,2\pi);\R^3)$ with $\int_0^{2\pi} \tilde f(\theta)d\theta=0$, and $\tilde g\in \R^3$. We define
 \begin{equation}\label{eqdefA}
  A := \int_0^{2\pi} \C G(\theta) \cdot (\tilde f(\theta)\otimes Q_te_\theta) d\theta
 \end{equation}
 and
 \begin{equation}
  B := \int_0^{2\pi} \C G(\theta) \cdot (\tilde g\otimes Q_te_r) d\theta.
 \end{equation}
 If we can show that $A=B=0$, by convexity of the functional in \eqref{eqvarprobbetzabt} we are done.
 
 We start from $A$. Fix $h,R\ge 2$,
let $\psi_3\in C^\infty_c(\R;[0,1])$ 
  such that $\int_\R\psi_3 dz=1$, and
 $\psi_R\in C^\infty_c((0,\infty);[0,1])$ such that $\int_0^\infty\psi_R dr=1$.
 We define $\tilde u\in W^{1,\infty}_0(\R^3;\R^3)$ by
 \begin{equation}
  \tilde u(\Phi_t(r,\theta,z)):= \tilde F(\theta) \psi_R(r) \psi_3(z)
  \hskip5mm
  \text{where}\hskip5mm
  \tilde F(\theta):=\int_0^\theta \tilde f(s)ds.
 \end{equation}
 We differentiate and obtain
\begin{equation}
D(\tilde u(\Phi_t(r,\theta,z)))
  =(D\tilde u) \circ\Phi_t \, D\Phi_t.
\end{equation}
The three derivatives on the left can be easily computed from the definition. 
Recalling \eqref{eqderPhit},
 \begin{equation}
  (D\tilde u)\circ\Phi_t = \frac1r\tilde f \psi_R \psi_3 \otimes Q_te_\theta +\tilde F \psi_R' \psi_3 \otimes Q_te_r + \tilde F\psi_R\psi_3' \otimes Q_te_3.
 \end{equation}
 Since $\Div \C\hat\contbeta=0$ distributionally and $\hat\contbeta\in L^1_\loc(\R^3;\R^{3\times 3})$, we have 
 \begin{equation*}
  \int_{\R^3} \C\hat\contbeta D\tilde  u dx=0.
 \end{equation*}
 We write this equation in cylindrical coordinates, writing the three terms separately
 and using \eqref{eqhatbetaphiG}.
 Since $\int_\R \psi_3 dz=1$ and $\int_\R\psi_3'dz=0$, this reduces to
 \begin{equation*}
  \begin{split}
   0=&\int_\R \int_0^\infty  \int_0^{2\pi} \C\hat\contbeta(\Phi_t(r,\theta,z)) \cdot (D\tilde u) (\Phi_t(r,\theta,z))
   r d\theta drdz\\
   =& \int_0^\infty  \int_0^{2\pi} \C G(\theta)\cdot
   (\frac1r\tilde f \psi_R\otimes Q_te_\theta ) d\theta dr
   + \int_0^\infty  \int_0^{2\pi} \C G(\theta)\cdot
  (\tilde F \psi_R'  \otimes Q_te_r )d\theta dr.
  \end{split}
 \end{equation*}
 The second term vanishes since $\int_0^\infty \psi_R'(r)dr=0$. We are left with the first term, which 
 recalling \eqref{eqdefA} takes the form
 \begin{equation*}
  0=\int_0^{\infty}  \frac1r \psi_R\int_0^{2\pi} \C G(\theta)\cdot
  ({\tilde f} \otimes Q_te_\theta)  d\theta dr
  = A \int_0^{\infty}  \frac1r \psi_Rdr.
 \end{equation*}
 Since $\int_0^\infty\psi_R(r)dr=1$, the integral is positive, and we conclude that $A=0$. 
 
 We now turn to $B$. 
 We fix  $\psi_3$ as above, choose $\Psi_R\in C^\infty(\R;[0,1])$,  such that $\Psi_R(0)=1$, and define
 \begin{equation*}
  U(\Phi_t(r,\theta,z)):= \tilde g \Psi_{R}(r)\psi_3(z).
 \end{equation*}
 As above, we compute
 \begin{equation*}
 ( DU)\circ\Phi_t=\tilde  g \Psi_R'\psi_3 \otimes Q_te_r + \tilde g \Psi_R \psi_3' \otimes Q_t e_3
 \end{equation*}
 and, computing again the integral and recalling that $\int_\R \psi_3'dz=0$ and $\int_0^\infty \Psi_R' dr=-1$, we obtain
 \begin{equation*}
  \begin{split}
   0
   =&\int_\R \int_0^\infty  \int_0^{2\pi} \C G(\theta)\cdot
   (\tilde g \Psi_R'(r) \psi_3(z) \otimes Q_te_r) d\theta drdz
   \\
   &+\int_\R \int_0^\infty  \int_0^{2\pi} \C G(\theta)\cdot
   ( \tilde g\Psi_R(r)\psi_3'(z) \otimes Q_te_3) d\theta drdz
   =- B  .
  \end{split}
 \end{equation*}
 Therefore $B=0$ and the proof of the first assertion is concluded.
 
 It remains to prove \eqref{eqdecaybetabt}. This follows immediately from the fact that $N$ is $-2$-homogeneous and continuous,
 \begin{equation}
 \int_{\R t} |N|(x-y)d\calH^1(y)
 \le \|N\|_{L^\infty(S^2)} \int_\R \frac{1}{|x-st|^2} ds =\frac{\pi\|N\|_{L^\infty(S^2)} }{\dist(x,\R t)}.
 \end{equation}
\end{proof}
 
\subsection{Line tension energy of straight dislocations}

The variational problem that characterizes $\beta_{b,t}$ provides the energy per unit length of the straight infinite dislocation,
$\psiC(b,t)$, which we will call the unrelaxed {line-tension energy density}. 
Indeed using $\beta_{b,t}$ we can compute the elastic energy induced by a straight dislocation along $\R t$ in a cylinder with axis $\R t$.
Precisely given $r,R,h\in(0,\infty)$ with $r<R$ we define the hollow cylinders
\begin{equation}\label{def:hol_cyl}
 T_h^{R,r}:=(B'_R\setminus B'_r)\times(0,h),
\end{equation}
and the full cylinders
\begin{equation}\label{def:full_cyl}
 T_h^R:=T_h^{R,0}=B'_R\times(0,h).
\end{equation}
We recall that $B'_\rho$ denotes the 2-dimensional disk of radius $\rho>0$, $B'_\rho:=\{(x_1,x_2): x_1^2+x_2^2<\rho^2\}$, and that before \eqref{eqdefPhit} we introduced $Q_t$ as a fixed rotation with $Q_te_3=t$.
Therefore using the radial structure of $\beta_{b,t}$, from \eqref{eqvarprobbetzabt}  we have
\begin{equation}\label{eqpsi0cylinder}
\int_{Q_tT_h^{R,r}}
\frac12 \C\contbeta_{b,t} \cdot\contbeta_{b,t}dy
=h \ln \frac{R}{r} \psiC (b, t)
\end{equation}
(see also \cite[Lemma~5.1(iv)]{ContiGarroniOrtiz2015}).

With this energy density one can associate 
to any $\mu\in\calM^1(\Omega)$
the line-tension energy
\begin{equation}\label{eq-line-tension}
 \int_\gamma \psiC(\theta,\tau)d\calH^1.
\end{equation}
As shown in \cite{ContiGarroniMassaccesi2015}, this functional may not be lower semicontinuous and needs to be relaxed (see Section~\ref{secgammaconv}    formula \eqref{psi-rel}  below).
 A variational characterization of $\psi_\C$, in which 
the elastic energy is minimized over cylinders,
has been obtained in \cite[Lemma~5.5]{ContiGarroniOrtiz2015}. To explain it, for
$\C$ as in \eqref{eqdefC},
$b\in\R^3$, $t\in S^2$, $h\ge R>r>0$ we define
 \begin{equation}\label{eqdefabthRr}
 \begin{split}
\infcyl(\C,b,t,h,R,r):=\inf \Big\{  \frac{1}{h\ln \frac Rr}\int_{Q_t T^{R,r}_h} \frac12 \C \beta\cdot\beta& dx:  \beta=\beta_{b,t}+Du,\\& u\in W^{1,1}(Q_tT^{R,r}_h;\R^3)\Big\}.
 \end{split}
 \end{equation}
Then $\psiC$ arises as limit of $\infcyl$ for $r\to0$. We improve over \cite[Lemma~5.5]{ContiGarroniOrtiz2015}  removing one error term, and prove the following.
\begin{lemma}\label{lemmacellproblemlb}
Assume that $\C$ obeys \eqref{eqdefC}.
\begin{enumerate}
 \item \label{lemmacellproblemlbA}
 Let $b\in\R^3$, $t\in S^2$, $h\ge R>0$.
Then \begin{equation}\label{eqlimroApsic}
\lim_{r\to0}\infcyl(\C,b,t,h,R,r)=\psiC(b,t).   
  \end{equation}
 \item \label{lemmacellproblemlbbcont}
There is $c>0$, depending only on $\C$, such that
\begin{equation}\label{eqpsiccontunif}
\frac1c|b|^2\le \psiC(b,t)\le c|b|^2
\end{equation}
and
\begin{equation}\label{eqpsiccontunifb}
 \psiC(b,t)\le \left(1+\frac{|b-b'|}{|b|+|b'|}\right)\psiC(b',t)+c|b-b'| (|b|+|b'|)
\end{equation}
for all $b$, $b'\in\R^3$, $t\in S^2$.
 \item \label{lemmacellproblemlbtcont}There is $c>0$, depending only on $\C$, such that
for all $b\in \R^3$, $t,t'\in S^2$, we have
\begin{equation}\label{eqpsictlip}
 \psiC(b,t)\le (1+c|t-t'|)\psiC(b,t).
\end{equation}
\end{enumerate}
\end{lemma}
\begin{proof}\ref{lemmacellproblemlbA}:
Since any function $u\in W^{1,1}(Q_tT^{R,r}_h;\R^3)$ can be extended to a function in 
$W^{1,1}(Q_tT^{R}_h;\R^3)$, we have  
 \begin{equation}\label{eqdefabthRrb}
 \begin{split}
  \infcyl(\C,b,t,h,R,r)=\inf \Big\{  \frac{1}{h\ln \frac Rr}\int_{Q_t T^{R,r}_h} &\frac12 \C \beta\cdot\beta dx: \beta\in L^1(Q_tT_h^R;\R^{3\times 3}),\\ &\curl\beta=b\otimes t \calH^1\LL (t\R\cap Q_t T^{R}_h) \Big\}.
 \end{split}
 \end{equation}
By \cite[Lemma~5.5]{ContiGarroniOrtiz2015} we get
\begin{equation}\label{eq:liminf_linear_problem}
\liminf_{r\to0} \infcyl(\C,b,t,h,R,r)\ge \psiC(b,t)-c\frac{|b|^2R}{h}.
\end{equation}
Moreover for $R>R'>r$ it holds
\begin{equation*}
 \infcyl(\C,b,t,h,R,r)\ge \frac{\log\frac {R'}{r}}{\log\frac Rr} \infcyl(\C,b,t,h, R',r).
\end{equation*}
Taking the limit as $r\to0$,
since $\lim_{r\to0}
\frac{\log\frac {R'}{r}}{\log\frac Rr}=1$ from \eqref{eq:liminf_linear_problem} we infer
\begin{equation*}
\liminf_{r\to0}\infcyl(\C,b,t,h,R,r)\ge 
\liminf_{r\to0}\infcyl(\C,b,t,h,R',r)\ge 
\psiC(b,t)-c\frac{|b|^2R'}{h}
\end{equation*}
for all $R'>0$, and therefore for $R'=0$.
At the same time, by \eqref{eqpsi0cylinder} we obtain
$\infcyl(b,t,h,R,r)\le\psiC(b,t)$ for all $r\in(0,R)$, and the proof of \eqref{eqlimroApsic} is concluded.

\ref{lemmacellproblemlbbcont}:
The bounds in \eqref{eqpsiccontunif} follow {from \cite[Lemma~5.1(iii)]{ContiGarroniOrtiz2015}} (the upper bound can be also easily proven 
inserting the bound in 
\eqref{eqdecaybetabt} into 
\eqref{eqpsi0cylinder} and integrating).
The remaing
bound (Eq.~\eqref{eqpsiccontunifb}) is proven using 
Lemma~\ref{lemmabetabtkernel1} and 
\eqref{eqpsi0cylinder}.
We  first observe that by \eqref{eqbetamuline2} we have
$\beta_{b,t}=\beta_{b',t}+\beta_{b-b',t}$.
Therefore for any $\delta>0$, fixing some $r<R\le h$ 
from \eqref{eqpsi0cylinder} we obtain
\begin{equation}
 \begin{split}
  \psiC(b,t)\le (1+\delta)\psiC(b',t)+
  \left(1+\frac1\delta\right) \psiC(b-b',t).
 \end{split}
\end{equation}
In turn, by the upper bound in 
\eqref{eqpsiccontunif}
we have
\begin{equation}
 \begin{split}
  \psiC(b,t)\le (1+\delta)\psiC(b',t)+
  c\left(1+\frac1\delta\right) |b-b'|^2.
 \end{split}
\end{equation}
Setting $\delta:=|b-b'|/(|b|+|b'|)\in(0,1]$ concludes the proof of \eqref{eqpsiccontunifb}.

\ref{lemmacellproblemlbtcont}: follows from 
\cite[Lemma~5.7]{ContiGarroniOrtiz2015}.
\end{proof}

\section{Model and main results}\label{sec-main}

\subsection{Dislocations in finite kinematics}
\label{secdislofinite} 

We work within the general framework of continuum mechanics, using finite kinematics. In the presence of dislocations, there is no smooth bijection between the reference configuration and the deformed configuration. Therefore any choice of a reference configuration includes a high degree of arbitrariness.
We use spatial variables, with an energy obtained integrating over the deformed configuration $\Omega$, and focus on the strain $\beta:\Omega\to\R^{3\times 3}$, with $\beta(x)$ seen as a linear map from the tangent space to the deformed configuration to the tangent space to the reference configuration, or the so-called (fictitious) ``intermediate configuration'' or ``lattice configuration''. Precisely, for $x\in\Omega$ (a point in the physical space occupied by the material) and $t\in S^2$ a direction in the same space, $\beta(x)t$ is the corresponding vector in the intermediate configuration, as illustrated in Figure~\ref{figdefconf}. It is then natural to expect that the line integral of $\beta$ over a closed (spatial) curve results in a lattice vector, the Burgers vector.
Therefore {we consider} the distributional rotation $\curl\beta$
as a measure of the density of dislocations.
Indeed, this is how normally the Burgers circuit argument is formulated, see for example \cite[Fig.~1.19]{HullBacon},
 \cite[(7)]{AcharyaBassani2000}, or \cite{LuckhausMugnai2010} for a mathematical treatment.

We can relate $\curl\beta$ to the dislocation density tensor as used in continuum mechanics. Within the customary  multiplicative decomposition of the strain into an elastic and a plastic part $F=F_eF_p$, the density of dislocations (in the sense of the Nye tensor) is given by 
\begin{equation}
\alpha:= \frac{1}{\det F_p} (\CCurl F_p) F_p^T
=(\curl F_e^{-1})\cof F_e,
\end{equation}
where $\curl$ denotes the rotation in the spatial configuration, whereas $\CCurl$ denotes the rotation in the material configuration.
We refer for example to \cite{CermelliGurtin2001,svendsen02,MielkeMueller2006,ReinaConti2014} for these formulas and their equivalence (to compare formulas it is useful to keep in mind that
in \cite{CermelliGurtin2001} curl is defined columnwise instead of rowwise as it is here, and that $\cof F=F^{-T}\det F$ for any invertible matrix $F$).  In this language, the variable $\beta$ we use corresponds to $F_e^{-1}$, seen as a spatial field. 
Whereas $\alpha$ measures the density of (geometrically necessary) dislocations seen as the total Burgers vector crossing a certain area in the reference configuration, $\curl\beta$ measures the total Burgers vector crossing a certain area in the deformed configuration. We prefer this version for consistency with the fact that we work in the spatial representation, and because $\curl\beta$ (at variance with $\curl\beta \cof \beta^{-1}$) can be easily understood distributionally.

In order to correctly formulate the variational model it is important to understand how frame indifference acts in this setting. In the usual material formulation of (dislocation-free) nonlinear elasticity, one considers a (bijective) deformation field $u:\omega\to\Omega$,
with $\omega\subseteq\R^3$ the reference configuration and $\Omega\subseteq\R^3$ the deformed (spatial) configuration. 
In this simple setting, $\beta:\Omega\to\R^{3\times 3}$ is defined by $\beta(u(x))=(Du(x))^{-1}$, and it is the gradient of the map $v:\Omega\to\omega$ which is the inverse to $u$.
Superimposing a rotation amounts to replacing $u$ by $u^Q(x):=Qu(x)$, $u^Q:\omega\to Q\Omega$,
so that the deformation gradient $Du(x)$ gets replaced by
$Du^Q(x)=QDu(x)$. 
The inverse $v^Q$ of $u^Q$ is given by
$v^Q(u^Q(x))=x$, which is the same as
$v^Q(Qy)=v(y)$. Differentiating this expression one obtains
$Dv^Q(Qy)Q=Dv(y)$, which is the same as
\begin{equation}\label{eqbetaqrotkin}
\beta^Q(Qy)Q=\beta(y). 
\end{equation}
This expression shows the action of rotations on the field $\beta$. The above computation based on the chain rule shows that the multiplication of $\beta$ by a rotation on the right is naturally coupled to a change of variables by the same rotation.
An hyperelastic material can then be modeled by the energy
\begin{equation}
 \int_\Omega W(\beta) dy,
\end{equation}
and material frame indifference leads to the requirement
\begin{equation}\label{eqWFQFFapp}
 W(FQ)=W(F) \text{ for all } F\in \R^{3\times 3}, Q\in\SO(3).
\end{equation}
With a change of variables, one can relate this expression to the usual integration over the reference configuration:
\begin{equation}\begin{split}
 \int_\Omega W(\beta(y)) dy
& =\int_{u(\omega)} W((Du)^{-1}(u^{-1}(y))) dy\\
& =\int_\omega W((Du)^{-1}(x)) \det Du(x) dx
 =\int_{\omega} \hat W(Du(x)) dx,
\end{split}\end{equation}
where $\hat W(F)=W(F^{-1})\det F$. One immediately sees that 
\eqref{eqWFQFFapp} is equivalent to the usual right-invariance $\hat W(QF)=\hat W(F)$, and the requirement that $\{ W=0\} \cap \{\det>0\}=\SO(3)$ is equivalent
to  $\{\hat W=0\} \cap \{\det>0\}=\SO(3)$.

If the bijective map $u$ does not exist globally, the same procedure can be performed locally, away from the dislocation cores (or using the intermediate configuration).
In the spatial formulation, we consider $\beta:\Omega\to\R^{3\times 3}$. It maps directions $t$ in the spatial configuration (in $\Omega$) onto directions $\beta t$ in the (fixed) lattice configuration. 
If we insert a rotation, mapping $\Omega$ to $Q\Omega$, the direction $t$ becomes $Qt$, and a point $y\in\Omega$ becomes  $Qy\in Q\Omega$. However, the vector in the lattice configuration is not modified, therefore necessarily
$\beta^Q(Qy)Qt=\beta(y)t$ for all $t\in\R^3$, which is the same as \eqref{eqbetaqrotkin}.

We next address how the distribution of dislocations, understood as $\curl\beta$, transforms under rotations, 
and how the line energy $\psiC$ transforms. The key fact, in terms of the field introduced in \eqref{eqbetaqrotkin}, is
\begin{equation}
(\curl \beta^Q)(Qy)Q=\curl \beta(y).
\end{equation}
We remark that this equation, written in terms of Nye's tensor $\alpha=\curl \beta\cof \beta^{-1}$, 
leads to $\alpha^Q(Qy)=\alpha(y)$, i.e., the Nye tensor is invariant under changes of coordinates.

\begin{figure}[t!]
 \begin{center}
  \includegraphics[width=12cm]{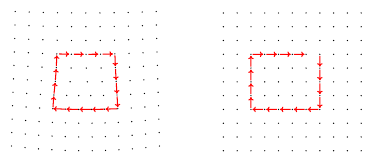}
 \end{center}
\caption{A Burgers circuit in finite elasticity. Left: spatial (deformed) configuration. A closed path goes around the dislocation. The path is composed of finitely many segments  $t_i$ that join an atom to a neighbour. Right: representation in the reference (more precisely, intermediate) configuration. Here the atoms are exactly on the reference Bravais lattice, each of the segments $t_i$ is mapped to a corresponding lattice vector which (in the continuum limit) is $\beta t_i$. The path does not close, the sum of the segments is the Burgers vector, which is naturally an element of the  (undeformed) crystal lattice. Therefore the contour integral of $\beta$ equals the Burgers vector of the dislocation.}
\label{figdefconf}
\end{figure}

\begin{lemma}\label{lemmarotatebetacurl}
\begin{enumerate}
 \item\label{lemmarotatebetacurlbeta}
Let $Q\in \SO(3)$, $F\in\R^{3\times 3}$,
$\lambda\in(0,\infty)$, $v\in\R^3$,
$\Omega\subseteq\R^3$ open. Given $\beta\in L^1_\loc(\Omega;\R^{3\times 3})$  we define $\hat \beta\in L^1_\loc(Q^T\Omega;\R^{3\times 3})$  by
\begin{equation}\label{eqdefhatbeta}
 \hat \beta(x):=F\beta(\lambda Qx+v) Q.
\end{equation}
Then
\begin{equation}\label{curlhatbeta}
(\curl \hat \beta)(x):=\lambda F(\curl \beta)(\lambda Qx+v) Q.
\end{equation}
\item\label{lemmarotatebetacurlbetameas}
 If $\curl\beta$ is a (matrix-valued) measure,  then 
 so is $\curl\hat \beta$, and for any Borel set $A$ one has 
 $(\curl\hat \beta)(A)=\frac1{\lambda^2} F((\curl\beta)(\lambda QA+v))Q$. In particular, if
 $\curl\beta=b\otimes t\calH^1\LL\gamma\in\calM^1(\R^3)$ 
 then $\curl\hat\beta=\frac1\lambda F\hat b\otimes Q^T\hat t\calH^1\LL  Q^T\frac{\gamma-v}\lambda\in\calM^1(\R^3)$, where $\hat b(x):=b(\lambda Qx+v)$ 
 and $\hat t(x):=t(\lambda Qx+v)$.
\item\label{lemmarotatebetacurlpsic} 
For any $\C$ which obeys \eqref{eqdefC} and any $Q\in \SO(3)$
the line-tension energy $\psiC$ obeys
\begin{equation}
 \psiC(b,t)=\psi_{\C_Q}(Q^Tb,Q^Tt)
\end{equation}
where $\C_Q$ is defined by
\begin{equation}\label{eqdefCQapp}
  \C_Q A\cdot A =
 \C (QAQ^T)\cdot (QAQ^T).
\end{equation}
\end{enumerate}
\end{lemma}
We remark that \eqref{eqdefhatbeta}
is the same as
\begin{equation}
\beta(y):=F^{-1}\hat \beta(Q^Ty) Q^T.
\end{equation}
\begin{proof}
\ref{lemmarotatebetacurlbeta}:
We first observe that (summing over repeated indices)
\begin{equation}
 \eps_{ijk}Q_{hj}Q_{pk}=Q_{si}\eps_{shp}.
\end{equation}
Indeed, the left-hand side is the $i$-th component of
$(Q^Te_h)\wedge (Q^Te_p)=Q^T(e_h\wedge e_p)$, and $e_i\wedge e_{i+1}=e_{i+2}$ (with indices taken modulo 3).
We now compute (again with implicit sums)
$\hat\beta_{ak}(x)=F_{ab}\beta_{bp}(\lambda Qx+v)Q_{pk}$ and
\begin{equation}\begin{split}
 (\curl\hat \beta)_{ai}(x)
 &=\eps_{ijk}\partial_j \hat \beta_{ak}(x)
 =\eps_{ijk} F_{ab} Q_{pk}(\partial_h \beta_{bp})(\lambda Qx+v) \lambda Q_{hj}\\
& =\lambda F_{ab} Q_{si} \eps_{shp}(\partial_h \beta_{bp})(\lambda Qx+v) \\
&=\lambda F_{ab} (\curl \beta)_{bs}(\lambda Qx+v) Q_{si}.
 \end{split}
\end{equation}
This proves \eqref{curlhatbeta}.

\ref{lemmarotatebetacurlbetameas}:
If $\beta$ is regular, by \ref{lemmarotatebetacurlbeta}
\begin{equation}
 \int_A \curl\hat\beta(x) dx=
 \int_A \lambda F\curl\beta(\lambda Qx+v)Q dx=
\lambda \int_{\lambda QA+v} \frac{1}{\lambda^3}F\curl\beta(y)Q dy,
 \end{equation}
the general case follows by density. Further,
\begin{equation}
\begin{split}
\frac{1}{\lambda^2} F (b\otimes t\calH^1\LL\gamma)(\lambda QA+v)Q&=
\frac{1}{\lambda^2} F \int_{(\lambda QA+v)\cap\gamma} b\otimes td\calH^1Q\\
&=
\frac{1}{\lambda}F \int_{A\cap \lambda^{-1}Q^T(\gamma-v)} \hat b\otimes \hat td\calH^1Q,
\end{split}
\end{equation}
and $F(\hat b\otimes \hat t)Q=(F\hat b)\otimes (Q^T\hat t)$.

\ref{lemmarotatebetacurlpsic}:
We recall \eqref{eqdefabthRr}, which can be equivalently written as
 \begin{equation}\label{eqdefabthRranhang}
 \begin{split}
\infcyl(\C,b,t,h,R,r):=\inf \Big\{  \frac{1}{h\ln \frac Rr}\int_{Q_t T^{R,r}_h}& \frac12 \C \beta\cdot\beta dx:  
\beta\in L^{1}(Q_t T^{R}_h;\R^{3\times 3}),\\
 &
 \curl\beta=b\otimes t\calH^1\LL (\R t\cap Q_t T^{R}_h)\Big\}.
 \end{split}
 \end{equation}
 For any $\beta$ as above we define $\hat\beta$ by
\begin{equation}
 \hat \beta(x):=Q^T\beta(Qx) Q.
\end{equation}
Obviously $\hat\beta\in L^{1}(Q^TQ_t T_h^R;\R^{3\times 3})$, and by \ref{lemmarotatebetacurlbetameas} we have
\begin{equation}
\curl\hat\beta
=Q^Tb\otimes Q^Tt \calH^1\LL (\R Q^Tt \cap Q^TQ_t T_h^R).
\end{equation}
With a change of variables
 \begin{equation}
 \begin{split}
\int_{Q_t T^{R,r}_h} \frac12 \C \beta\cdot\beta dx
=&
\int_{Q_t T^{R,r}_h} \frac12 \C(Q\hat\beta(Q^Tx)Q^T)\cdot
(Q\hat\beta(Q^Tx)Q^T)dx\\
=&
\int_{Q^T Q_t T^{R,r}_h} \frac12 \C_Q \hat\beta\cdot\hat\beta  dy.
\end{split}
 \end{equation}
Then using that $Q^T Q_t T^{R,r}_h=Q_{Q^Tt} T^{R,r}_h$
and taking the infimum over all choices of $\beta$
gives
\begin{equation}
 \infcyl(\C_Q,Q^Tb,Q^Tt,h,R,r)
 \le \infcyl(\C,b,t,h,R,r).
\end{equation}
Taking the limit $r\to0$ with Lemma~\ref{lemmacellproblemlb}
\ref{lemmacellproblemlbA}
we obtain
\begin{equation}
 \psi_{\C_Q}(Q^Tb,Q^Tt)\le \psiC(b,t),
\end{equation}
and analogous computation proves the converse inequality.
\end{proof}

We remark that in the geometrically linear setting one identifies the reference with the deformed configuration, hence the discussion above becomes largely irrelevant.
The linearization procedure is based on considering deformations $u$ close to the identity, in the sense that $u(x)=x+\delta v(x)$, with $v$ the scaled displacement
and $\delta$ the linearization parameter. One then 
introduces the strain measure $\xi(x):=Dv(x)$,
computes $F_e(x)=\Id+\delta \xi(x)$ 
and $\beta(x+\delta v(x))=(\Id+\delta \xi(x))^{-1}$. A Taylor series for $\delta\to0$, if all fields are regular, gives
$\beta(x)=\Id-\delta \xi(x)+O(\delta^2)$, 
and in the linear theory the strain measure is identified with $\xi(x)= -\frac{\beta(x)-\Id}{\delta}$.
In the presence of dislocation, the passage through $u$ and $v$ can only be done locally, but linearization still amounts to $\xi=-\frac{\beta-\Id}{\delta}$.
Obviously $\curl\xi =-\frac1\delta \curl\beta$ is then the easiest measure of the dislocation density. If one uses a quadratic energy the prefactor $\delta$ can be dropped, and the minus sign amounts to a minor change in the definition of the Burgers vector. 

In this paper we provide a unified mathematical approach for the linear and the nonlinear setting, using different assumptions on the energy density $W$ presented in the next section.

\subsection{Main assumptions}

We introduce the class of dilute dislocations identified with a class of divergence-free measures concentrated on polyhedral curves. 

\def\segment{s}
\begin{definition}\label{defdiluteness-curve}
 Given two positive parameters $\alpha,h>0$ 
 and an open set $\Omega\subseteq\R^3$
 we say that a  polyhedral curve $\gamma\subset\overline\Omega$ is $(h,\alpha)$-{\em dilute} if it is the union of finitely many closed segments $\segment_j\subset\overline\Omega$  such that
 \begin{enumerate}
  \item each $\segment_j$ has length at least $h$;
  \item if $\segment_j$ and $\segment_i$ are disjoint then their distance is at least $\alpha h$;
  \item if the segments $\segment_j$ and $\segment_i$ are not disjoint then they share an endpoint, and the angle between them is at least $\alpha$;
  \item $\gamma$ does not have endpoints inside $\Omega$.
 \end{enumerate}
The set of $(h,\alpha)$-dilute polyhedrals is denoted by $P(h,\alpha)$.
\end{definition}
The diluteness condition given in the Definition \ref{defdiluteness-curve} allows to define the set of compatible configurations.
The asymptotic analysis will be performed assuming that  the diluteness parameters $h$ and $\alpha$ are much larger then the lattice spacing $\eps$, in a sense made precise in \eqref{eqdefheps} below.

\begin{definition}
\label{defdilutedistribution}
A $(h,\alpha)$-dilute dislocation distribution is a measure $\mu\in 
\calM^1(\Omega)$ of the form $\mu=\theta\otimes \tau\calH^1\LL\gamma$ such that $\gamma\in P(h,\alpha)$.

Given a Bravais lattice $\calB$
(i.e., a set of the form  $\calB=F\Z^3$ for some invertible matrix $F\in\R^{3\times 3}$), we denote by
$\calM^1_\calB(\Omega)$ the set of measures $\mu\in\calM^1(\Omega)$ such that $\theta\in \calB$ $\calH^1$-almost everywhere on $\gamma$.
\end{definition}

We consider an elastic energy density $W$ in one of  two natural frameworks.
To simplify notation in the presence of a mixed growth condition we define for $p\in[1,2]$ the function {$\mixedgrowth:[0,\infty)\to[0,\infty)$}, by
\begin{equation}\label{eqdefmixedgrowth}
 \mixedgrowth(t):=t^p\wedge t^2.
\end{equation}

\paragraph{Assumption \HWFinite:} In a geometrically nonlinear setting, 
$W:\R^{3\times 3}\to[0,\infty)$  is minimized at the identity and is invariant under the right action of $\SO(3)$,
\begin{equation}\label{eqWfWFQ}
 {W(F)=W(FQ)} \text{ for all } F\in \R^{3\times 3}, Q\in \SO(3).
\end{equation}
We assume also that $W$ is Borel, twice differentiable in a neighbourhood of $\Id$, and
\begin{equation}\label{eqlowerboundWfinite}
\begin{split}
 \frac1c \mixedgrowth(\dist(F,\SO(3)))
\le W(F) 
\le c \mixedgrowth(\dist(F,\SO(3))),
\end{split}
\end{equation}
for some $p\in (1,2]$ and $c>0$.

\paragraph{Remark.} If \HWFinite\ holds, then the tensor 
\begin{equation}\label{eq-C-tensor-nonlin}
 \C:=D^2W(\Id)
\end{equation}
satisfies condition \eqref{eqdefC},
{and there is a modulus of continuity $\omega:[0,\infty)\to[0,\infty)$ with $\omega(\delta)\to0$ as $\delta\to0$ such that
\begin{equation}
 \left|W(\Id+A)-\frac12 \C A\cdot A \right|\le \omega(|A|)|A|^2.
\end{equation}
For any $Q\in\SO(3)$, using 
$W(Q(\Id+A))=W(Q(\Id+A)Q^T)$ one obtains
\begin{equation}\label{eqWQIdACQ}
 \left|W(Q(\Id+A))-\frac12 \C_Q A\cdot A \right|\le \omega(|A|)|A|^2,
\end{equation}
where $\C_Q\in\R^{3\times 3\times 3\times 3}_\sym$  is defined from $\C$ and $Q$ by
\begin{equation}\label{eqdefCQ}
  \C_Q A\cdot A =
 \C (QAQ^T)\cdot (QAQ^T).
\end{equation}
It is easy to see that it also obeys condition \eqref{eqdefC}.}

\paragraph{Assumption \HWLin:} In a geometrically linear setting, 
$W:\R^{3\times 3}\to[0,\infty)$  
is minimized at zero and only depends on the symmetric part of its argument, 
\begin{equation}
 W(F)=W(F+S) \text{ for all } F\in \R^{3\times 3}, S\in \R^{3\times 3}_\skw.
\end{equation}
We assume also that $W$ is Borel, twice differentiable in a neighbourhood of 0,
and
\begin{equation}\label{eqlowerboundWlin}
 \frac1c \mixedgrowth(|F+F^T|) \le W(F) 
 \le c \mixedgrowth(|F+F^T|) ,
\end{equation}
for some $p\in (1,2]$ and $c>0$.

\paragraph{Remark.} If \HWLin\ holds, then the tensor 
\begin{equation}\label{eq-C-tensor-lin}
 \C:=D^2W(0)
\end{equation}
satisfies condition \eqref{eqdefC}.

It will be useful to have some properties of $\mixedgrowth$. We denote by $\mixedgrowth^{**}$ its convex envelope, which obeys
 \begin{equation}\label{eqmixedgrowthconv} \mixedgrowth^{**}\le\mixedgrowth\le c_p \mixedgrowth^{**},
\end{equation}
for some $c_p>0$.
\begin{remark}\label{rem-phip}
{$(i)$}
The function $\mixedgrowth$
 for every $\delta>0$ satisfies
  \begin{equation}\label{eqlemmaVabG2}
  \mixedgrowth(|a+b|)\le (1+\delta) \mixedgrowth(|a|)+\left(1+\frac1\delta\right) \mixedgrowth(|b|) \quad\text{ for all } a,b\in\R.
 \end{equation}
In particular, with $\delta=1$ we obtain
 \begin{equation}\label{eqlemmaVabG}
  \mixedgrowth(|a+b|)\le 2 \mixedgrowth(|a|){+2\mixedgrowth(|b|) \quad\text{ for all } a,b\in\R}.
 \end{equation}
 To prove \eqref{eqlemmaVabG2}, we 
 distinguish two cases. If $|a|\le1$ and $|b|\le 1$, then
  \begin{equation*}\begin{split}
   \mixedgrowth(|a+b|) \le& |a+b|^2  \\
   \le&
   (1+\delta) |a|^2+(1+\frac1\delta)|b|^2=(1+\delta) \mixedgrowth(|a|)+(1+\frac1\delta)\mixedgrowth(|b|).
\end{split}  \end{equation*}
 Otherwise, we write
  \begin{equation*}\begin{split}
   \mixedgrowth(|a+b|) \le |a+b|^p \le& (|a|+|b|)^p = (|a|+|b|)^2 (|a|+|b|)^{p-2}\\
   \le &
   \left[(1+\delta) |a|^2+(1+\frac1\delta)|b|^2\right](|a|+|b|)^{p-2}.
\end{split}  \end{equation*}
We recall that $p-2\le0$, and that in this case $|a|+|b|\ge1$.
If $|a|\le 1$, then  $|a|^2(|a|+|b|)^{p-2}\le |a|^2=\mixedgrowth(|a|)$. If 
$|a|> 1$, then  $|a|^2(|a|+|b|)^{p-2}\le
|a|^2|a|^{p-2}=|a|^p=\mixedgrowth(|a|)$. Therefore
$|a|^2(|a|+|b|)^{p-2}\le\mixedgrowth(|a|)$, and the same for $b$. This concludes the proof.

{
$(ii)$ For any  $f\in L^p(\Omega;\R^{3\times 3})$ there are $a,b\in L^p(\Omega;\R^{3\times 3})$  with $|a|\le 1$ everywhere, $|b|\ge1$ wherever $b\ne0$, $f=a+b$, and
\begin{equation}\label{2prigidity_bis}
\int_{\Omega} (|a|^2+ |b|^p )dx =
\int_{\Omega} \mixedgrowth(|f|)dx.
\end{equation}
To see this, it suffices to set $a:=f\chi_{\{|f|\le1\}}$, and $b:=f-a$.
}
\end{remark}

\subsection{$\Gamma$-convergence to a line tension model}
\label{secgammaconv}
We now introduce a small parameter $\eps>0$, which in this semidiscrete model represents the lattice spacing, and 
(given an open set $\Omega\subseteq\R^3$ and a Bravais lattice $\calB\subset\R^3$)
the class of admissible configurations 
\begin{equation}\label{Astar}
 \calA^*_\eps:=\{(\mu,\beta)\in 
 \calM_{\eps\calB}^1(\Omega)\times   L^1(\Omega;\R^{3\times 3})
 :
 \curl\beta=\mu\}.
\end{equation}
 For a Borel set $A\subseteq\Omega$ we define
\begin{equation}\label{eqdefFepsbetaA}
 F_\eps[\beta,A]:=\frac{1}{\eps^2\ln\frac1\eps}
 \int_A W(\beta) dx
 =:\frac{1}{\eps^2\ln\frac1\eps}\Elast[\beta,A]
\end{equation}
and write briefly  $F_\eps[\beta]:=F_\eps[\beta,\Omega]$.
The asymptotic analysis will be performed for any diluteness parameters $\alpha_\eps$ and $h_\eps$
that obey
\begin{equation}\label{eqdefheps}
\lim_{\eps\to0} \frac{\log \frac{1}{\alpha_\eps h_\eps}}{\log \frac1\eps}=
\lim_{\eps\to0} \alpha_\eps=
\lim_{\eps\to0} h_\eps
=0\,.
\end{equation}

For pairs in the set of admissible configurations $\calA^*_\eps$ we shall use the following notion of convergence.
\begin{definition}\label{defconvergence}
We say that 
$(\mu_\eps,\beta_\eps)\in \mathcal{M}_{\eps\calB}^1(\Omega)\times L^1(\Omega;\R^{3\times3})$ converges to $(\mu,\eta,Q)
\in\mathcal M^1_{\calB}(\Omega)\times L^1_\loc(\Omega;\R^{3\times 3}) \times \SO(3)
$ in finite kinematics with $p$ growth if  
{$\frac1\eps\mu_\eps$ locally weak-* converges to $\mu$, in the sense that}
\begin{equation}\label{eqweakconvmuepsmu}
\frac1\eps\mu_\eps\weakstarto\mu         \hskip5mm
{\text{ in $\Omega'$, for all $\Omega'\subset\subset\Omega$}}
\end{equation}
and there are $Q_\eps\in \SO(3)$ such that $Q_\eps\to Q$ and 
\begin{equation}\label{eqetaepstheorem}
 \frac{Q_\eps^T\beta_\eps-\Id}{\eps(\ln\frac1\eps)^{1/2}}
 \weakto \conteta \text{ weakly in } L^{q}_\loc(\Omega;\R^{3\times 3})
 \end{equation}
 for {$q:=\frac32\wedge p$}.

 We say that 
$(\mu_\eps,\beta_\eps)\in \mathcal{M}_{\eps\calB}^1(\Omega)\times L^1(\Omega;\R^{3\times3})$ converges to $(\mu,\eta)
\in\mathcal M^1_\calB(\Omega)\times L^1_\loc(\Omega;\R^{3\times 3}) $ 
in infinitesimal kinematics with $p$ growth if  
\eqref{eqweakconvmuepsmu} holds and 
\begin{equation}\label{eqetaepstheoremlin}
 \frac{\beta_\eps}{\eps(\ln\frac1\eps)^{1/2}}
 \weakto \conteta \text{ weakly in } L^{q}_\loc(\Omega;\R^{3\times 3}).
 \end{equation}
 \end{definition}

{The local weak-$*$ convergence in \eqref{eqweakconvmuepsmu} can be equivalently defined testing with elements of $C^0_c(\Omega)$.}

In what follows, for $b\in \calB$ and $t\in S^2$, the function $\psi_\C^{\rm rel}(b,t)$ denotes the $\calH^1$-elliptic envelope of $\psi_\C(b,t)$  and it is given by
\begin{equation}\label{psi-rel}
 \begin{split}
  \psiC^{\rm rel}(b,t):=\inf\{&\int_\gamma\psi_\C(\theta(x),\tau(x))d\calH^1(x)\colon\ \nu=\theta\otimes\tau\calH^1\LL\gamma\in\calM_{\calB}^1(B_{1/2}),\\
  &\supp(\nu-b\otimes t\calH^1\LL(\R t\cap B_{1/2}))\subset\subset B_{1/2}
  \}.
 \end{split}
\end{equation}
In \cite{ContiGarroniMassaccesi2015} it is proven that $\psi_\C^{\rm rel}$ provides the energy density of the relaxation of the line tension energy given in \eqref{eq-line-tension}.
{By a change of variables, one easily sees that the same holds for the function $\hat\psi(b,t):=\psi(b,Qt)$, for any $Q\in\SO(3)$.}

The main result of the paper is then the following compactness and $\Gamma$-convergence statement {in the subcritical regime $p<2$}.

\begin{theorem}\label{thm:Gamma-limit}
Let $\Omega\subset\R^3$ be a bounded Lipschitz set and assume that $(h_\eps,\alpha_\eps)$ obeys \eqref{eqdefheps}.
Assume also that $W$ obeys \HWFinite\ for some $p\in(1,2)$.
Then the functionals
\begin{equation}
F_\eps^\subcr[\mu,\beta]:=
\begin{cases}
  F_\eps[\beta], & \text{ if }
  (\mu,\beta)\in \calA^*_\eps \text{ and $\mu$ is  $(h_\eps,\alpha_\eps)$-dilute},\\
  \infty ,& \text{ otherwise,}
 \end{cases}
\end{equation}
$\Gamma$-converge, with respect to the convergence in finite kinematics with $p$ growth in the sense of Definition~\ref{defconvergence}, to
\begin{equation}\label{eqgammalimtheorem}
\Ffinite[\mu,\eta,Q]:=\int_\Omega \frac12 \C_Q \eta\cdot \eta dx + \int_\gamma \psi_\C^\rel (b,Qt) d\calH^1,
\end{equation}
if  $(\mu,\eta,Q)\in \mathcal M^1_{\calB}(\Omega)\times L^2(\Omega;\R^{3\times 3}) \times \SO(3)$, 
 $\mu=b\otimes t\calH^1\LL\gamma$ and $\curl\eta=0$, and $\infty$ otherwise, 
 $\C_Q$ as in \eqref{eq-C-tensor-nonlin} and \eqref{eqdefCQ}, $\psi_\C^\rel$  as in \eqref{psi-rel}.

 Further, if $\Omega$ is connected then any 
sequence with $F_\eps^\subcr[\mu_\eps,\beta_\eps]$ bounded has a subsequence that converges in the same topology.

If instead $W$ obeys \HWLin, then the corresponding assertions hold 
with respect to convergence in infinitesimal kinematics with $p$ growth, with \eqref{eqgammalimtheorem} replaced by
\begin{equation}\label{eqgammalimlim}
\Flin[\mu,\eta]:=\int_\Omega \frac12 \C \eta\cdot \eta dx + \int_\gamma \psi_\C^\rel (b,t) d\calH^1
\end{equation}
if  $(\mu,\eta)\in \mathcal M^1_{\calB}(\Omega)\times L^2(\Omega;\R^{3\times 3})$, 
 $\mu=b\otimes t\calH^1\LL\gamma$ and $\curl\eta=0$, and $\infty$ otherwise, with $\C$ as in \eqref{eq-C-tensor-lin}.
\end{theorem}

We remark that there are several equivalent ways of treating the sequence of rotations $Q_\eps\to Q$. One simple observation is that  by Lemma~\ref{lemmarotatebetacurl}\ref{lemmarotatebetacurlpsic}, $\psi_\C^\rel (b,Qt)=\psi_{\C_Q}^\rel (Q^Tb,t)$, 
which can be inserted in the second term of 
\eqref{eqgammalimtheorem}.
One can also replace $\C_Q$ by $\C$ in the first term, if the fields are redefined accordingly; in order to keep 
the differential constraints one should also rotate the domain of integration. Specifically, given a sequence 
$(\mu_\eps,\beta_\eps)
\in \calA^*_\eps(\Omega)$ (defined as in \eqref{Astar}) that converges as above, 
 for each $\eps$ one considers the pair $(\hat\mu_\eps,\hat\beta_\eps)\in \calA^*_\eps(Q_\eps\Omega)$ defined by
\begin{equation}
 \hat \beta_\eps(y):=\beta_\eps(Q_\eps^T y) Q_\eps^T\,,\hskip1cm
 \hat\mu_\eps (A):=\mu_\eps(Q_\eps^TA)Q_\eps^T.
\end{equation}
Then $F_\eps^\subcr[\hat\mu_\eps,\hat\beta_\eps,Q_\eps\Omega]=
F_\eps^\subcr[\mu_\eps,\beta_\eps,\Omega]$,
$\hat \mu_\eps\weakstarto \hat\mu:=b\otimes \hat t \calH^ 1\LL Q\gamma$, with $\hat t(y):=Q t(Q^Ty)$,
and setting 
$\hat\eta(y):=Q\eta(Q^Ty)Q^T$ condition \eqref{eqetaepstheorem} becomes
\begin{equation}
 \frac{\hat\beta_\eps-\Id}{\eps(\log\frac1\eps)^{1/2}} \weakto \hat \eta \text{ weakly in } L^q_\loc(Q\Omega;\R^{3\times 3})
\end{equation}
(the fact that the domain changes along the sequence is not a problem for $L^q_\loc$ convergence since $Q_\eps\to Q$). The limiting functional in 
\eqref{eqgammalimtheorem} then takes the form
\begin{equation}
\int_{Q\Omega} \frac12 \C \hat\eta\cdot \hat\eta dx + \int_{\hat \gamma} \psi_\C^\rel (b,\hat t) d\calH^1,
\end{equation}
with $\hat\mu=b\otimes \hat t \calH^1\LL\hat\gamma$ and $\hat t$ the unit tangent to $\hat\gamma:=Q\gamma$.
For simplicity we stick to the formulation in which the integration domain $\Omega$ is fixed.

Related results have been proven before in  \cite{ContiGarroniOrtiz2015,garroni2020nonlinear}. We present here a more general argument that permits to prove Theorem \ref{thm:Gamma-limit} in a unified way for finite and infinitesimal kinematics. Our argument also provides the $\Gamma$-convergence result for different types of core regularizations, which are needed in the case $p=2$ and are discussed in Subsection \ref{secdiffregular}. 

In Section~\ref{seccellpb} we give a unified treatment for the cell problem formula. Building on this, compactness and lower bound are then proven in Section~\ref{seccomp}. 
In particular, we shall  introduce in \eqref{eqdefelbeps} an auxiliary functional, for which we will prove the lower bound result. 
The auxiliary functional
is chosen so that, after rescaling, it is below $F_\eps^\subcr $, with corresponding bounds holding for other regularizations as discussed below,
and then all results follow at once.
Analogously, in Section \ref{secupperbound}, the upper bound will be proven for a second auxiliary functional  defined in \eqref{eqdefeubeps}, which after rescaling is
larger than $F_\eps^\subcr$, up to a small error term which can be controlled. Moreover we show that the recovery sequence needed for this upper bound has a good decay near the dislocation line, so that different types of core regularization do not change its energy asymptotically.

In Section \ref{sec-mainproofs} we will collect the proofs of the $\Gamma$-convergence results (Theorem \ref{thm:Gamma-limit} and Theorem \ref{thm:Gamma-limit2} below) that are obtained as straightforward consequences of the results proved in Section~\ref{seccellpb}, Section~\ref{seccomp}  and Section \ref{secupperbound}.
All proofs are given explicitly in the case \HWFinite. The case \HWLin\ is very similar therefore we only point out a few relevant differences.

\subsection{Extension to different core regularizations}
\label{secdiffregular}

In the case of quadratic growth (in the sense $p=2$) one needs a different regularization and possibly a different set of admissible configurations.  This requires different convergence properties and compactness results. 
The $\Gamma$-convergence result instead can be obtained as a consequence of the 
arguments developed to prove Theorem~\ref{thm:Gamma-limit}.

In the literature one considers configurations $(\mu, \beta)\in\calM^1(\Omega)\times L^1(\Omega;\R^{3\times3})$ with $\curl \beta= \mu$ and the associated energy 
\begin{equation}\label{eq-energy-core}
 \int_{\Omega\setminus (\supp\mu)_\eps} W(\beta) dx,
\end{equation}
where  $(\supp\mu)_\eps:=\{x\in\Omega\colon\dist(x,\supp\mu)<\eps\}$.
This is known as  the core-region approach.
Alternatively, one can replace the  condition $\curl\beta=\mu$
with a different condition
{that only involves the behavior of $\beta$  away from the singularity},
that we will call $\rho$-compatibility of the pair $(\mu,\beta)$, see Definition~\ref{defbetamucompatible} below. Since the energy \eqref{eq-energy-core} does not depend on the value of $\beta$ inside the core region, one can show that the asymptotics of the rescaled energy does not depend on the set of admissible configurations chosen.
Nevertheless for the compactness result an extension argument is needed, which goes beyond the scope of the present work.
A third option is a regularization via mollification of the measure, which smears out the singularity on a scale $\eps$.

\begin{definition}\label{defbetamucompatible}
Let $\Omega\subset\R^3$ be a bounded Lipschitz set, $\rho>0$.
A pair $(\mu,\beta)\in  \calM^1(\Omega)\times L^1(\Omega\setminus (\supp\mu)_\rho;\R^{3\times 3})$ is 
$\rho$-compatible if
there are an extension $\tilde\mu\in\mathcal M^1(\R^3)$ of $\mu$ and $\beta_0\in L^{1}(\Omega;\R^{3\times3})$ such that  $\curl\beta_0=0$ in $\Omega$ and $\beta=\beta^{\tilde\mu}+\beta_0$ in $\Omega\setminus (\supp\tilde\mu)_\rho$, where $\beta^{\tilde\mu}$ is the solution obtained from $\tilde\mu$ via Theorem~\ref{theoremsolr3}.
\end{definition}
{We recall that 
by \cite[Lemma 2.3]{ContiGarroniMassaccesi2015}) 
for any 
$\mu\in\calM_{\eps\calB}^1(\Omega)$
there exists an extension $\tilde\mu\in\calM_{\eps\calB}^1(\R^3)$ such that $\mu=\tilde\mu\LL\Omega$ and $|\tilde\mu|(\R^3)\le c(\Omega)|\mu|(\Omega)$.}

We define, for $\eps>0$ and $\rho>0$,
\begin{equation}\label{eqdefacoreepss}
 \calA^\core_{\rho,\eps}:=\{(\mu,\beta)\in 
 \calM_{\eps\calB}^1(\Omega)\times   L^1(\Omega;\R^{3\times 3})
 :
 (\mu,\beta) \text{ is $\rho$-compatible}\},
\end{equation}
and
\begin{equation}\label{eq:Amoll}
 \begin{split}
  \calA^\moll_\eps:=\{(\mu,\beta)\in 
  \calM_{\eps\calB}^1(\Omega)\times   &L^1(\Omega;\R^{3\times 3})
  :
  \curl\beta=\varphi_\eps*\tilde\mu,\\
  &\text{ for some extension $\tilde\mu\in\calM_{\eps\calB}^1(\R^3)$ of $\mu$}
  \}.
 \end{split}
\end{equation}
Here $\varphi_\eps(x):=\eps^{-3}\varphi(x/\eps)$ is a fixed mollifier at scale $\eps$; as usual $\varphi\in C^\infty_c(B_1;[0,\infty))$ and $\|\varphi\|_{L^1}=1$.

\begin{lemma}\label{lem:comparison-admissible-pairs}
 \begin{enumerate}
  \item\label{lem:comparison-admissible-pairscoremon} 
  $\calA^{\core}_{\rho,\eps}\subseteq
\calA^{\core}_{\rho',\eps}$ 
   for every $\rho'\ge\rho$;
  \item\label{lem:comparison-admissible-pairsmollcore} $\calA_\eps^\moll\subseteq$ for every $\rho\ge\eps$;
  \item\label{lem:comparison-admissible-pairsstarcore} 
 $\calA^*_\eps\subseteq\calA^{\core}_{\rho,\eps}$ for every $\rho>0$. 
  \end{enumerate} 
\end{lemma}
\begin{proof}
\ref{lem:comparison-admissible-pairscoremon} is immediate from the definition.

 \ref{lem:comparison-admissible-pairsmollcore}:
 Let $(\mu,\beta)\in \calA_\eps^\moll$, and select an extension $\tilde\mu$ as in the definition of 
 $\calA_\eps^\moll$.
 We define $\beta_1:=\beta^{\tilde\mu\ast\varphi_\eps}-
 \beta^{\tilde\mu}$, 
 and fix $\eps'\in(0,\eps)$ such that $\supp\varphi_\eps\subset B_{\eps'}$. By the integral representation in \eqref{eqbetaNmu} we have that $\beta_1\in C^\infty(\R^3\setminus (\supp\tilde\mu)_{\eps'};\R^{3\times 3})$. We show that $\beta_1$ is exact, in the sense that the circulation on every closed curve $\gamma\subset
 \R^3\setminus (\supp\tilde\mu)_{\eps'}$ vanishes,
 \begin{equation}\label{eqbeta1circ}
  \int_\gamma \beta_1  d\tau =0.
 \end{equation}
To see this, we observe that (again by  \eqref{eqbetaNmu}) $\beta^{\tilde\mu\ast\varphi_\eps}=\beta^{\tilde\mu}\ast\varphi_\eps$, which implies
 \begin{equation}
  \int_\gamma \beta^{\tilde\mu\ast\varphi_\eps} d\tau =
  \int_\gamma \int_{B_{\eps'}} 
  \beta^{\tilde\mu}(\cdot-z)\varphi_\eps(z) dz   d\tau=
  \int_{B_{\eps'}} \varphi_\eps(z)\left(
  \int_{\gamma-z} \beta^{\tilde\mu}  d\tau\right) dz. 
 \end{equation}
Since $\curl\beta^{\tilde\mu}=0$ in $\R^3\setminus \supp\tilde\mu$, 
and $\gamma-z\subset \R^3\setminus \supp\tilde\mu$ for all $z\in B_{\eps'}$,
the last integrand does not depend on $z$. Therefore
 \begin{equation}
  \int_\gamma \beta^{\tilde\mu\ast\varphi_\eps}  d\tau =
  \int_{\gamma} \beta^{\tilde\mu} d\tau 
 \end{equation}
which proves \eqref{eqbeta1circ}. In turn, this implies that there is $u_1\in C^\infty(\R^3\setminus ({\supp\tilde\mu})_{\eps'};\R^3)$ such that 
$\beta_1=Du_1$ in this set.
We now select $\theta\in C^\infty_c([0,\eps))$ such that $\theta=1$ on $[0,\eps']$, and  define $u_0\in W^{1,\infty}_\loc(\R^3;\R^3)$ by
\begin{equation}
 u_0(x):=u_1(x) (1-\theta(\dist(x,\supp \tilde\mu))).
\end{equation}
At this point we define $\beta_0:=Du_0+\beta-\beta^{\tilde\mu\ast\varphi_\eps}
{\in L^1(\Omega;\R^{3\times 3})}$.
From $\curl\beta=\varphi_\eps\ast\tilde\mu$ in $\Omega$
we obtain $\curl\beta_0=0$ in $\Omega$; from $Du_0=Du_1=
\beta^{\tilde\mu\ast\varphi_\eps}-
 \beta^{\tilde\mu}$ in  $\Omega\setminus (\supp\tilde\mu)_\eps$ we obtain 
$\beta=\beta^{{\tilde\mu}}+\beta_0$ in $\Omega\setminus (\supp\tilde\mu)_\eps$.
Hence $\beta\in \calA^{\core}_{\eps,\eps}\subseteq \calA^{\core}_{\rho,\eps}$ for all $\rho\ge \eps$.

 \ref{lem:comparison-admissible-pairsstarcore}:
Let $(\mu,\beta)\in \calA^*_\eps$, 
{and let $\tilde\mu\in\calM_{\eps\calB}^1(\R^3)$ be an extension of $\mu$.}
Setting $\beta_0:=\beta-\beta^{\tilde\mu}$, one sees that
$\curl\beta_0=\curl(\beta-\beta^{\tilde\mu})=0$ in $\Omega$, which concludes the proof.
 \end{proof}

In the set of admissible configurations $\calA^\core_{\rho_\eps,\eps}$ (with $\rho_\e\to0$) we introduce the following convergence, which is a variant of the one in Definition~\ref{defconvergence}. The key difference is that now the value of $\beta_\eps$ in the core region $(\supp\mu_\eps)_{\rho_\eps}$ is ignored.
\begin{definition}\label{defconvergence2}
We say that 
$(\mu_\eps,\beta_\eps)\in \mathcal{M}_{\eps\calB}^1(\Omega)\times L^1(\Omega;\R^{3\times3})$ converges to $(\mu,\eta,Q)
\in\mathcal M^1_\calB(\Omega)\times L^1_\loc(\Omega;\R^{3\times 3}) \times \SO(3)
$ in finite kinematics with $p$ growth and radius $\rho_\eps$ if  $\rho_\eps\to0$,
\begin{equation}\label{eqweakconvmuepsmu2}
\frac1\eps\mu_\eps\weakstarto\mu         \hskip5mm
{\text{ in $\Omega'$, for all $\Omega'\subset\subset\Omega$,}}
\end{equation}
and there are $Q_\eps\in \SO(3)$ such that $Q_\eps\to Q$ and 
\begin{equation}\label{eqetaepstheorem2}
 \frac{Q_\eps^T\beta_\eps-\Id}{\eps(\ln\frac1\eps)^{1/2}}  \chi_{\Omega\setminus (\supp\mu_{\eps})_{\rho_{\eps}}}
 \weakto \conteta \text{ weakly in } L^{q}_\loc(\Omega;\R^{3\times 3})
 \end{equation}
 for {$q:=\frac32\wedge p$}.
 
 We say that 
$(\mu_\eps,\beta_\eps)\in \mathcal{M}_{\eps\calB}^1(\Omega)\times L^1(\Omega;\R^{3\times3})$ converges to $(\mu,\eta)
\in\mathcal M^1_\calB(\Omega)\times L^1_\loc(\Omega;\R^{3\times 3}) $ 
in infinitesimal kinematics with $p$ growth and radius $\rho_\eps$ if $\rho_\eps\to0$,
\eqref{eqweakconvmuepsmu2} holds, and 
\begin{equation}\label{eqetaepstheoremlin2}
 \frac{\beta_\eps}{\eps(\ln\frac1\eps)^{1/2}}  \chi_{\Omega\setminus (\supp\mu_{\eps})_{\rho_\eps}}
 \weakto \conteta \text{ weakly in } L^{q}_\loc(\Omega;\R^{3\times 3}).
 \end{equation}
 \end{definition}

We then have the following $\Gamma$-convergence result {which includes the critical case $p=2$}.

\begin{theorem}\label{thm:Gamma-limit2}
Let $\Omega\subset\R^3$ be a bounded Lipschitz set and assume that $(h_\eps,\alpha_\eps)$ obeys \eqref{eqdefheps}. Assume also that $W$ obeys \HWFinite \ for some $p\in(1,2]$.
Let $\rho_\eps\to0$ be
such that {$\rho_\eps\ge\eps$} and
\begin{equation}
 \lim_{\eps\to0} \frac{\ln \rho_\eps}{\ln\eps}=1.
\end{equation}
\begin{enumerate}
 \item 
The functionals
\begin{equation}
F_\eps^*[\mu,\beta]:=
\begin{cases}
  F_\eps[\beta,\Omega\setminus{(\supp\mu)_{\rho_\eps}}], & \text{ if  $(\mu,\beta)\in\mathcal A_\eps^*$,\
  $\mu$ is $(h_\eps,\alpha_\eps)$-dilute,}\\
  \infty, & \text{ otherwise,}
 \end{cases}
\end{equation}
$\Gamma$-converge to the functional $\Ffinite$ defined in \eqref{eqgammalimtheorem} with respect to the convergence in finite kinematics with $p$ growth and radius $\rho_\eps$  in Definition~\ref{defconvergence2}.
\item The same holds also
for the functionals
\begin{equation}
F_\eps^\core[\mu,\beta]:=
\begin{cases}
  F_\eps[\beta,\Omega\setminus{(\supp\mu)_{\rho_\eps}}], & \text{ if  $(\mu,\beta)\in \calA^\core_{{\rho_\eps,\eps}}$}\\
  &\text{ and 
  $\mu$ is $(h_\eps,\alpha_\eps)$-dilute},\\
  \infty, & \text{ otherwise}.
 \end{cases}
\end{equation}
\item\label{thm:Gamma-limit2moll} The same holds also for the functionals 
\begin{equation}
F_\eps^\moll[\mu,\beta]:=
\begin{cases}
  F_\eps[\beta,\Omega], & \text{ if  $(\mu,\beta)\in \calA^\moll_\eps$,\
  $\mu$ is $(h_\eps,\alpha_\eps)$-dilute,}\\
  \infty, & \text{ otherwise}
 \end{cases}
\end{equation}
with respect to the convergence of Definition~\ref{defconvergence}.
 \end{enumerate}
If  instead $W$ obeys \HWLin, then the corresponding assertions hold 
with respect to convergence in infinitesimal kinematics, with $\Ffinite$ replaced by $\Flin$ defined  in
\eqref{eqgammalimlim}.
\end{theorem}

We remark that this result does not contain a compactness statement. This requires different arguments and will be addressed elsewhere. 
Only in the case of assertion \ref{thm:Gamma-limit2moll} compactness can be obtained from the rigidity estimate as in the case of Theorem~\ref{thm:Gamma-limit}.

\section{Straight dislocations in a cylinder}
\label{seccellpb}

This section deals with the cell problem, which describes a single straight dislocation in a cylinder. We first present in Lemma~\ref{lemmarigidcylinderhole} a rigidity estimate with mixed growth for an hollow cylinder, then in Lemma~\ref{lemmacompactnesscellpb} the standard coercivity statement which shows the origin of the logarithmic divergence. Afterwards, in Proposition~\ref{propcellproblem}, we prove a lower bound, relating the energy in a cylinder around a straight dislocation to the line-tension energy $\psiC$ defined in \eqref{eqvarprobbetzabt}.

\subsection{Rigidity and coercivity}

The first result shows that the constant in the rigidity estimate with mixed growth from
\cite[Theorem~1.1]{ContiDolzmannMueller2014} for domains $T_h^{R,r}$ 
with $r\le \frac12 R$ and $R\le h$
can be estimated with $h^2/R^2$, and in particular does not depend on $r$.
We recall the definition of the tubes in  \eqref{def:hol_cyl} and \eqref{def:full_cyl}, and the definition of the function $\Phi_p$ in \eqref{eqdefmixedgrowth} and its properties (see Remark \ref{rem-phip}).

\begin{lemma}\label{lemmarigidcylinderhole}
For any $p\in(1,2]$
 there is $c=c(p)>0$ such that for any $r,R,h>0$ with $2r\le R\le h$, any $Q'\in\SO(3)$, and any $\beta\in L^1(Q'T_h^{R,r};\R^{3\times 3})$ with $\curl\beta=0$ there are $Q\in \SO(3)$ such that
 \begin{equation}\label{rigidity-cylinder}
  \int_{Q'T_h^{R,r}} \mixedgrowth(|\beta-Q|) dx \le c \frac{h^2}{R^2}\int_{Q'T_h^{R,r}} \mixedgrowth(\dist(\beta,\SO(3))) dx 
 \end{equation}
and  {$S\in \R^{3\times 3}_\skw$} such that
 \begin{equation}\label{korn-cylinder}
  \int_{Q'T_h^{R,r}} \mixedgrowth(|\beta-S|) dx \le c \frac{h^2}{R^2}\int_{Q'T_h^{R,r}} \mixedgrowth(|\beta+\beta^T|) dx .
 \end{equation}
 \end{lemma}

\begin{proof}
 We first show that it suffices to prove both assertions for $Q'=\Id$. Indeed, given
 $\beta\in L^1(Q'T_h^{R,r};\R^{3\times 3})$
 one defines
 $\hat\beta\in L^1(T_h^{R,r};\R^{3\times 3})$
 by \begin{equation}
\hat\beta(x):=Q'^T\beta(Q'x)Q'.     
    \end{equation}
Then (by Lemma~\ref{lemmarotatebetacurl}\ref{lemmarotatebetacurlbeta}) $\curl\hat\beta=0$, $\dist(\hat\beta(x),\SO(3))=\dist(\beta(Q'x),\SO(3))$, 
 $|\hat\beta+\hat\beta^T|(x)
 =|\beta+\beta^T|({Q'x})$,
  and 
 $|\hat\beta-A|(x)=|\beta-Q'AQ'^T|(Q'x)$ for any $A\in\R^{3\times 3}$. 
 Therefore the assertions for $\hat\beta$ translate in the desired estimates for $\beta$. In the rest of the proof we deal with $Q'=\Id$.
 
 By scaling we can assume $R=1$. 
  We next show that it suffices to prove both assertions for $h=1$. Indeed, given $h>1$ we 
fix $0=z_0<z_1<\dots <z_N=h-1$ such that
 $z_{i}\le z_{i-1}+\frac12$ 
 {for $1\le i\le N$}, 
 $z_{i}\ge z_{i-2}{+ \frac12}$ {for $2\le i\le N$}, and $N= \lceil 2h-2\rceil$ (for example, $z_i=i/2$ for $i<N$, and $z_N=h-1$). We then apply \eqref{rigidity-cylinder} to each cylinder $T_1^{1,r}+z_ie_3$ and obtain matrices $Q_i\in \SO(3)$ such that
 \begin{equation}\label{eqphipbetaqi}
  \int_{T_1^{1,r}+z_ie_3} \mixedgrowth(|\beta-Q_i|) dx \le c \int_{T_1^{1,r}+z_ie_3} \mixedgrowth(\dist(\beta,\SO(3))) dx .
 \end{equation}
By \eqref{eqlemmaVabG},
 \begin{equation}\label{eqphipqiqi1a}
\Phi_p( |Q_i-Q_{i-1}|)\le 2\mixedgrowth(|\xi-Q_i|)+2
 \mixedgrowth(|\xi-Q_{i-1}|)
 \end{equation}
for every $\xi\in\R^{3\times 3}$. 
 From $r\le\frac12$ and $z_i\le z_{i-1}+\frac12$ we obtain
$\calL^3((T_1^{1,r}+z_ie_3)\cap (T_1^{1,r}+z_{i-1}e_3))\ge \frac38\pi$, so that
\begin{equation}\label{eqphipqiqi1b}
 \Phi_p( |Q_i-Q_{i-1}|)\le
 c \int_{T_1^{1,r}+z_ie_3} \mixedgrowth(|\beta-Q_i|) dx+c \int_{T_1^{1,r}+z_{i-1}e_3} \mixedgrowth(|\beta-Q_{i-1}|) dx.
\end{equation}
Using  $z_{i}\ge z_{i-2}{+ \frac12}$ we see that the overlap is finite, hence
\begin{equation}\label{eqphipqiqi1}
 \sum_{i=1}^N\Phi_p( |Q_i-Q_{i-1}|)
 \le c \int_{T_h^{1,r}} \mixedgrowth(\dist(\beta,\SO(3))) dx .
\end{equation}
By \eqref{eqlemmaVabG2} with $\delta=1/N$,
for $1\le i\le N$ we have
\begin{equation}
 \Phi_p( |Q_i-Q_{0}|)
 \le \left(1+\frac1N\right) 
 \Phi_p( |Q_{i-1}-Q_{0}|)
 + (1+N) \Phi_p( |Q_{i}-Q_{i-1}|),
\end{equation}
and iterating
\begin{equation}
 \Phi_p( |Q_i-Q_{0}|)
 \le \left(1+\frac1N\right)^i(1+N) 
\sum_{j=1}^i  \Phi_p( |Q_{j}-Q_{j-1}|).
\end{equation}
Therefore, using $i\le N$ and \eqref{eqphipqiqi1},
\begin{equation}\label{eqphipqiq0}
\Phi_p( |Q_i-Q_0|)\le c N   \int_{T_h^{1,r}} \mixedgrowth(\dist(\beta,\SO(3))) dx 
\end{equation}
for each $i\in\{0,\dots, N\}$ and, since $N\le 2h$, we conclude,
applying \eqref{eqlemmaVabG}, \eqref{eqphipqiq0} and
\eqref{eqphipbetaqi},
\begin{equation}\begin{split}
\int_{T_h^{1,r}} \mixedgrowth(|\beta-Q_0|) dx \le&
c\sum_{i=0}^N\left[\mixedgrowth(|Q_i-Q_0|)+
 \int_{T_1^{1,r}+z_ie_3} \mixedgrowth(|\beta-Q_i|) dx\right]
\\
\le & c h^2\int_{T_h^{1,r}} \mixedgrowth(\dist(\beta,\SO(3))) dx.
\end{split}\end{equation}

It remains to prove \eqref{rigidity-cylinder} in the case $R=h=1$, $Q'=\Id$. 
Assume first that  $r=2^{-N}$ for some integer  $N\ge1$. 
If  $N=1$ then, since  $\curl\beta=0$, using \cite[Theorem 1.1]{ContiDolzmannMueller2014} 
on two overlapping simply connected subsets of $T_1^{1,1/2}$ 
and then \eqref{eqlemmaVabG}
one obtains that
there is $Q\in \SO(3)$ such that
\begin{equation}\label{eq:cdm}
 \int_{T_1^{1,1/2}} \mixedgrowth(|\beta-Q|) dx \le 
 c\int_{T_1^{1,1/2}} \mixedgrowth(\dist(\beta,\SO(3))) dx.
\end{equation}
The argument is the same in the linear case,
using \cite[Theorem 2.1]{ContiDolzmannMueller2014}.
For the same reason, 
for any $\rho>0$ and any $\beta\in L^1(T_\rho^{\rho,\rho/4};\R^{3\times 3})$ with $\curl\beta=0$ there is $Q\in\SO(3)$ with
\begin{equation}\label{eq:cdm4}
 \int_{T_\rho^{\rho,\rho/4}} \mixedgrowth(|\beta-Q|) dx \le 
 c\int_{T_\rho^{\rho,\rho/4}} \mixedgrowth(\dist(\beta,\SO(3))) dx;
\end{equation}
by scaling the constant does not depend on $\rho$.

If instead $N\ge2$ we divide $T_1^{1,r}$ into dyadic sets $C_{i,j}$ of the form $
T_\rho^{\rho,\rho/4}$. More precisely we set 
\begin{equation*}C_{0,0}:=T_1^{1,1/4}\ \text{ and }\
 C_{i,j}:= T_{R_i}^{R_i,R_i/4}+jR_ie_3,
\end{equation*}
with $R_i:=2^{-i}$, $i=1,\dots, N-2$ and $j=0,\dots,2^i-1$.
We now apply \eqref{eq:cdm4} to each $C_{i,j}$ and find $Q_{i,j}\in\SO(3)$ such that
\begin{equation}\label{eq:rigidity-cdm}
 \int_{C_{i,j}} \mixedgrowth(|\beta-Q_{i,j}|) dx \le 
 c E_{i,j},\hskip1cm E_{i,j}:=\int_{C_{i,j}} \mixedgrowth(\dist(\beta,\SO(3))) dx,
\end{equation} 
with an universal constant.  Notice that $\calL^3(C_{i,j}\cap C_{i-1,k_j})\ge c  R_i^3$ for $i\ge1$ and
\begin{equation*}
 k_j:=\begin{cases}
  j/2,&\text{if $j$ is even},\\
  (j-1)/2,&\text{if $j$ is odd}
 \end{cases}
\end{equation*}
and $j=0,\dots,2^i-1$.
Then by \eqref{eqlemmaVabG}, arguing as in \eqref{eqphipqiqi1a}-\eqref{eqphipqiqi1b}
\begin{equation*}
 R_i^3 \mixedgrowth(|Q_{i,j}-Q_{i-1,k_j}|)\le 
 c(E_{i,j}+E_{i-1,k_j}),
\end{equation*}
so that from \eqref{eqlemmaVabG} we have
\begin{equation}\label{eq:diff-rotations}
 \begin{split}
R_i^3\mixedgrowth(|Q_{i,j}-Q_{0,0}|)\le 2&R_i^3\mixedgrowth(|Q_{i-1,k_j}-Q_{0,0}|)
+c (E_{i,j}+E_{i-1,k_j}).
 \end{split}
\end{equation}
For every $i$ let
\begin{equation}\label{eq:diff-rotations0}
 a_i:=\sum_{j =0}^{2^i-1}R_i^3 \mixedgrowth(|Q_{i,j}-Q_{0,0}|).
\end{equation}
Obviously $a_0=0$. Moreover by \eqref{eq:diff-rotations} 
\begin{equation}\begin{split}
 a_i\le &\sum_{j =0}^{2^i-1}
2R_i^3\mixedgrowth(|Q_{i-1,k_j}-Q_{0,0}|)
+c \sum_{j =0}^{2^i-1} (E_{i,j}+E_{i-1,k_j}).
\end{split}
 \end{equation}
In each sum, $k_j$ takes twice each value in $\{0,\dots, 2^{i-1}-1\}$. Recalling that $R_{i-1}=2R_i$, for $i\ge 1$ we obtain
 \begin{equation*}
 a_i\le \frac12 a_{i-1}
 + c(E_i+E_{i-1})
\end{equation*}
with $E_i:=\sum_{j=0}^{2^i-1}E_{i,j}$, for every $i\ge0$.
By induction we obtain
\begin{equation*}
 a_i \le c \sum_{k=1}^i \frac{1}{2^{i-k}} (E_k+E_{k-1}),
\end{equation*}
for every $i\ge1$.
We conclude that
\begin{equation}\label{eq:diff-rotations2}
 \sum_{i=0}^{N-2} a_i \le c \sum_{i=0}^{N-2}E_i.
\end{equation}
By \eqref{eqlemmaVabG},
\eqref{eq:diff-rotations0}, \eqref{eq:rigidity-cdm} and \eqref{eq:diff-rotations2} we obtain
\begin{equation*}
 \begin{split}
 \int_{T_{1}^{1,2^{-N}}}
\mixedgrowth(|\beta-Q_{0,0}|) dx
&\le \sum_{i=0}^{N-2}\sum_{j=0}^{2^i-1}\int_{C_{i,j}}\mixedgrowth(|\beta-Q_{0,0}|)dx
\\
&\le c \sum_{i=0}^{N-2} a_i + c\sum_{i=0}^{N-2}\sum_{j=0}^{2^i-1}
\int_{C_{i,j}}
\mixedgrowth(|\beta-Q_{i,j}|) dx
\le c E,
 \end{split}
\end{equation*}
with $E:=\sum_{i=0}^{N-2}E_i$. Again by scaling the above inequality holds with the same constant in $T_{2^{N}\rho}^{2^N\rho,\rho}$ for any $\rho>0$ and $N\in\N$.\\
Finally if $r\in(0,\frac12]$ is arbitrary we 
select $N\in\N$ such that $\frac12 <2^Nr\le 1$. Since $r\le\frac12$, we have $N\ge 1$, and therefore
$2^{-N}\le\frac12< 2^Nr$. This implies 
\begin{equation*}
(r,1)=(r, 2^Nr)\cup (2^{-N},1). 
\end{equation*}
For the same reason,
\begin{equation*}
 (0,1)=(0, 2^Nr) \cup (1-2^Nr,1).
\end{equation*}
Therefore
$$T_1^{1,r}=T_1^{1,2^{-N}}\cup T_{2^{N}r}^{2^Nr,r}
\cup \left[(1-2^Nr)e_3+ T_{2^{N}r}^{2^Nr,r}\right]=:C^1\cup C^2\cup C^3.
$$
We apply the rigidity estimate to each of these three cylinders separately 
and obtain rotations $Q^1$, $Q^2$, $Q^3$. As the overlap between them is not uniformly controlled we introduce a fourth cylinder. We let $\rho:=\frac{2^{-N}+r}2$ be the average between the two values of the inner radius and set
\begin{equation*}
 C^4:=(\frac12-2^{N-1}\rho)e_3+T^{2^N\rho,\rho}_{2^N\rho}=
 (B'_{2^N\rho}\setminus B_\rho')\times\left(\frac12-2^{N-1}\rho,
 \frac12+2^{N-1}\rho\right).
\end{equation*}
We then apply the rigidity estimate also to $C^4$, and obtain another rotation $Q^4$. 
From $\frac12<2^Nr\le 1$ we obtain $\frac34<2^N\rho\le1$
hence $2^N\rho-2^{-N}>\frac14$. This implies
 $\mathcal L^3(C^1\cap C^4)\ge c >0$.
Similarly, $2^Nr-\rho\ge\frac18$ and $2^Nr-(\frac12-2^{N-1}\rho)\ge \frac18$ show that 
$\mathcal L^3(C^i\cap C^4)\ge c >0$ for $i=2,3$. Therefore we can conclude
using the triangular inequality \eqref{eqlemmaVabG}.
\end{proof}

\begin{lemma}\label{lemmacompactnesscellpb}
Assume that $W$ satisfies \HWFinite\ or \HWLin\  for some $p\in 
(1,2]$. There is $c>0$ such that
for any $0<\eps\le r<R\le h$, with $2r\le R$,  $b\in\calB$, $t\in S^2$,  and
$\beta\in L^1(Q_t \Cyl^{R}_h;\R^{3\times 3})$ such that
  \begin{equation}\label{eq:curlbeta}
   \curl\beta = \eps b\otimes t \calH^1\LL (\R t\cap Q_t\Cyl^{R}_h)
  \end{equation}
we have 
  \begin{equation} 
   c \eps^2 h  |b|\ln\frac Rr \le \int_{Q_t \Cyl^{R,r}_h} W(\beta) dx .
\end{equation}
The constant depends only on $W$, $p$, and $\calB$.
\end{lemma}
\begin{proof}
We prove the statement for the case $h=R$. If $h> R$, 
it suffices to apply this bound on each of the $\lfloor h/R\rfloor$ disjoint subsets of $T_h^{R,r}$ which have the same shape as $T_R^{R,r}$, and then to use that $h/R\le 2\lfloor h/R\rfloor$. For $b=0$ there is nothing to prove. Recalling that $b\in\calB$, we see that we can assume $|b|\ge c>0$.

{Since $\curl\beta=0$ in $T_R^{R,r}$, by
\HWFinite\ and Lemma~\ref{lemmarigidcylinderhole}}
there is a rotation $Q\in\SO(3)$ such that
\begin{equation}
\int_{Q_t \Cyl^{R,r}_R} \mixedgrowth(|\beta-Q|)dx \le c \int_{Q_t \Cyl^{R,r}_R} 
W(\beta) dx .
\end{equation}
In the linear case the same holds with $S\in\R^{3\times 3}_\skw$.

Moreover by  condition \eqref{eq:curlbeta}, for every circle $C$ with radius $\rho\in(r,R)$, 
centered in a point of the segment $(0,Rt)$,
and contained in the plane orthogonal to $t$,
\begin{equation}
\eps|b|\le  \int_C |\beta-Q| d\calH^1 .
\end{equation}
By  monotonicity and Jensen's inequality, we then have
\begin{equation}
\mixedgrowth^{**}\left(\frac{\eps|b|}{2\pi\rho}\right)\le 
\mixedgrowth^{**}\left(\frac{1}{2\pi\rho}\int_C |\beta-Q| d\calH^1 \right)
\le \frac{1}{2\pi\rho}\int_C \mixedgrowth(|\beta-Q|) d\calH^1 
\end{equation}
where $\mixedgrowth^{**}$ is the convex envelope of $\mixedgrowth$.
Integrating over all circles $C$ by Fubini's theorem we
easily get
\begin{equation}\label{eqinteralcompacthepsi}
h \int_r^R 
\mixedgrowth^{**}\left(\frac{\eps|b|}{2\pi\rho}\right) \rho d\rho \le c \int_{Q_t \Cyl^{R,r}_h} 
W(\beta) dx .
\end{equation}

If $p=2$ then $\mixedgrowthnop_2^{**}(t)=\mixedgrowthnop_2(t)=|t|^2$ and integrating the left-hand side concludes the proof.

It remains to estimate the left-hand side of \eqref{eqinteralcompacthepsi} for $p<2$.
By \eqref{eqmixedgrowthconv} there are
$\eta$, $\eta'>0$ such that
$$\eta |t|^2\le \mixedgrowth^{**}(t)\quad\text{in}\quad [0,1/\pi]\qquad\text{and}\qquad \eta'|t|^p\le\mixedgrowth^{**}(t)\quad\text{in}\quad [1/(2\pi),\infty).$$
If $\eps|b|\le 2r $, we obtain
\begin{equation}\label{eq412}
\frac{1}{\eta} 
h \int_r^R 
\mixedgrowth^{**}\left(\frac{\eps|b|}{2\pi\rho}\right) \rho d\rho\ge
h \int_r^R \frac{\eps^2|b|^2}{(2\pi \rho)^2} \rho d\rho =
\frac{1}{(2\pi)^2}
h\eps^2|b|^2 \ln \frac{R}{r}.
\end{equation}
If $\eps|b|\ge R$, we have
\begin{equation}\label{eqsadfsdf}
\begin{split}
\frac{1}{\eta'} h \int_r^R 
\mixedgrowth^{**}\left(\frac{\eps|b|}{2\pi\rho}\right) \rho d\rho 
&\ge 
h \int_r^R \frac{\eps^p|b|^p}{(2\pi \rho)^p} \rho d\rho 
\\&=\frac{h\eps^p|b|^p}{(2\pi)^p(2-p)} 
(R^{2-p}-r^{2-p})\\
&\ge \frac{1-2^{p-2}}{(2\pi)^p(2-p)}
h\eps^p|b|^p R^{2-p}.
\end{split}
\end{equation} 
We observe that $s^{2-p}\ge (2-p)\ln s$ for all $s>0$.
Using that $\eps\le r<R$,
\begin{equation*}
\eps^2 \ln \frac Rr \le \eps^2 \ln \frac R\eps\le {\frac{1}{2-p}} \eps^p R^{2-p} \,.
\end{equation*}
Combining these two concludes the proof in this case.

Finally we consider the case $2r<\eps|b|< R$. We have 
\begin{equation}
\eta'h \int_{r}^{\e|b|}\frac{\eps^p|b|^p}{(2\pi \rho)^p} \rho d\rho +
\eta h \int_{\e|b|}^{R}\frac{\eps^2|b|^2}{(2\pi \rho)^2} \rho d\rho 
\le 
h \int_r^R 
\mixedgrowth^{**}\left(\frac{\eps|b|}{2\pi\rho}\right) \rho d\rho  \,.
\end{equation} 
Estimating the {two integrals} as in \eqref{eq412} and \eqref{eqsadfsdf} and  observing that for some $c_p>0$ it holds, since $|b|\ge \ln|b|$,
\begin{equation*}
\begin{split}
\frac{1-2^{p-2}}{(2\pi)^p(2-p)}
h \e^2|b|^2+
\frac{1}{(2\pi)^2}
h\e^2|b|^2\log\frac{R}{\e|b|}\ge& c_p h \e^2|b|\left(\log\frac{\e|b|}{\e}+\log\frac{R}{\e|b|}\right)\\
=c_p
h \e^2|b| \log \frac{R}{\e}\,,
\end{split}
\end{equation*}
with $\eps\le r$
we conclude.
\end{proof}

\subsection{Lower bound for a cylinder}
We present in Proposition~\ref{propcellproblem} the key estimate for the lower bound. We work in a unified framework able to treat
jointly the case of finite and infinitesimal kinematics and therefore
frame the problem via a generic energy density $V$.
In the rest of the paper we use the results of this section only via Corollary~\ref{corcellproblem}.

\newcommand\betalin{\xi}
\newcommand\Qroten{{\tilde Q}}
\begin{proposition}\label{propcellproblem}
{Assume that $W$ satisfies \HWFinite\ or 
\HWLin\ for some $p\in(1,2]$.
Define $V(F):=W(\Id+\Qroten F\Qroten^T)$ for some $\Qroten\in\SO(3)$ in the first case, and $V:=W$ in the second case.}
Let $\C^V:=D^2V(0)$,
$c_K>0$,
and $\ell_0>0$.
There 
is a nondecreasing function $\sigma:(0,\infty)\to(0,\infty)$, with 
$\lim_{s\to0}\sigma(s)=0$, depending only on $V$, $p$, $c_K$ and $\ell_0$ with the following property:
if  $0<2r\le R\le h$, $\eps>0$, $b\in\R^3$, $t\in S^2$, $\betalin\in L^1(Q_t \Cyl^{R}_h;\R^{3\times 3})$ obey
 $|b|\ge\ell_0$,
 \begin{equation}\label{eqproplbcylassck}
  \| \mixedgrowth(|\betalin|)\|_{L^1(Q_t\Cyl^{R,r}_h)}
  \le c_K{
  \eps^2 h{|b|^2}\ln \frac Rr},
 \end{equation}
and
\begin{equation}\label{eqcurlbetalinpropilb}
   \curl\betalin = \eps b\otimes t \calH^1\LL (\R t\cap Q_t\Cyl^{R}_h),
  \end{equation}
then necessarily
  \begin{equation} 
\frac{1}{\eps^2 h  \ln\frac Rr }
  \int_{Q_t \Cyl^{R,r}_h} V(\betalin) dx \ge \psi_{\C^V}(b,t) \left(1- \sigma\left(\frac{r}{R}+  \frac{\eps |b|}{r}\right)\right).
  \end{equation}
\end{proposition}

\begin{corollary}\label{corcellproblem}
Assume that $W$ satisfies \HWFinite\ for some $p\in (1,2]$.
Let $\C:=D^2W(\Id)$ {and let $\calB$ be a Bravais lattice.}
  There is a nondecreasing function $\tilde\sigma:(0,\infty)\to(0,\infty)$, with $\lim_{s\to0}\tilde\sigma(s)=0$, depending only on $W$, $\calB$ and $p$, with the following property:
  if
 $0<2r\le R\le h$, $\eps>0$, $b\in\calB$, $t\in S^2$, $\beta\in L^1(Q_t \Cyl^{R,r}_h;\R^{3\times 3})$ such that 
 \begin{equation}
  (\mu,\beta)\quad\text{ is $r$-compatible with} \quad\mu:=\eps b\otimes t \calH^1\LL (\R t\cap Q_t\Cyl^{R}_h)
 \end{equation}
 (in the sense of Definition~\ref{defbetamucompatible}, for the domain $Q_tT^R_h$)
{and $Q\in \SO(3)$ then}
  \begin{equation}
   \begin{split}\label{eq:cell-problem-nonlinear}
&\frac{1}{\eps^2 h  \ln\frac Rr }
   \int_{Q_t \Cyl^{R,r}_h} W(\beta) dx \\
   &\ge  \psiC(b,Qt) \left(1- \tilde\sigma\left(
   {\frac{r}{R}}+ \frac{\|\beta-Q\|_{L^1(Q_t\Cyl^{R,r}_h)}}{R^3}+{\frac{\eps |b|}{r}}\right)\right).
   \end{split}
  \end{equation}
Assume that $W$ satisfies \HWLin\ for some $p\in(1,2]$ and let $\C:=D^2W(0)$. Then the same holds by replacing  \eqref{eq:cell-problem-nonlinear} with
  \begin{equation}
 \begin{split}\label{eq:cell-problem-linear}
  &\frac{1}{\eps^2 h  \ln\frac Rr }
  \int_{Q_t \Cyl^{R,r}_h} W(\beta) dx 
 \ge  \psiC(b,t) \left(1- {\tilde\sigma\left(
  {\frac{r}{R}}+ 
  {\frac{\eps |b|}{r}}\right)}\right).
 \end{split}
\end{equation}
\end{corollary}
\begin{proof}
Assume \HWFinite\ holds.  
Without loss of generality, $R\le h\le 2R$. Otherwise we decompose the cylinder 
in 
$\lfloor h/R\rfloor$ pieces with height $h'\in [R,2R)$, apply the result to 
each of them and sum.
By $r$-compatibility, there are an extension $\tilde\mu\in \calM^1(\R^3)$ of $\mu$ and $\beta_0\in L^{1}(Q_t T^{R}_h;\R^{3\times 3})$ with
$\curl\beta_0=0$ in $Q_t T^{R}_h$ and
$\beta=\beta^{\tilde\mu}+\beta_0$ in $Q_t T^{R,r}_h$.
We use $\beta^{\tilde\mu}+\beta_0$ to define an extension
$\beta\in L^1(Q_t T^{R}_h;\R^{3\times3})$ such that
\begin{equation}\label{eqcurlbetacorr}
 \curl\beta = \eps b\otimes t \calH^1\LL (\R t\cap Q_t\Cyl^{R}_h).
\end{equation}
By Lemma~\ref{lemmacellproblemlb}
there is $c_u>0$ such that
$\psiC(b,t)\le c_u |b|^2$ for all $b$ and $t$. If 
$\int_{Q_t \Cyl^{R,r}_h} W(\beta) dx\ge c_u \eps^2 h |b|^2 \ln\frac Rr$, we are done. Therefore we can assume the converse inequality holds.
 By the rigidity estimate in Lemma~\ref{lemmarigidcylinderhole} and \eqref{eqlowerboundWfinite}, there is a rotation $Q_*\in\SO(3)$ such that
\begin{equation}\label{eqrigiditycor}
 \int_{Q_t \Cyl^{R,r}_h} \mixedgrowth(|\beta-Q_*|)dx \le c \int_{Q_t \Cyl^{R,r}_h} 
W(\beta) dx 
 < cc_u \eps^2 h |b|^2 \ln\frac Rr .
\end{equation}
We define
\begin{equation}
\betalin:=Q_*^T\beta-\Id\quad
\text{ and }\quad
V(F):=
W(\Id+Q_*F Q_*^T) \text{ for } F\in\R^{3\times 3},
\end{equation}
so that
$V(\xi)=
W(\beta Q_*^T)=W(\beta)$
and  $D^2V(0)=\C_{Q^*}$, which was
defined from $\C=D^2W(\Id)$ in \eqref{eqdefCQ}.
We observe that 
 \eqref{eqcurlbetacorr}
gives
 \begin{equation}
\curl\betalin=Q_*^T\curl\beta=
\eps Q_*^Tb\otimes t \calH^1\LL (\R t\cap Q_t\Cyl^{R}_h),
 \end{equation}
 which is \eqref{eqcurlbetalinpropilb} with $Q_*^Tb$ in place of $b$.
Equation
\eqref{eqrigiditycor} implies
\eqref{eqproplbcylassck} with $c_K:=cc_u$, which depends only on $\C$.
By Proposition~\ref{propcellproblem} with $\ell_0:=\min\{|b|: b\in\calB\setminus\{0\}\}$ and the above estimates we obtain
  \begin{equation} \label{eqlbpsiqtb}
\frac{1}{\eps^2 h  \ln\frac Rr }
  \int_{Q_t \Cyl^{R,r}_h} W(\beta) dx \ge \psi_{\C_{Q^*}}(Q_*^Tb,t) \left(1- \sigma\left(
  \frac{r}{R}+ \frac{\eps |b|}{r}\right)\right).
  \end{equation}
  By Lemma~\ref{lemmarotatebetacurl}\ref{lemmarotatebetacurlpsic},
  $\psi_{\C_{Q^*}}(Q_*^Tb,t)=\psiC(b,Q_*t)$.
By  Lemma~\ref{lemmacellproblemlb}\ref{lemmacellproblemlbtcont} there is $c>0$, depending only on $\C$, such that
\begin{equation}
\psiC(b,Qt) \le 
 (1+c|Q-Q_*|)  \psiC(b,Q_*t)
\end{equation}
which implies
\begin{equation}
 (1-c|Q-Q_*|) \psiC(b,Qt) \le \psiC(b,Q_*t) .
\end{equation}
Therefore
\begin{equation} \label{lowerboundW}
\begin{split}
&\frac{1}{\eps^2 h  \ln\frac Rr }
   \int_{Q_t \Cyl^{R,r}_h} W(\beta) dx \\
 &  \ge  \psiC(b,Qt) 
   \left(1- \sigma\left( {\frac{r}{R}+} \frac{\eps |b|}{r}\right)\right)(1-c|Q-Q_*|).
   \end{split}
  \end{equation}
  From \eqref{eqrigiditycor}, treating separately the part with $|\beta-Q_*|<1$ and the one with $|\beta-Q_*|\ge 1$, we have
  \begin{equation}\begin{split}
\frac\pi2 R^2h&|Q-Q_*|\le \|\beta-Q_*\|_{L^1(Q_t\Cyl^{R,r}_h)}+\|\beta-Q\|_{L^1(Q_t\Cyl^{R,r}_h)} \\
   &\le C\Bigl(\eps h R |b| \ln^{1/2} \frac Rr + {\eps^{2/p}R^{2/p'} h  |b|^{2/p}\ln^{1/p}\frac Rr}\Bigr)+\|\beta-Q\|_{L^1(Q_t\Cyl^{R,r}_h)},
               \end{split}
  \end{equation}
where $p'$ is the conjugate of $p$ (i.e., $p'=\frac{p}{p-1}$).
  Therefore, recalling that $h\in [R,2R]$
  {and that $\log t\le t^2$  for all $t>0$},
  \begin{equation}\label{Q-Q*}\begin{split}
   |Q-Q_*|&\le  C\Bigl(\frac{\eps  |b|}{R} \ln^{1/2} \frac Rr + \frac{\eps^{2/p}}{R^{2/p}}  
   |b|^{2/p}\ln^{1/p}\frac Rr+\frac{\|\beta-Q\|_{L^1(Q_t\Cyl^{R,r}_h)}}{R^3}\Bigr)\\
   &\le  C\Bigl(\frac{\eps  |b|}{r}+ \frac{\eps^{2/p}}{r^{2/p}}  
   |b|^{2/p}+\frac{\|\beta-Q\|_{L^1(Q_t\Cyl^{R,r}_h)}}{R^3}\Bigr).
               \end{split}
  \end{equation}
Plugging \eqref{Q-Q*} into \eqref{lowerboundW} we easily obtain \eqref{eq:cell-problem-nonlinear}. 

{The proof of \eqref{eq:cell-problem-linear} is similar but simpler. Indeed, one obtains from Lemma~\ref{lemmarigidcylinderhole} a matrix $S_*\in\R^{3\times 3}_\skw$
with the property corresponding to \eqref{eqrigiditycor}. One then sets 
$\xi:=\beta-S_*$, which obeys $\curl\xi=\curl\beta$. Therefore
 Proposition~\ref{propcellproblem}   gives,
 using $V(\xi)=W(\xi)=W(\beta)$,
 \eqref{eqlbpsiqtb}
with $\psiC(b,t)$ instead of $\psi_{\C_{Q^*}}(Q_*^Tb,t)$, and 
\eqref{eq:cell-problem-linear} follows.}
\end{proof}

\newcommand{\Deltaj}{\Delta_j}
\begin{proof}[Proof of Proposition~\ref{propcellproblem}]
We first show that there is $c\ge1$ such that
\begin{equation}\label{lemmaVab}
 V(a+b)\le c V(a)+c\mixedgrowth(|b|) \text{ for all $a,b\in\R^{3\times3}$}.
\end{equation}
To prove this, assume first \HWFinite, and to shorten notation let $a':=\Qroten a\Qroten^T$, $b':=\Qroten b\Qroten^T$.
From
\eqref{eqlowerboundWfinite}
we obtain 
\begin{equation}\begin{split}              
 V(a+b)&=W(\Id+a'+b')\le c \mixedgrowth(\dist(\Id+a'+b',\SO(3))) .
             \end{split}
\end{equation}
Using $\dist(\Id+a'+b',\SO(3))\le
\dist(\Id+a',\SO(3))+|b'|$, monotonicity of $\mixedgrowth$,
\eqref{eqlemmaVabG} and then \eqref{eqlowerboundWfinite}
 again we have
\begin{equation}\begin{split}
 V(a+b)&\le c 
 \mixedgrowth(\dist(\Id+a',\SO(3))+|b|)
 \\
 &\le 2c \mixedgrowth(\dist(\Id+a',\SO(3)))+
 2c \mixedgrowth(|b|)\\
&\le 2c^2 W(\Id+a')+2c \mixedgrowth(|b|)
 = 2c^2V(a) +2c\mixedgrowth(|b|),
  \end{split}
\end{equation}
and then \eqref{lemmaVab} follows. 

If \HWLin\ holds, one similarly computes using \eqref{eqlowerboundWlin} and \eqref{eqlemmaVabG} 
\begin{equation}
\begin{split}
 V(a+b)&=W(a+b)\le c \mixedgrowth(|a+b|)
 \\
& \le 2c \mixedgrowth(|a|)+ 2c \mixedgrowth(|b|)
\le 2c^2V(a)+2c \mixedgrowth(|b|).
\end{split}
\end{equation}
{This concludes the proof of \eqref{lemmaVab}.}

We reason by contradiction and assume that there 
exist $\rho>0$ 
and  sequences $r_j$, $R_j$, $h_j$,
 $\eps_j$, $b_j$, $t_j$, $\xi_j$ as in the statement 
 {such that, writing $Q_j:=Q_{t_j}$},
\begin{equation}
\frac{1}{\eps_j^2 h_j  \ln\frac {R_j}{r_j}} 
\int_{{Q_j} \Cyl^{R_j,r_j}_{h_j}} V(\xi_j) dx\le
 (1-2\rho)
 \psi_{\C^V}({b_j},t_j),
\end{equation}
\begin{equation}\label{eqassxicK}
 {\|\mixedgrowth(|\xi_j|)\|_{L^1({Q_j}\Cyl^{R_j,r_j}_{h_j}
)}}
 \le c_K {{\eps_j^2 h_j  |b_j|^2}\ln\frac{ R_j}{r_j}},
\end{equation}
 and
 \begin{equation}\label{eqparkinfty0}
\Deltaj:=\frac{r_j}{R_j}+ \frac{\eps_j |b_j|}{r_j}\to0.
 \end{equation}
  Since $\psi_{\C^V}$ is {positively} two-homogeneous in the first argument (see 
\eqref{eqvarprobbetzabt} and \eqref{eqbetamuline2}, or
\cite[Lemma~5.1(iii)]{ContiGarroniOrtiz2015}),
\begin{equation}
\frac{1}
{\eps_j^2 h_j |b_j|^2 \ln\frac {R_j}{r_j}} 
\int_{Q_{j }\Cyl^{R_j,r_j}_{h_j}} V(\xi_j) dx\le
 (1-2\rho)
 \psi_{\C^V}\left(\frac{b_j}{|b_j|},t_j\right).
\end{equation}
After passing to a subsequence, we can assume that $\frac{b_j}{|b_j|}\to b_*$ and $Q_j\to Q_*$, for some $b_*\in S^2$,
$Q_*\in\SO(3)$; this implies $t_j=Q_je_3\to t_*:=Q_*e_3$.
By continuity of $\psi_{\C^V}$
{(which follows from
\eqref{eqvarprobbetzabt} and \eqref{eqbetamuline2}),} 
\begin{equation}\label{eqcontradictcellpb3dtmp}
\limsup_{j\to\infty} \frac{1}
{\eps_j^2 h_j |b_j|^2 \ln\frac {R_j}{r_j}} 
\int_{{Q_j} \Cyl^{R_j,r_j}_{h_j}} V(\xi_j) \le
 (1-2\rho) \psi_{\C^V}(b_*, t_*).
\end{equation}
Choose now $\lambda>0$ such that $\lambda c_K\le \rho
 \psi_{\C^V}(b_*, t_*)$.
By \eqref{eqcontradictcellpb3dtmp} and \eqref{eqassxicK},
\begin{equation}\label{eqcontradictcellpb3d}
\limsup_{j\to\infty} \frac{1}
{\eps_j^2 h_j |b_j|^2 \ln\frac {R_j}{r_j}} 
\int_{{Q_j} \Cyl^{R_j,r_j}_{h_j}} V(\xi_j) +\lambda \mixedgrowth(|\xi_j|)dx\le
 (1-\rho) \psi_{\C^V}(b_*, t_*).
\end{equation}
Let $\delta\in(0,\frac12]$ be fixed. We observe that the sets
\begin{equation}
 C^{i,k}_j:=
 {Q_j}((B'_{\delta^k R_j}\setminus B'_{\delta^{k+1}R_j}) \times (i\delta^k R_j, (i+1)\delta^k R_j))
\end{equation}
for 
$k\in\N\cap[1, (1-\delta)\frac{\ln R_j/r_j}{\ln 1/\delta}-1]$, 
$i\in \N\cap[0, \frac{h_j}{\delta^k R_j}-1]$
are disjoint and contained in ${Q_j} \Cyl^{R_j,r_j
}_{h_j}$.
In particular for any such $k$ we have
\begin{equation}\label{eqdeltaRkj}
\delta^{k+1}R_j\ge R_j^{\delta}r_j^{1-\delta}. 
\end{equation}
We choose $k_j$ (depending on $j$) such that
\begin{equation}
\begin{split}
 {\left\lfloor (1-\delta)\frac{\ln R_j/r_j}{\ln 1/\delta}-1\right\rfloor}  
\int_{\cup_i C^{i,k_j}_j} V(\xi_j) +\lambda \mixedgrowth(|\xi_j|) dx \\
\le \int_{{Q_j} 
\Cyl^{R_j,r_j}_{h_j}} V(\xi_j)  +\lambda \mixedgrowth(|\xi_j|)dx
\end{split}
\end{equation}
and then $i_j$ (depending on $j$) with
\begin{equation}\label{C^i,k}
\begin{split}
\left\lfloor (1-\delta)\frac{\ln R_j/r_j}{\ln 1/\delta}-1\right\rfloor&
 \left\lfloor \frac{h_j}{\delta^{k_j} R_j} \right\rfloor 
  \int_{C^{i_j,k_j}_j} V(\xi_j) +\lambda \mixedgrowth(|\xi_j|)dx \\
 &\le  \int_{{Q_j} 
\Cyl^{R_j,r_j}_{h_j}} V(\xi_j)+\lambda \mixedgrowth(|\xi_j|)dx.
\end{split}
\end{equation}
Inserting \eqref{C^i,k} in \eqref{eqcontradictcellpb3d} and taking the limit, using $R_j/r_j\to\infty$ and that $h_j\ge R_j$ and $k_j\ge1$ imply
$\lfloor \frac{h_j}{\delta^{k_j} R_j} \rfloor\ge 
\frac{h_j}{\delta^{k_j} R_j} (1-\delta)$, 
we obtain
\begin{equation}\label{eqestimcyl1dei}
 \limsup_{j\to\infty} 
 \frac{(1-\delta)^2}{\eps_j^2|b_j|^2 \delta^{k_j}R_j\ln\frac 1\delta}
 \int_{C^{i_j,k_j}_j} V(\xi_j) +\lambda \mixedgrowth(|\xi_j|) dx \le  (1-\rho) \psi_{\C^V}(b_*, 
t_*).
\end{equation}
We now rescale to a fixed {tube} $\Cyl_1^{1,\delta}$.
As in the entire
proof, the maps are defined and integrable over the entire cylinder $T^1_1$, but the estimates are only on the restriction to the tube  $T^{1,\delta}_1$.
Precisely, we define 
$\bar\xi_j\in L^1(\Cyl_1^{1};\R^{3\times 3})$ by
\begin{equation}\label{eqdefbarxi}
 \bar\xi_j(x):={Q_j^T}\xi_j(\delta^{k_j}R_j {Q_j}x+  i_j \delta^{k_j}R_jt_j){Q_j}
\end{equation}
and observe that it obeys (see Lemma~\ref{lemmarotatebetacurl}\ref{lemmarotatebetacurlbetameas}, and recall that $Q_j^Tt_j=e_3$)
\begin{equation}\label{eqcurlbarxij}
 \curl\bar\xi_j=\frac{\eps_j{Q_j^T}b_j}{\delta^{k_j}R_j} \otimes e_3 \calH^1\LL 
(0,1)e_3.
\end{equation}
Using \eqref{eqestimcyl1dei}, with a change of variables we obtain
\begin{equation}\label{eqs23434}
 \limsup_{j\to\infty} 
 \frac{ (\delta^{k_j}R_j)^2(1-\delta)^2}{\eps_j^2|b_j|^2 \ln\frac 1\delta}
 \int_{\Cyl_1^{1,\delta}} V({{Q_j}}\bar\xi_j {Q_j}^T) dx \le  (1-\rho) \psi_{\C^V}(b_*, t_*)
\end{equation}
and, for sufficiently large $j$,
\begin{equation}\label{eqs23434b}
 \frac{ (\delta^{k_j}R_j)^2 }{\eps_j^2|b_j|^2 }
 \int_{\Cyl_1^{1,\delta}} \mixedgrowth(|\bar\xi_j|)dx \le   \frac2\lambda\psi_{\C^V}(b_*, t_*) \ln \frac1\delta.
\end{equation}
In the rest of the argument we shall use \eqref{eqcurlbarxij}, \eqref{eqs23434}, 
\eqref{eqs23434b} and the definition of $\psi_{\C^V}$ to reach a contradiction.

By Remark~\ref{rem-phip}(ii)
and \eqref{eqs23434b} there are 
$A_j,B_j:\Cyl_1^{1,\delta}\to\R^{3\times 3}$ such that 
       \begin{equation}\label{eqxijrigab}
        \bar\xi_j=
    A_j+B_j,
       \end{equation}
$|A_j|\le 1$ everywhere, 
$|B_j|\ge 1$ on the set $\{B_j\ne0\}$,
and for sufficiently large $j$
\begin{equation}\label{eqAjBj2p}
 \frac{ (\delta^{k_j}R_j)^2}{\eps_j^2|b_j|^2 }
 \int_{\Cyl_1^{1,\delta}}
( |A_j|^2+|B_j|^p )dx \le c C \ln \frac1\delta.
\end{equation}
Recalling that $\eps_j|b_j|/r_j\to0$ and that $r_j/R_j\to0$, we compute
with \eqref{eqdeltaRkj}
\begin{equation}\label{eqestconefsdd}
   \frac{ \delta^{k_j}R_j }{\eps_j {|b_j|}}\ge
 \frac{r_j}{\eps_j{|b_j|}} (R_j/r_j)^{\delta} \to\infty.
\end{equation}
From \eqref{eqAjBj2p} we obtain $A_j\to0$ in $L^2(T_1^{1,\delta};\R^{3\times 3})$ and $B_j\to0$ in $L^p(T_1^{1,\delta};\R^{3\times 3})$. 
We define  $\hat A_j, \hat B_j\in L^1({T_1^{1,\delta}};\R^{3\times 3})$ by 
\begin{equation}\label{eqrescalehataa}
\hat A_j:=\frac{\delta^{k_j}R_j}{\eps_j|b_j| } A_j\hskip5mm\text{ and 
}\hskip5mm
\hat B_j:=\frac{\delta^{k_j}R_j}{\eps_j|b_j| }B_j.
   \end{equation}
With \eqref{eqAjBj2p} and
\eqref{eqestconefsdd}   
    we obtain that 
    \begin{equation}
     \limsup_{j\to\infty} \|\hat B_j\|_{L^p(T_1^{1,\delta})} +
     \limsup_{j\to\infty} \|\hat A_j\|_{L^2(T_1^{1,\delta})} <\infty.
    \end{equation}
Passing to a further subsequence, $\hat A_j\weakto \hat A_*$ weakly in $L^2(T_1^{1,\delta};\R^{3\times 3})$. 
Let $\eta\in(0,1)$ be fixed and set
\begin{equation}
E_j^\eta:=\{x\in T^{1,\delta}_1: 
(|A_j|+|B_j|)(x)\le\eta\}.
\end{equation}
{We recall that} $A_j\to0$ in $L^2$ and $B_j\to0$ in $L^p$. In particular, they converge  to zero in measure, and therefore $\calL^3(T^{1,\delta}_1\setminus E^\eta_j)\to0$.
On $E^\eta_j$ we have $B_j=0$
 and in particular $\hat B_j
 =\hat B_j \chi_{T_1^{1,\delta}\setminus E^\eta_j} 
 \weakto0$ in $L^p(T_1^{1,\delta};\R^{3\times 3})$.
By the differentiability of  $V$ in a neighbourhood of the origin, if $\eta$ is sufficiently small we have
\begin{equation}\label{eqlbtaylor}
\begin{split}
V(F)\ge 
{V(0)+DV(0)\cdot F+}
\frac12 \C^V F \cdot F - \omega({\eta}) |F|^2 \\
\text{ for all } 
F\in\R^{3\times 3} \text{ with } |F|\le \eta,
\end{split}
\end{equation}
where $\C^V:=D^2V(0)$ and
$\omega:[0,\infty)\to[0,\infty)$ is monotone, continuous, with $\omega(0)=0$.
We shall use this estimate with $F={Q_j}A_jQ_j^T$ on $E^\eta_j$.
As $F=0$ is a minimizer of $V$, $V(0)=0$ and $DV(0)=0$, so that
\begin{equation}
\begin{split}
 V({Q_j}\bar\xi_jQ_j^T)\ge &
  V({Q_j}A_jQ_j^T)\chi_{E^\eta_j} \\
 \ge& \frac12 \C^V
({Q_j}A_jQ_j^T)\cdot ({Q_j}A_jQ_j^T)\chi_{E^\eta_j} -\omega(\eta) 
|A_j|^2
\end{split}
\end{equation}
which is the same as
\begin{equation}
 \frac{(\delta^{k_j}R_j)^2}{\eps_j^2|b_j|^2 } 
 V({Q_j}\bar\xi_jQ_j^T)\ge \frac12 \C^V  ({Q_j}\hat A_jQ_j^T) \cdot ({Q_j}\hat A_jQ_j^T)\chi_{E^\eta_j} 
-\omega(\eta) |\hat A_j|^2.
\end{equation}
We recall that $\calL^3(T^{1,\delta}_1\setminus E^\eta_j)\to0$, which together with
$  \hat A_j\weakto \hat A_*$ and $Q_j\to Q_*$ implies
\begin{equation}
 {Q_j} 
 \hat A_jQ_j^T\chi_{E^\eta_j} \weakto {Q_*}
\hat  A_*Q_*^T\text{ weakly in $L^2(T_1^{1,\delta};\R^{3\times 3})$}.
\end{equation}
As $\hat A_j$ is bounded in $L^2(T_1^{1,\delta};\R^{3\times 3})$
\begin{equation}
\liminf_{j\to\infty} \int_{T^{1,\delta}_1}
\frac{(\delta^{k_j}R_j)^2}{\eps_j^2|b_j|^2 } 
V( {Q_j}\bar\xi_jQ_j^T)dx \ge \int_{T^{1,\delta}_1}\frac12 \C^V  ( {Q_*}\hat A_*Q_*^T)\cdot ( {Q_*}\hat 
A_*Q_*^T)dx -c \omega(\eta)
\end{equation}
for every $\eta\in (0,1)$ sufficiently small and therefore for $\eta=0$.
Hence recalling \eqref{eqs23434}, and 
{with $\C^V_{Q^*}$  defined from $\C^V$ as} 
in \eqref{eqdefCQ}, 
for any $\delta\in(0,\frac12]$ we have
\begin{equation}\label{eqqdeltasdf}
\frac{(1-\delta)^2}{\ln \frac1\delta}
\int_{T^{1,\delta}_1}\frac12 \C^V_{Q_*} \hat A_*\cdot \hat A_*dx\le (1-\rho) 
\psi_{\C^V}(b_*, t_*).
\end{equation}
To conclude it remains to relate the left-hand side to $\psiC$. We shall show  below that there is $u_*\in W^{1,1}(T_1^{1,\delta};\R^3)$ such that
\begin{equation}\label{eqhatadust}
 \hat A_*=Du_*+\beta_{Q_*^Tb_*,e_3},
\end{equation}
 with $\beta_{b,t}$ defined as in
Lemma~\ref{lemmabetabtkernel1}. We recall the definition of 
$\infcyl$ in \eqref{eqdefabthRr},  and see 
that $Du_*+\beta_{Q_*^Tb_*,e_3}$ is an admissible test field, so that  \eqref{eqqdeltasdf} implies
\begin{equation}
{(1-\delta)^2}
\infcyl(\C^V_{Q_*},Q_*^Tb_*,e_3,1,1,\delta)\le (1-\rho) 
\psi_{\C^V}(b_*, t_*).
\end{equation}
Taking the limit $\delta\to0$ with Lemma~\ref{lemmacellproblemlb}, 
we obtain
\begin{equation}
\psi_{\C^V_{Q_*}}(Q_*^Tb_*,e_3)\le (1-\rho) 
\psi_{\C^V}(b_*, t_*).
\end{equation}
Recalling that by 
Lemma~\ref{lemmarotatebetacurl}\ref{lemmarotatebetacurlpsic}
one has
$\psi_{\C^V_{Q_*}}(Q_*^Tb_*,Q_*^Tt_*)=
\psi_{\C^V}(b_*,t_*)$
and that
$Q_*e_3=t_*$ gives the desired
contradiction.

It remains to prove \eqref{eqhatadust}. We define
$\bar\beta_j\in L^1(T_1^1;\R^{3\times 3})$ as the strain field associated to the Burgers vector $Q_j^Tb_j/|b_j|$ appearing (after scaling) in \eqref{eqcurlbarxij},
\begin{equation}
\bar\beta_j :=\beta_{Q_j^Tb_j/|b_j|, e_3}.
\end{equation}
By \eqref{eqcurlbarxij}
we have $\curl(\frac{\delta^{k_j}R_j}{\eps_j|b_j|}\bar \xi_j-  {\bar\beta}_j)=0$ in $T_1^1$ 
and there is $v_j\in W^{1,1}(T_1^1;\R^3)$ such that $
\frac{\delta^{k_j}R_j}{\eps_j|b_j|}\bar \xi_j ={\bar\beta_j}+Dv_j$.
Inserting \eqref{eqxijrigab} and \eqref{eqrescalehataa} we obtain
\begin{equation}
 \hat A_j+\hat B_j = {\bar\beta_j} + Dv_j \hskip5mm\text{ in } T_1^{1,\delta}.
\end{equation}
We recall that $\hat A_j\weakto A_*$ in $L^2(T_1^{1,\delta};\R^{3\times 3})$ and
$\hat B_j\weakto 0$ in $L^p(T_1^{1,\delta};\R^{3\times 3})$. Further,
$b_j/|b_j|\to b_*$, $Q_j\to Q_*$, $t_j\to t_*$, and $\beta_{b,t}$ depends continuously on $b$ and $t$ away from the singularity, which implies that
$\bar\beta_j$ converges uniformly in $T^{1,\delta}_1$ to
\begin{equation}
\bar\beta_*:=\beta_{Q_*^Tb_*, e_3}.
\end{equation}
Therefore the sequence $v_j$ (possibly after adding a constant to each of them) is bounded in $W^{1,p}(T_1^{1,\delta};\R^3)$ and
there is $ v_*\in W^{1,1}(T_1^{1,\delta};\R^3)$ with 
\begin{equation}
 \hat A_*= {\bar\beta_*} + D v_* \hskip5mm\text{ in } T_1^{1,\delta}.
\end{equation}
This
concludes the proof of \eqref{eqhatadust} and of the proposition.
\end{proof}

The cell problems from \cite{ContiGarroniOrtiz2015,garroni2020nonlinear} are all immediate consequences  
of
Proposition~\ref{propcellproblem}, via Corollary~\ref{corcellproblem}.
As an easy consequence one can also prove a lower bound when the dislocation is straight.
\begin{proposition}
Assume that $W$ satisfies \HWFinite\ for some $p\in(1,2]$.
Let $\C:=D^2W(\Id)$.
For any $0<R\le h$, $b\in\calB$, $t\in S^2$, let 
  $(\mu_\eps,\beta_\eps)\in \calM^1(Q_t\Cyl^{R}_h)\times L^1(Q_t \Cyl^{R}_h;\R^{3\times 3})$ be $\rho_\eps$-compatible, with $\rho_\eps\ge\eps$, 
  in the sense of Definition \ref{defbetamucompatible}, 
with $\mu_\eps:=\eps b\otimes t \calH^1\LL (\R t\cap Q_t\Cyl^{R}_h)$.
 Assume that
 \begin{equation}
  \lim_{\eps\to0}\frac{\log\eps}{\log\rho_\eps}=1
 \end{equation}
and
 \begin{equation}\label{eq:diff-rotation}
 \lim_{\eps\to0}\|\beta_\eps-Q\|_{L^1(Q_tT^{R,\rho_\eps}_h)}=0,
 \end{equation}
for some $Q\in \SO(3)$. Then 
 \begin{equation}\label{eq:asymptotic-nonlinear-cell}
  \psiC(b,Qt)\le \liminf_{\eps\to0} \frac{1}{\eps^2 h \ln\frac R{\eps}}
  \int_{Q_tT^{R,\rho_\eps}_h} W(\beta_\eps) dx .
 \end{equation}
Assume that $W$ satisfies \HWLin\ for some $p\in(1,2]$.
Let $\C:=D^2W(0)$. Then the same holds by replacing 
\eqref{eq:asymptotic-nonlinear-cell} with
 \begin{equation}\label{eq:asymptotic-linear-cell}
 \psiC(b,t)\le \liminf_{\eps\to0} \frac{1}{\eps^2 h \ln \frac R\e}
 \int_{Q_tT^{R,\rho_\eps}_h} W(\beta_\eps) dx .
\end{equation}
\end{proposition}

\begin{proof}
The proof follows by applying Corollary \ref{corcellproblem} with $r=r_\eps:=\rho_\eps\ln\frac1{\rho_\eps}$.
 \end{proof}

\section{Compactness and lower bound}\label{seccomp}
In this section we provide the proofs of  the compactness and the lower bound needed to prove Theorem \ref{thm:Gamma-limit}.
The proof is similar to the one of \cite[Prop.~6.6]{ContiGarroniOrtiz2015} using Proposition~\ref{propcellproblem} {(via Corollary \ref{corcellproblem})}  instead of {\cite[Lemma~5.8]{ContiGarroniOrtiz2015}}.
We discuss the various arguments separately in order to be able to reuse them.

\subsection{Rigidity}\label{sectionextensionrigidity}

We first recall a rigidity result from \cite{ContiGarroniRigidity}, which is based {on the Friesecke-James-Müller geometric rigidity \cite{FrieseckeJamesMueller2002}
and on the Bourgain-Brezis critical integrability bound
 \cite{BourgainBrezis2007},}
{and extend it to smaller exponents.}
\begin{proposition}[From \cite{ContiGarroniRigidity}]\label{proprigiditycras}
 Let $\Omega\subset\R^3$ be a bounded connected Lipschitz set,
 $q\in[1,\frac32]$. Then there is a constant $c=c(\Omega,q)$ such that for any $\contbeta\in L^q(\Omega;\R^{3\times 3})$
 such that $\curl\beta$ is a finite measure
 there are $Q\in \SO(3)$ such that
 \begin{equation}
  \|\contbeta-Q\|_{L^{q}(\Omega)} \le c \|\dist(\contbeta,\SO(3))\|_{L^{q}(\Omega)} +c|\curl\contbeta|(\Omega)
 \end{equation}
 and ${S}\in \R^{3\times 3}_\skw$ such that
\begin{equation}
  \|\contbeta-{S}\|_{L^{q}(\Omega)} \le c \|\contbeta+\contbeta^T\|_{L^{q}(\Omega)} +c|\curl\contbeta|(\Omega).
 \end{equation}
 \end{proposition}
 \begin{proof}
 The important case is the critical exponent
$q=\frac32$, which was proven in \cite{ContiGarroniRigidity}. The 
 case $q\in [1,\frac32)$ follows from the same argument, using $s=q$ in the proof of Theorem~1 
 in \cite{ContiGarroniRigidity}, and using Hölder's inequality in (8). Then after scaling Eq.~(11) becomes
 \begin{equation}
  \|\beta-R\|_{{L^q}(Q_r)} \le c \|\dist(\beta,\SO(3))\|_{{L^q}(Q_r)} + 
  (\calL^3(Q_r))^{\frac1q-\frac23}
  |\curl\beta|(Q_r)),
 \end{equation}
the factor $(\calL^3(Q_r))^{\frac1q-\frac23}=r^{\frac3q-2}$ is uniformly bounded since $Q_r\subseteq \Omega$.  We remark that the second term in \cite[Eq.~(11)]{ContiGarroniRigidity} incorrectly contains $Du$ instead of $\beta$.
The rest of the argument does not make use of the specific value of $s$ and is unchanged.
{The modification in the linear case is the same.}
\end{proof}

  \subsection{Compactness}

We recall the definition of  the class $\calA^*_\eps$ of admissible pairs  in \eqref{Astar} 
and {that of} the elastic energy
$\Elast[\beta,A]$ in \eqref{eqdefFepsbetaA} for fields $\beta\in L^1(\Omega;\R^{3\times 3})$ and Borel sets $A\subseteq\Omega$.

\begin{proposition}\label{propcptmubd}
  Let $\Omega\subset\R^3$ be a bounded Lipschitz set.
Assume  $W$ obeys
\HWFinite\ or \HWLin\  for some $p\in (1,2]$.
Let  $(\mu_\eps,\beta_\eps)\in\calA^*_\eps$ be such that $\mu_\eps$ is $(h_\eps,\alpha_\eps)$-dilute in the sense of Definition \ref{defdilutedistribution}, where $(h_\eps,\alpha_\eps)$ obeys \eqref{eqdefheps}, and

 $$\Elast[\beta_\eps, \Omega\setminus(\supp\mu_\eps)_{\rho_\eps}] \le E \eps^2 \ln \frac1\eps,$$ 
 for some $E>0$ and for some $\rho_\eps\geq \eps$ such that  
 \begin{equation}\label{eqlnrhoepsepscpt}
 \lim_{\eps\to0}\frac{\log\rho_\eps}{\log\eps}=1. 
 \end{equation}
Then there is ${c}>0$ such that for any $\Omega'\subset\subset\Omega$ 
\begin{equation}\label{eq:compactness1}
 \limsup_{\eps\to0}\frac1\eps |\mu_\eps|(\Omega')\le {cE}.
\end{equation}
Moreover there are $\mu\in  \calM_{\calB}^1(\Omega)$ and a subsequence $\eps_k$ such that 
\begin{equation}\label{eq:compactness2}
\frac{1}{\eps_k}\mu_{\eps_k}\weakstarto\mu \text{ in } \Omega'
\text{ for every $\Omega'\subset\subset\Omega$.}
\end{equation}
\end{proposition}

\begin{remark}Proposition \ref{propcptmubd} still holds true if we replace  $\calA^*_\eps$ with
  $\calA^\core_{\rho_\eps,\eps}$
   or with $\calA^\moll_\eps$ defined in \eqref{eqdefacoreepss} and \eqref{eq:Amoll} respectively. Indeed, in both cases is suffices to extend (separately for each $i$) the restriction $\beta_\eps|_{T^i_\eps\setminus \hat T^i_\eps}$ to a function $\tilde\beta_\eps$ in $L^1(T^i_\eps;\R^{3\times 3})$ with $\curl\tilde\beta_\eps=\mu\LL T^i_\eps$
as was done at the beginning of the proof of Corollary~\ref{corcellproblem}.   
\end{remark}

\begin{proof}
 
Since $\mu_\eps$ is dilute, we have that $\mu_\eps= \sum_i \eps b^i_\eps\otimes t_\eps^i {\cal H}^1\LL\gamma^i_\eps$, where $\gamma^i_\eps\subset\overline\Omega$ are finitely many closed segments satisfying the diluteness conditions {(i)--(iv)} of Definition \ref{defdiluteness-curve}. 
Possibly subdividing the segments into smaller ones, we can assume that they have length in $[h_\eps, 2h_\eps]$.
We set $\gamma_\eps:=\cup_i\gamma_\eps^i$.
We fix $\Omega'\subset\subset\Omega$ and
 set $I_\eps:=\{i: \gamma^i_\eps\cap\Omega'\ne\emptyset\}$.
We choose $r_\eps=\rho_\eps$, $R_\eps:=(\alpha_\eps h_\eps)^2$, $\delta_\eps:=\alpha_\eps h_\eps$.
We let $S_\eps^i\subset\R$ be a segment of length 
\begin{equation}\label{eqcalh1sepsi}
\calH^1(S_\eps^i)=
\calH^1(\gamma_\eps^i)-2\delta_\eps\ge (1-2\alpha_\eps) \calH^1(\gamma_\eps^i),
\end{equation}
$A_\eps^i$ an affine isometry that maps $S_\eps^ie_3$ into $\gamma_\eps^i$ and the midpoint of $S_\eps^ie_3$ to the midpoint of $\gamma_\eps^i$. We define the cylinders
\begin{equation}\label{def:cylinders-compactness}
 \cyl^i_\eps:=A_\eps^i(B'_{R_\eps}\times S_\eps^i)\text{ and }
 \hat \cyl^i_\eps:=A_\eps^i(B'_{r_\eps}\times S_\eps^i) \,.
\end{equation}
Disjoint segments in the family $\{\gamma_\eps^i\}_i$ are separated by $\alpha_\eps h_\eps\gg R_\eps$. {Since} the  angle between two non-disjoint segments is {at least} $\alpha_\eps$, {and} for small $\eps$ {we have} $\delta_\eps \tan \frac12\alpha_\eps > R_\eps$ {we obtain that} the sets $\{\cyl_{\eps}^i\}_i$ are pairwise disjoint and $\cyl^i_\eps\cap \gamma^j_\eps=\emptyset$ for all $j\ne i$. 
Further, for $\eps$ sufficiently small 
we have $2h_\eps+R_\eps\le\dist(\Omega',\partial\Omega)$ which implies
$T_\eps^i\subseteq\Omega$
for all $i\in I_\eps$.
Therefore
\begin{equation}\label{eqcompsum}
\begin{split}
\sum_{i\in I_\eps}  \Elast[\beta_\eps, \cyl_{\eps}^i\setminus \hat \cyl_\eps^i]\le&
\Elast[\beta_\eps,\Omega \cap (\gamma_\eps)_{R_\eps} \!\!\setminus\!\! (\gamma_\eps)_{\rho_\eps}]\\
\le&
 \Elast[\beta_\eps,\Omega\setminus(\gamma_\eps)_{\rho_\eps}].
 \end{split}
\end{equation}
By Lemma~\ref{lemmacompactnesscellpb}, 
\begin{equation}\label{eqsumicylinders}
 c \sum_{i\in I_\eps} \eps^2 \calH^1(S^i_\eps) {|b_\eps^i|}\ln \frac{R_\eps}{r_\eps}  \le 
  \Elast[\beta_\eps,\Omega\setminus(\gamma_\eps)_{\rho_\eps}]\le { E} \eps^2 \ln\frac1\eps.
\end{equation}
By the diluteness condition
{\eqref{eqdefheps}} and \eqref{eqlnrhoepsepscpt},
\begin{equation}\label{eqrrepslog}
 \lim_{\eps\to0}\frac{\ln \frac{R_\eps}{r_\eps}}{\ln \frac{1}{\eps}}=
 \lim_{\eps\to0}\frac{\ln \frac{\alpha_\eps^2h_\eps^2}{\rho_\eps}}{\ln \frac{1}{\eps}}=1.
\end{equation}
Since $\calH^1(S^i_\eps) \ge (1-2\alpha_\eps) \calH^1(\gamma^i_\eps) $, we conclude that {for small $\eps$}
\begin{equation}\label{eqmuepsomegap}
\frac1\eps|\mu_\eps|(\Omega')
{=\sum_{i\in I_\eps} |b_\eps^i| \calH^1(\gamma_\eps^i\cap\Omega')}
\le c  E.\end{equation}
Therefore there is a subsequence {such that $\eps^{-1}\mu_\eps$ converges to a limiting measure $\mu\in\calM(\Omega';\R^{3\times 3})$, 
with $|\mu|(\Omega')\le c E$. Since $\Omega'$ was arbitrary, we have $\mu\in\calM(\Omega;\R^{3\times 3})$, and there is a diagonal subsequence such that $\eps_k^{-1}\mu_{\eps_k}\weakstarto\mu$ in all $\Omega'\subset\subset\Omega$.} {By \cite{ContiGarroniMassaccesi2015} the limit also belongs to $\calM_\calB^1(\Omega)$. This concludes the proof of \eqref{eq:compactness1} and \eqref{eq:compactness2}.}
\end{proof}

\begin{proposition}\label{propcptmubdbeta}
Let $\Omega\subset\R^3$ be a  bounded, connected Lipschitz set.
Assume $W$ obeys
\HWFinite\ or \HWLin\ for some $p\in(1,2)$. Let $(\mu_\eps,\beta_\eps)\in\calA^*_\eps$ be such that $\mu_\eps$ is $(h_\eps,\alpha_\eps)$-dilute in the sense of Definition \ref{defdilutedistribution}, where $(h_\eps,\alpha_\eps)$ obeys \eqref{eqdefheps}, and
$$\Elast[\beta_\eps, \Omega] \le E \eps^2 \ln \frac1\eps$$ for some $E>0$,

Then there are a subsequence $\eps_k\to0$,
$\conteta\in L^{1}_\loc(\Omega;\R^{3\times 3})$ with $\curl\conteta=0$,
$\mu\in \calM_\calB^1(\Omega)$
and $Q_*\in\SO(3)$
such that $(\mu_{\eps_k},\beta_{\eps_k})$ converges to 
$(\mu,\eta, Q_*)$ in finite kinematics with $p$ growth if \HWFinite\ holds, and  
to $(\mu,\eta)$ in infinitesimal kinematics with $p$ growth if \HWLin\ holds, in the sense of Definition~\ref{defconvergence}.

\end{proposition}
\begin{proof}
We first remark that Proposition~\ref{propcptmubd} holds for this sequence {(with $\rho_\eps=\eps$)}. Fix a connected Lipschitz set $\Omega'\subset\subset\Omega$.
Then by \eqref{eq:compactness1} for $\eps$ small enough we have 
\begin{equation}\label{eq:compactness3}
|\mu_\eps|(\Omega')\le {cE}\eps.
\end{equation}
Let 
\begin{equation}
{d_\eps}:=\dist(\beta_\eps,\SO(3))\in L^1(\Omega).
\end{equation}
{From \HWFinite\ we obtain the pointwise bound $d_\eps^2\wedge d_\eps^p=\mixedgrowth(d_\eps)\le c W(\beta_\eps)$.}
Since 
{$q:=\frac32\wedge p<2$,}
\begin{equation}
 \left(\int_{\Omega\cap \{{d_\eps}\le 1\}} {d_\eps^q} dx\right)^{1/q} \le 
 {(\calL^3(\Omega))^{\frac1q-\frac12}}
 \left(\int_{\Omega\cap 
 \{{d_\eps}\le 1\}} {d_\eps^2} dx \right)^{1/2}\le c\left(\eps^2 \ln\frac1\eps E\right)^{1/2}.
\end{equation}
{Using first $q\le p$ and then $q< 2$, for $\eps$ small}
\begin{equation}
\left( \int_{\Omega\cap \{{d_\eps}>1\}} {d_\eps^q} dx\right)^{1/q} \le
\left( \int_{\Omega\cap \{{d_\eps}>1\}} {d_\eps^p} dx \right)^{1/q}
 {\le c\left(\eps^2 \ln\frac1\eps E\right)^{1/2}.}
\end{equation}
Combining these two estimates,
\begin{equation}
 \|\dist(\beta_\eps,\SO(3))\|_{L^q(\Omega)} \le
c\eps \ln^{1/2}\frac1\eps E^{1/2}.
\end{equation}
By {Proposition \ref{proprigiditycras} and  \eqref{eq:compactness3},}
there is $Q_\eps\in\SO(3)$ such that
{for small $\eps$}
\begin{equation}\label{eqrigidhatbetaeps}
{ \|\beta_\eps-Q_\eps\|_{L^{q}(\Omega')}}\le c |\mu_\eps|(\Omega') + c \|\dist(\beta_\eps,\SO(3))\|_{L^q(\Omega)} 
 \le  
c\eps \ln^{1/2}\frac1\eps E^{1/2}
\end{equation}
{where $c$ may depend on $\Omega'$}.
The rotation $Q_\eps$, however, may be chosen not to depend on $\Omega'$.
Passing to a further  subsequence we can assume $Q_{\eps_k}\to Q_*$ for some $Q_*\in\SO(3)$.
Taking a diagonal subsequence, this holds for all subsets $\Omega'$.

From \eqref{eqrigidhatbetaeps} we see that, after extracting a further subsequence, there is $\eta\in L^q(\Omega';\R^{3\times 3})$ such that
\begin{equation}\label{eqeetaepsweakc}
\eta_{\eps_k}:=
\frac{Q_{\eps_k}^T\beta_{\eps_k}-\Id}{{\eps_k} \ln^{1/2}\frac1{\eps_k}} 
\weakto \eta \text{ weakly in $L^q(\Omega';{\R^{3\times 3}})$}.
\end{equation}
Again, after taking a diagonal subsequence this holds for all $\Omega'$.
{In order to prove $\curl\eta=0$, we write}
\begin{equation}
 \curl\eta_\eps = 
\frac{Q_\eps^T\curl \beta_\eps}{\eps \ln^{1/2}\frac1\eps}=\frac{1}{\eps\ln^{1/2}\frac1\eps} Q_\eps^T\mu_\eps .
\end{equation}
As $|\mu_\eps|(\Omega')\le c \eps$, $\curl \eta_\eps\to0$ distributionally in $\Omega'$.

The proof in the geometrically linear case is analogous.
\end{proof}

\subsection{Lower bound}

  {We define, given 
 a   pair $(\mu,\beta)\in\calM_{\eps\calB}^1(\Omega)\times L^1(\Omega;\R^{3\times 3})$ and  $\rho,\eps>0$,}
  \begin{equation}\label{eqdefelbeps}
   \Elb_{\rho,\eps}[\mu,\beta]:=\begin{cases}\displaystyle
 \Elast[\beta, \Omega\setminus(\supp\mu)_{\rho}], 
 & \text{ if $(\mu,\beta)\in\mathcal A^\core_{\rho,\eps}$},\\
 \infty, & \text{ otherwise}.
   \end{cases}
  \end{equation}
We recall that the class $\mathcal A^\core_{\rho,\eps}$ of admissible pairs
was defined in \eqref{eqdefacoreepss}.
As $\calA_\eps^*\subseteq\calA_{\rho,\eps}^\core$ for every $\rho>0$ {by Lemma~\ref{lem:comparison-admissible-pairs}\ref{lem:comparison-admissible-pairsstarcore}},
for any $(\mu,\beta)\in \calM^1_{\eps\calB}(\Omega)\times L^1(\Omega;\R^{3\times 3})$ and all $\rho>0$ we have 
\begin{equation}\label{eqelbesubcr}
\frac{1}{\eps^2\log\frac1\eps}\Elb_{\rho,\eps}[\mu,\beta]\le F^\subcr_\eps[\mu,\beta]. 
\end{equation}
 This will permit (in Section~\ref{sec-mainproofs}) to prove the lower bound in Theorem~\ref{thm:Gamma-limit} by means of a lower bound on $\Elb_{\rho_\eps,\eps}$ for a good choice of $\rho_\eps$. Precisely, we assume that
$\rho_\eps\to0$ is such that
\begin{equation}\label{largercore}
 \lim_{\eps\to0}\frac{\log\rho_\eps}{\log\eps}=1.
\end{equation}

\begin{theorem}\label{theoremlowerboundsec5}
 Let $\Omega\subset\R^3$ be open, bounded, Lipschitz.
Let
$\rho_\eps\to0$ be such that \eqref{largercore} holds true.
Let
$(\mu_\eps,\beta_\eps)\in\calA^\core_{\rho_\eps,\eps}$ be such that $\mu_\eps$ is $(h_\eps,\alpha_\eps)$-dilute, where $(h_\eps,\alpha_\eps)$ obeys \eqref{eqdefheps}.

Assume that $W$ obeys \HWFinite\ for some $p\in(1,2]$. 
Assume that there are a sequence $\eps_k\to0$  and $(\mu,\eta,Q)\in\mathcal{M}^1_{\calB}(\Omega)\times L^1_{\rm loc}(\Omega;\R^{3\times3})\times \SO(3)$, with $\mu=b\otimes t\calH\LL\gamma$ and $\curl\conteta=0$,  such that 
$(\mu_{\eps_k},\beta_{\eps_k})$ converges to 
$(\mu,\eta,Q)$  in finite kinematics with $p$ growth and radius $\rho_{\eps_k}$ in the sense of Definition~\ref{defconvergence2}.

Then 
\begin{equation}\label{eqlowerboundprop}
 \int_\gamma \psi_\C^\rel (b,Qt) d\calH^1+ \int_\Omega\frac12  \C_Q\conteta\cdot \conteta dx \le \liminf_{k\to\infty} \frac{1}{\eps_k^2\ln \frac1{\eps_k}} \Elb_{\rho_{\eps_k},\eps_k}[\mu_{\eps_k},\contbeta_{\eps_k}],
\end{equation}
with $\C:=D^2W(\Id)$ and $\psi_\C^\rel$ defined as in \eqref{psi-rel}.
Further, if the right-hand side of
\eqref{eqlowerboundprop} is finite  then
$\eta\in L^2(\Omega;\R^{3\times 3})$. 

Assume $W$ obeys \HWLin\ for some $p\in(1,2]$.  Then the same holds with respect to the convergence in infinitesimal kinematics with $p$ growth and radius $\rho_{\eps_k}$  in the sense of Definition~\ref{defconvergence2} 
with $\psi_\C^\rel(b,Qt)$ replaced by 
$\psi_\C^\rel(b,t)$ and $\C_Q$ by $\C=D^2W(0)$. 
\end{theorem}

\begin{proof}
We discuss the proof in the finite case, the linear case is similar.
For notational simplicity we write $\eps$ for $\eps_k$.
Without loss of generality we can assume that $\Omega$ is connected and that the $\liminf$ is finite and, possibly passing to a further subsequence, that it is a limit; in particular
\begin{equation}\label{eqElbE}
 \Elb_{\rho_\eps,\eps}[\mu_\eps,\beta_\eps]\le E \eps^2\log\frac1\eps
\end{equation}
{for some $E>0$.}
During the proof we shall extract further subsequences, which are not made explicit in the notation.\\
By hypothesis on $(\mu_\eps,\beta_\eps)$ we have that 

\begin{equation} \label{eq:conv-mu}
 \frac1\eps\mu_\eps\weakstarto\mu   \hskip5mm{\text{in $\Omega'$ for all $\Omega'\subset\subset\Omega$,}}
 \end{equation}
and there is $Q_\eps\in \SO(3)$  such that $Q_\eps\to Q$ and 
\begin{equation}\label{eq:conv-beta}
\eta_\eps:= \frac{Q_\eps^T\beta_\eps-\Id}{\eps(\ln\frac1\eps)^{1/2}}  \chi_{\Omega\setminus (\supp\mu_{\eps})_{\rho_{\eps}}}
 \weakto \conteta \text{ weakly in } L^{q}_\loc(\Omega;\R^{3\times 3})
\end{equation}
for {$q:=\frac32\wedge p$}.

Fix a connected Lipschitz set $\Omega'\subset\subset\Omega$.
The key idea in the proof is to use Proposition~\ref{propcellproblem} and the double differentiability of $W$ near the identity to show that, writing $\mu_\eps=\eps b_\eps\otimes t_\eps\calH^1\LL\gamma_\eps$, there are $\sigma_\eps\to0$ and 
two disjoint subsets
$\Omega'_\eps, \omega'_\eps\subset\Omega\setminus (\gamma_\eps)_{\rho_\eps}$ such that 
$\calL^3(\Omega'\setminus\Omega'_\eps)\to0$,
\begin{equation}\label{eqlinetensbea}
(1-\sigma_\eps)\int_{\Omega'\cap \gamma_\eps} \psiC({ b_\eps}, Q_\eps t_\eps)d\calH^1 \le  \frac{1}{\eps^2\ln\frac1\eps}
\int_{\omega_\eps'}
W(\beta_\eps)dx,
\end{equation}
and
\begin{equation}\label{eq:elasticenergy}
\int_{\Omega'_\eps}\frac12\C_{Q_\eps}\eta_\eps\cdot\eta_\eps dx-\sigma_\eps \le
\frac{1}{\eps^2\ln\frac1\eps}
 \int_{\Omega'_\eps}W(\beta_\eps)dx.
\end{equation}

We start by showing \eqref{eqlinetensbea}, following the setting in
the proof of Proposition~\ref{propcptmubd}.
Since $\mu_\eps$ is dilute we can write $\mu_\eps=\sum_i \eps b^i_\eps\otimes t_\eps^i {\cal H}^1\LL\gamma^i_\eps$, with $\gamma_\eps=\cup_i \gamma_\eps^i$ $(h_\eps,\alpha_\eps)$-dilute, 
and we can assume $\calH^1(\gamma_\eps^i)\in[h_\eps,2h_\eps]$.
 We define $R_\eps$, $S_\eps^i$, $T_\eps^i$, $  \hat \cyl^i_\eps $, $I_\eps$ as in the proof of Proposition~\ref{propcptmubd} where now  $r_\eps:=\max\{\rho_\eps,\eps h_\eps^{-2}\}$. We set 
 \begin{equation}
 \gamma'_\eps:=\bigcup_{i\in I_\eps}\gamma^i_\eps
\hskip5mm\text{ and }\hskip5mm\omega'_\eps:=\bigcup_{i\in I_\eps} T^i_\eps\setminus \hat T^i_\eps\subseteq
 (\gamma'_\eps)_{R_\eps}\setminus (\gamma_\eps)_{\rho_\eps}.
 \end{equation}
 Clearly $\Omega'\cap\gamma_\eps\subseteq \gamma_\eps'$.
 As in Proposition~\ref{propcptmubd},
 for $\eps$ sufficiently small the $T^i_\eps$ are disjoint and
 $T^i_\eps\subseteq\Omega$ for all $i\in I_\eps$.
 By Corollary \ref{corcellproblem} with $h=\calH^1(S_\eps^i)$, $R=R_\eps$, $r=r_\eps$, for every index $i\in I_\eps$ we get
  \begin{equation}\label{eqbdelbonefyl}
  \Elast[\beta_\eps, \cyl_{\eps}^i\setminus \hat \cyl_\eps^i]\ge 
  {(1-\tilde\sigma^i_\eps)}
  \eps^2 \calH^1(S_\eps^i) \ln\frac {R_\eps}{r_\eps} \psiC(b_\eps^i,Q_\eps t_\eps^i) ,
  \end{equation}
  {where}
  \begin{equation} 
  \begin{split}
  {\tilde\sigma_\eps^i:=  \tilde\sigma}\left({\frac{r_\eps}{R_\eps}}+ \frac{\|\beta_\eps-Q_\eps\|_{L^1(\Omega'\setminus(\gamma_\eps)_{r_\eps})}}{R_\eps^3}+ {\frac{\eps |b_\eps^i|}{r_\eps}}\right),
  \end{split}
  \end{equation}
and $\tilde\sigma\colon(0,\infty)\to(0,\infty)$ nondecreasing and such that $\lim_{s\to0}\tilde\sigma(s)=0$. We show that the argument of $\tilde \sigma$ converges uniformly in $i$ to zero.
From the definitions of $R_\eps$ and $r_\eps$, the diluteness condition \eqref{eqdefheps} and \eqref{largercore} we obtain $r_\eps/R_\eps\to0$.
By \eqref{eq:conv-beta} we have that $\sup_\eps\|\eta_\eps\|_{L^1(\Omega')}<\infty$, and recalling 
the definition of $R_\eps$ and the diluteness assumption,
\begin{equation}
\begin{split}
 \lim_{\eps\to0} 
\frac{\|\beta_\eps-Q_\eps\|_{L^1(\Omega'\setminus(\gamma_\eps)_{r_\eps})}}{R_\eps^3}\le
\lim_{\eps\to0} \frac{\|\eta_\eps\|_{L^1(\Omega')} \eps \ln^{1/2}\frac1\eps}{\alpha_\eps^6h_\eps^6}=0.
\end{split}
 \end{equation}
{We remark that 
\eqref{eq:compactness1} implies that for sufficiently small $\eps$ one has}
\begin{equation}\label{eqbepsiunif}
 {\sum_{i\in I_\eps} |b_\eps^i|\calH^1(\gamma_i^\eps)} \le {cE},
\end{equation}
{and using $\calH^1(\gamma_i^\eps) \ge h_\eps$} and  {$r_\eps\ge\eps h_\eps^{-2}$}, we obtain
\begin{equation}
{ \lim_{\eps\to0} \sup_{i\in I_\eps} \frac{\eps |b_\eps^i|}{r_\eps} 
\le 
 \lim_{\eps\to0} h_\eps \sum_{i\in I_\eps} |b_\eps^i|\calH^1(\gamma_i^\eps)
\le \lim_{\eps\to0} c h_\eps=0.}
\end{equation}
Therefore for all ${i\in I_\eps}$
\begin{equation} 
 \begin{split}
\tilde\sigma_\eps^i\le{\hat\sigma_\eps:=  \tilde\sigma}\left({\frac{r_\eps}{R_\eps}}+ \frac{c \eps \ln^{1/2}\frac1\eps}{\alpha_\eps^6h_\eps^6} +ch_\eps\right)
 \end{split}
\end{equation}
with $\hat\sigma_\eps\to0$.
{In turn, summing \eqref{eqbdelbonefyl} over $i\in I_\eps$ and using \eqref{eqcalh1sepsi} we obtain
  \begin{equation} \label{eqlbgigpres}
  \begin{split}
 \frac{1}{\eps^2\ln\frac1\eps} \Elast[\beta_\eps,  \bigcup_{i\in I_\eps}\cyl_{\eps}^i\setminus \hat \cyl_\eps^i]\ge & (1-\hat\sigma_\eps) \sum_{i\in I_\eps}\calH^1(S_\eps^i) \frac{\ln\frac {R_\eps}{r_\eps}}{\ln\frac1\eps} \psiC(b_\eps^i,Q_\eps t_\eps^i) \\
 \ge &(1-\hat\sigma_\eps) \frac{\ln\frac {R_\eps}{r_\eps}}{\ln\frac1\eps} (1-2\alpha_\eps)\int_{\gamma_\eps'} \psiC(b_\eps,Q_\eps t_\eps) d\calH^1.
  \end{split}
  \end{equation}
By 
\begin{equation}
 \lim_{\eps\to0}\frac{\ln \frac{R_\eps}{r_\eps}}{\ln \frac{1}{\eps}}=
 \lim_{\eps\to0}\frac{\ln \frac{\alpha_\eps^2h_\eps^4}{{\max\{\rho_\eps, \eps h_\eps^{-2}\}}}}{\ln \frac{1}{\eps}}=1,
\end{equation}
 the factor in front of the integral converges to 1.   Therefore 
\eqref{eqlinetensbea} follows.}
{In turn, from \eqref{eqlinetensbea}} we obtain
\begin{equation}
\liminf_{\eps\to0} \int_{ \gamma_\eps'} \psiC(b_\eps,Q_\eps t_\eps) d\calH^1 \le
\liminf_{\eps\to0}   \frac{1}{\eps^2\ln\frac1\eps}
\int_{\omega_\eps'}
W(\beta_\eps)dx. 
\end{equation}
By $Q_\eps\to Q$, \eqref{eqpsictlip} and the relaxation result of \cite[Theorem 3.1]{ContiGarroniMassaccesi2015}, this implies
\begin{equation}\label{eqlinetensbeadue} 
\int_{\Omega'\cap\gamma} \psi_\C^\rel (b,Qt) d\calH^1\le 
\liminf_{\eps\to0}   \frac{1}{\eps^2\ln\frac1\eps}
\int_{\omega_\eps'}
W(\beta_\eps)dx. 
\end{equation}

We next show \eqref{eq:elasticenergy}.
Since $W$ is twice differentiable at the identity, there is $\omega:[0,\infty)\to[0,\infty)$, {monotone and} continuous in $0$, with $\omega(0)=0$, such that for any $\delta>0$ sufficiently small we have
\begin{equation}
W({\Id+F})\ge \frac12 \C F \cdot F - \omega({\delta}) |F|^2  \text{ for all } F\in\R^{3\times 3} \text{ with } |F|\le \delta.
\end{equation}
We shall apply this with $F=(Q_\eps^T\beta_\eps(x)-\Id)\chi_{\Omega\setminus (\gamma_\eps)_{\rho_\eps}}
{(x)}=\eps\ln^{1/2}\frac1\eps \eta_\eps(x)$, away from the boundary, the singularity, and the parts where $\eta_\eps$ is exceptionally large. To make this precise, we first define
\begin{equation}
 \hat\Omega'_\eps:=\Omega'\setminus (\gamma'_\eps)_{R_\eps}\setminus (\partial\Omega')_{R_\eps}\subseteq \Omega'\setminus 
 (\gamma_\eps)_{R_\eps}.
\end{equation}
As $\Omega'$ is Lipschitz, $\gamma_\eps$ is dilute and
$\calH^1(\gamma^i_\eps)\in[h_\eps, 2h_\eps]$, using 
{\eqref{eqbepsiunif} and $|b_i|\ge c>0$} we obtain
\begin{equation}\label{eqhatomegapes}
\calL^3(\Omega'\setminus \hat\Omega'_\eps)
\le 
c(\Omega')R_\eps + 
c \sum_{i\in I_\eps} 
(R_\eps^3+h_\eps R_\eps^2)
 \le 
 cR_\eps+
 c R_\eps^2 \sum_{i\in I_\eps}  \calH^1(\gamma_\eps^i)\to0.
\end{equation}
We then choose $L_\eps\to\infty$ such that
 $L_\eps\eps \ln^{1/2}\frac1\eps+ L_\eps^2  \omega(L_\eps\eps 
\ln^{1/2}{\frac1\eps})\to0$ and define
\begin{equation}\begin{split}
 \Omega'_{\eps}:=&\{x\in\hat\Omega'_\eps: |\beta_\eps(x)-Q_\eps|\le L_\eps \eps\ln^{1/2}\frac1\eps\}
=\{x\in\hat\Omega'_\eps: |\eta_\eps(x)|\le L_\eps\}.
\end{split}
\end{equation}
Since $\eta_\eps$ is bounded in $L^q(\Omega';{\R^{3\times 3}})$, $\calL^3(\hat\Omega'_\eps \setminus \Omega'_\eps)\to0$,
and recalling \eqref{eqhatomegapes} we obtain $\calL^3(\Omega'\setminus \Omega'_\eps)\to0$.
For any $x\in\Omega'_{\eps}$ we have
\begin{equation}
\begin{split}
 W(\beta_\eps(x))=&W(\beta_\eps(x)Q_\eps^T)\\
 \ge&
 \frac12 \C (\beta_\eps(x)Q_\eps^T-\Id)\cdot (\beta_\eps(x)Q_\eps^T-\Id) \\
 &\hskip1cm - |\beta_\eps(x)Q_\eps^T-\Id|^2
 \omega({L_\eps\eps \ln^{1/2}(1/\eps)})  \\
 \ge&\eps^2\ln \frac1\eps\left(
 \frac12 \C_{Q_\eps} \eta_\eps(x)\cdot\eta_\eps (x)- L_\eps^2
 \omega({L_\eps\eps \ln^{1/2}(1/\eps)})  \right) .
 \end{split}
\end{equation}
Therefore
\begin{equation}
\begin{split}
\frac{1}{\eps^2\ln\frac1\eps}
 \int_{\Omega'_{\eps}} W(\beta_\eps) dx \ge& 
 \int_{\Omega'_{\eps}} \frac12\C_{Q_\eps}\eta_\eps\cdot\eta_\eps dx-  L_\eps^2\omega({L_\eps\eps \ln^{1/2}(1/\eps)}) 
 \calL^3(\Omega)
  \end{split}
\end{equation}
which is \eqref{eq:elasticenergy}.
Moreover
by \eqref{eq:conv-beta} and $\calL^3(\Omega'\setminus \Omega'_\eps)\to0$ we obtain that
\begin{equation}
 \eta_\eps\chi_{\Omega'_\eps}\weakto \eta
 \text{ in } L^1(\Omega';{\R^{3\times 3}})
\end{equation}
and, by lower semicontinuity and
$\C_{Q_\eps}\to\C_Q$, 
\begin{equation}\label{eqelastertatar}
 \int_{\Omega'} \frac12  \C_Q\eta\cdot\eta dx \le \liminf_{\eps\to0}
\frac{1}{\eps^2\ln\frac1\eps}
 \int_{\Omega'_{\eps}} W(\beta_\eps) dx.
\end{equation}
Combining  \eqref{eqlinetensbeadue}
and \eqref{eqelastertatar}, recalling that $\omega'_\eps\cap\Omega'_\eps=\emptyset$,  we obtain
\begin{equation}\
 \int_{\Omega'\cap\gamma} \psi_\C^\rel (b,Qt) d\calH^1+ \int_{\Omega'}\frac12  \C_Q\conteta\cdot \conteta dx \le \liminf_{k\to\infty} \frac{1}{\eps_k^2\ln \frac1{\eps_k}} \Elb_{\rho_{\eps_k},\eps_k}[\mu_{\eps_k},\contbeta_{\eps_k}]
\end{equation}
for all Lipschitz sets $\Omega'\subset\subset\Omega$; 
as the right-hand side does not depend on $\Omega'$ the same holds on $\Omega$. This proves
\eqref{eqlowerboundprop}.

If the right-hand side is finite, then $\eta+\eta^T\in L^2(\Omega;\R^{3\times 3})$. By Korn's inequality, recalling that $\curl\eta=0$, this implies $\eta\in L^2(\Omega;\R^{3\times3})$.
\end{proof}

\section{Upper bound}\label{secupperbound}

A key technical result of this section, which permits a refinement of the argument used for the upper bound in \cite{ContiGarroniOrtiz2015,garroni2020nonlinear} is the following statement. It shows that the upper bound for the unrelaxed problem can be taken at constant $\beta$.

\begin{proposition}\label{propupperboundcontunrelax2}
Let $\C$ obey (\ref{eqdefC}).
 Let $\mu=b\otimes t \calH^1\LL \gamma \in \calM_\calB^1(\R^3)$ be polygonal, and let $\contbeta\in L^{3/2}(\R^3;\R^{3\times3})$ be the solution to
 \begin{equation}\label{eqsystembetamu}
  \begin{cases}
   \curl\contbeta=\mu,\\
   \Div \C\contbeta=0.
  \end{cases}
 \end{equation}
Let $\nu_\eps$ be the {positive} measure defined by
\begin{equation*}
 \nu_\eps:= \frac{1}{\ln\frac1\eps}\frac12  \C\contbeta\cdot \contbeta  \calL^3\LL(\R^3\setminus (\gamma)_\eps).
\end{equation*}
 Then $\nu_\eps$ converges weakly{-$*$} to
 \begin{equation}
  \nu_0:=\psiC(b,t)\calH^1\LL \gamma.
 \end{equation}
 In particular, if $\Omega\subset\R^3$ is an open set with $|\mu|(\partial\Omega)=0$, then
\begin{equation}\label{eqprop61omega}
 \lim_{\eps\to0} \frac{1}{\ln \frac1\eps}
 \int_{\Omega\setminus (\gamma)_\eps} \frac12 \C \contbeta\cdot  \contbeta dx  = \int_{\Omega\cap\gamma}  \psiC(b,t) d\calH^1.
\end{equation} 
Further, the positive measures
\begin{equation}\label{eqdefhatnueps}
\hat\nu_\eps:= \frac{1}{\ln\frac1\eps} |\contbeta|^2 \calL^3\LL(\R^3\setminus (\gamma)_\eps)
\end{equation}
are uniformly locally bounded, in the sense that for all $R>0$ one has 
\begin{equation}\label{eqsupepsnuepsbr}
\limsup_{\eps\to0}\hat\nu_\eps(B_R)<\infty,
\end{equation}
{and for any $s\in (0,1)$ obey
\begin{equation}\label{eqepssnuhat}
 \limsup_{\eps\to0}\hat\nu_\eps((\gamma)_{\eps^s})\le c (1-s).
\end{equation}}
\end{proposition}

\begin{proof}
Since $\mu$ is polygonal, we have that $\gamma=\cup_i\gamma_i$, with the $\gamma_i$ being finitely many pairwise disjoint segments in $\R^3$. For the same reason, $b\in L^\infty(\gamma,\calH^1;\calB)$.
 By Theorem \ref{theoremsolr3}\ref{lemmasolr3lp}, for any $x\not\in\overline \gamma$ we have
\begin{equation*}
 |\contbeta|(x)\le \frac{c}{\dist(x,\gamma)} = c \max_i f_i(x)
\end{equation*}
where
\begin{equation}\label{eqdeffi}
f_i(x):=\frac{1}{\dist(x, \gamma_i)}\,.
\end{equation}
In particular, this shows that 
\begin{equation}\label{eqnuepsnuepsi}
 \nu_\eps\le
c\hat\nu_\eps\le
 \frac{c}{\ln\frac1\eps} \sum_i f_i^2 \calL^3\LL(\R^3\setminus (\gamma)_\eps)\le c
\sum_i \nu_\eps^i,
 \end{equation}
where
\begin{equation*}
 \nu_\eps^i:=
\frac{1}{\ln\frac 1\eps} f_i^2 \calL^3\LL(\R^3\setminus (\gamma_i)_\eps). 
 \end{equation*}
 We fix an index $i$ and estimate $\nu_\eps^i$ separately near $\gamma_i$ and far away. Let $\delta,R>0$.
Since $f_i\le 1/\delta$ outside $(\gamma_i)_\delta$, 
\begin{equation}\label{eqnuioutside}
 \nu_\eps^i(B_R\setminus (\gamma_i)_\delta) \le  \frac{\delta^{-2} }{\ln\frac1\eps}\calL^3(B_R).
\end{equation}
We recall that we denote by $B_r(x)$ and $B'_r(x)$ the balls of radius $r>0$ centered at $x$ in $\R^3$ and $\R^2$ respectively; for $x=0$ we simply write $B_r$ and $B'_r$.
 Let now $x\in \overline{\gamma_i}$, $r>0$.  Then
\begin{equation}\label{eqnuinside}
\begin{split}
 \nu_\eps^i(B_r(x)) &\le \frac{1}{\ln\frac 1\eps} \left( 
 \int_{B_r\setminus B_\eps} \frac{1}{|y|^2} dy + 
 2r \int_{B_r'\setminus B_\eps'} \frac{1}{|y'|^2}dy'\right)\\
& =\frac{1}{\ln\frac1\eps} (4\pi(r-\eps) + 2r \cdot 2\pi \cdot \ln \frac{r}{\eps})\\
&\le  \frac{4\pi r + 4\pi r\ln {(r/\eps)}}{\ln(1/\eps)} .
 \end{split}
\end{equation}
Note that the first integral in the first inequality of \eqref{eqnuinside} is needed to estimate $\nu_\eps^i(B_r(x))$
 in the case that $x$ is near the endpoints of $\gamma_i$. 
For any $\delta\in (0,\calH^1(\gamma_i))$ we can cover $(\gamma_i)_\delta$ with at most $3\calH^1(\gamma_i)/\delta$ balls $B_{2\delta}(x^k)$, with $x^k\in\overline\gamma_i$.  Therefore, using \eqref{eqnuinside} with $r=2\delta$ and \eqref{eqnuioutside},
\begin{equation*}
\nu_\eps^i(B_R) \le \frac{3\calH^1(\gamma_i)}{\delta} 
\frac{8\pi\delta + {8\pi\delta\ln(2\delta/\eps)}}{\ln(1/\eps)}  +
 \frac{\delta^{-2} }{\ln(1/\eps)}\calL^3(B_R),
\end{equation*}
which implies
\begin{equation}\label{eqnuepsibounded}
 \limsup_{\eps\to0} \nu_\eps^i(B_R) \le 24\pi \calH^1(\gamma_i)<\infty
\end{equation} 
for all $R>0$. Recalling \eqref{eqnuepsnuepsi}, this proves \eqref{eqsupepsnuepsbr}. 

{For any $s\in (0,1)$ the same computation, using $\delta=\eps^s$ in \eqref{eqnuinside}, shows that
\begin{equation}\label{eqboundepss}
 \limsup_{\eps\to0} \nu_\eps^i((\gamma_i)_{\eps^s}) \le 3\calH^1(\gamma_i)\limsup_{\eps\to0}
\frac{8\pi + 8\pi\ln(2\eps^{s-1})}{\ln\frac1\eps}   =24\pi \calH^1(\gamma_i) (1-s),
\end{equation}
and similarly $\limsup_{\eps\to0} \nu_\eps^i((\gamma_j)_{\eps^s})=0$ for $j\ne i$. 
This proves \eqref{eqepssnuhat}.
}

By \eqref{eqnuepsibounded}, after taking a subsequence we can assume that for each $i$ $\nu_\eps^i\weakstarto\nu_0^i$ for some Radon measure $\nu_0^i$.
By \eqref{eqnuioutside} we obtain $\nu_0^i(B_R\setminus (\gamma_i)_\delta)=0$ for any $\delta$ and $R$; taking 
$\delta\to0$ and $R\to\infty$ this gives $\nu^i_0(\R^3\setminus \overline{\gamma_i})=0$.
By \eqref{eqnuinside} we obtain 
\begin{equation*}
 \nu_0^i(B_r(x))\le 4\pi r \le 4\pi  \calH^1(B_r(x)\cap \gamma_i)\text{ for any } x\in \overline{\gamma_i}, r\in (0, \calH^1(\gamma_i)).
\end{equation*}
Therefore $\nu_0^i\ll \calH^1\LL \gamma_i$. Summing over $i$ (and possibly taking a further subsequence), $\nu_\eps\weakto\nu_0$ with $\nu_0\ll \calH^1\LL\gamma$.

In order to conclude the proof it suffices to show that for any $i$ we have
\begin{equation}
\frac{d\nu_0}{d\calH^1\LL\gamma}(x)=\psiC(b_i,t_i) \hskip5mm \text{ for $\calH^1$-almost every }x\in\gamma_i,
\end{equation}
where we recall that $b_i\in\calB$ and $t_i\in S^2$ are defined by 
$\mu=\sum_i b_i\otimes t_i\calH^1\LL\gamma_i$.

{For $x\in\gamma_i$ we}
denote by $T_r^x$ a cylinder with axis parallel to $t_i$,  radius $r>0$ and height $2r$ centered at  $x$,
\begin{equation*}
 T_r^x:= x + Q_{t_i} ( B_r'\times (-r,r)),
\end{equation*}
where as usual $Q_{t_i}$ is an element of $\SO(3)$ with $Q_{t_i}e_3=t_i$.
By the Radon-Nikodým theorem, for $\calH^1$-almost every $x\in \gamma_i$ there is a sequence $r_k\to0$ such that $\nu_0(\partial T^x_{r_k})=0$ and
\begin{equation*}
 \frac{d\nu_0}{d\calH^1\LL\gamma_i}(x)=\lim_{k\to\infty} \frac{\nu_0(T_{r_k}^x)}{\calH^1(\gamma_i\cap T_{r_k}^x)}.
\end{equation*}
We can assume that $x$ is not an endpoint of $\gamma_i$. For sufficiently large $k$, $\calH^1(\gamma_i\cap T_{r_k}^x)=2r_k$.
Recalling that $\nu_\eps(T_{r_k}^x)\to \nu_0(T_{r_k}^x)$, the definition of $\nu_\eps$, and $\dist(x,\gamma_j)>0$ for $j\ne i$,
\begin{equation}\label{eqdnudh1gi}
 \frac{d\nu_0}{d\calH^1\LL\gamma_i}(x)=\lim_{k\to\infty}  \frac{1}{2r_k}
 \lim_{\eps\to0} \frac{1}{\ln\frac1\eps}\int_{T_{r_k}^x\setminus (\gamma_i)_\eps}
\frac12 \C\contbeta\cdot\contbeta dy. 
\end{equation}
By Theorem \ref{theoremsolr3}\ref{lemmasolr3intrep}
\begin{equation*}
 \contbeta(y)=\int_{\R^3} N(y-z) d\mu(z)=\sum_j \int_{\gamma_j} N(y-z) b_j\otimes t_j d\calH^1(z)=:\sum_j\contbeta_j(y).
\end{equation*}
Assume first $j\ne i$, so that $\dist(x,\gamma_j)>0$.
Since $|N(z)|\le c/|z|^2$, we have $|\contbeta_j|\le c$ in $T_{r_k}^x$ for sufficiently large $k$.
The term $\beta_i$ instead is compared with $\contbeta_{b_i, t_i}$.
By  Lemma \ref {lemmabetabtkernel1} we have that
\begin{equation*}
 \contbeta_{b_i,t_i}(y-x)-\contbeta_i(y)= \int_{(x+\R t_i)\setminus \gamma_i} 
 N(y-z) (b_i\otimes t_i) d\calH^1(z)
\end{equation*}
which implies, using again $|N(z)|\le c/|z|^2$,
\begin{equation*}
| \contbeta_{b_i,t_i}(y-x)-\contbeta_i(y)|\le c \hskip5mm\text{ for } y\in T_{r_k}^x
\end{equation*}
for sufficiently large $k$ {(with $c$ depending on $x$ and $i$)}. 

As for any $\delta{\in(0,1]}$ 
\begin{equation*}
 \C\contbeta(y)\cdot\contbeta(y) \le (1+\delta) \C\contbeta_{b_i,t_i}(y-x)\cdot \contbeta_{b_i,t_i}(y-x)
 + \frac c\delta |\contbeta(y)-\contbeta_{b_i,t_i}(y-x)|^2
\end{equation*}
we obtain, recalling \eqref{eqpsi0cylinder} in order to evaluate the integral, 
\begin{equation*}
\begin{split}
 \int_{T_{r_k}^x\setminus (\gamma_i)_\eps}
\frac12 \C\contbeta\cdot\contbeta dy
&\le (1+\delta)\int_{T_{r_k}^0\setminus (\R t_i)_\eps}
\frac12 \C\contbeta_{b_i,t_i} \cdot\contbeta_{b_i,t_i}dy
+ \frac c\delta r_k^3\\
&=(1+\delta) 2r_k \ln \frac{r_k}{\eps} \psiC (b_i, t_i)+ \frac c\delta r_k^3.
\end{split}
\end{equation*}
Therefore {\eqref{eqdnudh1gi} gives}
\begin{equation*}
 \frac{d\nu_0}{d\calH^1\LL\gamma_i}(x)\le (1+\delta) \lim_{k\to\infty}  
 \lim_{\eps\to0} \frac{\ln \frac{r_k}{\eps }}{\ln\frac1\eps} \psiC(b_i,t_i)
{=(1+\delta) \psiC(b_i,t_i)}
 \end{equation*}
for any $\delta>0$, so that
\begin{equation*}
 \frac{d\nu_0}{d\calH^1\LL\gamma_i}(x)\le 
 \psiC(b_i, t_i).
\end{equation*}
In order to prove the converse bound we start from \eqref{eqdnudh1gi} and notice that $\beta$ is an admissible test function in \eqref{eqdefabthRrb} for $\infcyl(\C,b_i, t_i, {2r_k}, r_k,\eps)$.
Therefore
\begin{equation}
 \frac{d\nu_0}{d\calH^1\LL\gamma_i}(x)\ge \liminf_{k\to\infty}  
 \liminf_{\eps\to0} 
 \frac{\ln\frac{r_k}{\eps}}{\ln\frac1\eps}
 \infcyl(\C,b_i, t_i, {2r_k}, r_k,\eps)=
 \psiC(b_i,t_i),
\end{equation}
where in the last step we used
Lemma~\ref{lemmacellproblemlb}\ref{lemmacellproblemlbA}.  This concludes the proof.
\end{proof}

\begin{remark}\label{remarkbetamollcore}
{The same
argument also yields
 \begin{equation}\label{eqbetamollcore}
  \limsup_{\eps\to0} \int_{(\gamma)_{2\eps}} |\beta\ast\varphi_\eps|^2 dx <\infty.
 \end{equation}
 To see this, we observe that $|\beta|\le c\sum_i f_i$ implies $|\beta\ast\varphi_\eps|\le c\sum_i (f_i\ast \varphi_\eps)$. With \eqref{eqdeffi}
 we obtain $|f_i\ast \varphi_\eps|\le c/\eps$ everywhere, and therefore
 $\|\beta\ast\varphi_\eps\|_{L^\infty} \le c/\eps$.
 For each segment
 $\gamma_i$ one computes $\calL^3((\gamma_i)_\eps)=
 \pi\eps^2\calH^1(\gamma_i)+\frac{4\pi}3\eps^3$,
and since there are finitely many of them we have $\calL^3((\gamma)_{2\eps})\le c \eps^2 + c \eps^3$. Condition \eqref{eqbetamollcore} follows.}
\end{remark}

We next deal with a technical issue related to the intersection of $\gamma$ and the boundary of $\Omega$. Indeed, in the above proposition one has integrated over the complement of an $\eps$-neighbourhood of the entire curve $\gamma\subset\R^3$, whereas in 
$\Eub_\eps$ one integrates over the complement of  an $\eps$-neighbourhood of the part of $\gamma$ which is contained within $\Omega$. It is apparent that this may generate difficulties if $\gamma$ is tangential to $\partial\Omega$, including the extreme case in which one of the segments composing $\gamma$ is contained in $\partial\Omega$. This situation is obviously not generic, and indeed we show in the next result that in the generic situation this boundary effect vanishes in the limit. 

Let $\Omega\subseteq\R^3$ be open and Lipschitz, $\gamma\subset\R^3$ a finite union of closed segments. We  say that $\gamma$ is transversal to $\partial\Omega$  if $\gamma\cap\partial\Omega$ consists of finitely many points, each of them belongs to only one of the segments, is a point where $\partial\Omega$ has a normal, and the normal is not orthogonal to $\gamma$.

\begin{lemma}\label{lemmatransversal}
Let $\Omega\subseteq\R^3$ be an open Lipschitz set, $\gamma\subset\R^3$ a finite union of closed segments. 
\begin{enumerate}
 \item\label{lemmatransversalgeneric} For ($\calL^3$-)almost all $a\in\R^3$ the curve $a+\gamma$ is transversal to $\partial\Omega$. 
\item\label{lemmatransversallimit} If $\gamma$ is transversal to $\partial\Omega$, then
in the setting of Proposition~\ref{propupperboundcontunrelax2}
\begin{equation}\label{eqepssnuhat3}
 \limsup_{\eps\to0}
 \frac1{\log\frac1\eps} \int_{\Omega\cap (\gamma)_{\eps}\setminus(\Omega\cap\gamma)_\eps}|\beta|^2 dx =0.
\end{equation}
\end{enumerate}
\end{lemma}

The proof of the second part uses the following geometric fact.
\begin{lemma}\label{lemmaliprand}
 Let $\Omega\subseteq\R^n$ be an open Lipschitz set, $x\in\partial\Omega$ with normal $\nu\in S^{n-1}$,  $v\in S^{n-1}$ with $v\cdot\nu>0$. Then there are $c,c',\rho>0$ such that
 \begin{equation}\label{eqlemmaliprand}
t\le c  \,\dist(x+tv, \Omega) \hskip5mm\text{for all $t\in[0,\rho]$}
 \end{equation}
and
\begin{equation}\label{eqlemmaliprand2}
  |x-y|\le c' |x+tv-y|  \hskip5mm\text{for all $t\in[0,\rho]$, $y\in\Omega$}.
 \end{equation}
 \end{lemma}
\begin{proof}
 After a change of variables we can assume that $x=0$, $\nu=e_n$, 
 $v_n>0$, 
 and there is a Lipschitz function $f:\R^{n-1}\to\R$ with $f(0)=0$, $Df'(0)=0$, and
 \begin{equation}
  B_r\cap\Omega=B_r\cap\{y_n<f(y')\}.
 \end{equation}
for some $r>0$.
We write $y=(y',y_n)$, with $y'\in\R^{n-1}$ and $y_n\in\R$, and the same for the other vectors.
 As $Df'(0)=0$ and $v_n>0$, we can 
choose $\rho\in(0,\frac12r]$ such that $|f(y')|\le \frac 12v_n|y'|$ for all $|y'|\le \rho$.
 For $t\in[0,\rho]$ we have $tv\in B_r\setminus \Omega$ and 
 \begin{equation}
\dist(tv,\Omega)=\dist(tv,\Omega\cap B_r)
=\dist(tv,\partial \Omega\cap B_r),   
 \end{equation}
which implies
\begin{equation}\label{dqdisttvonega}
 \dist(tv,\Omega)\ge \min_{y'\in\R^{n-1}} |tv-(y',f(y'))|.
\end{equation}
 We then estimate 
 $|y'|\le t+|y'-tv'|$ and, with $|v'|\le1$,
 \begin{equation}
  |tv_n-f(y')|\ge tv_n-\frac12 v_n\,|y'|
  \ge \frac12 v_nt -\frac12 v_n\,|tv'-y'|
 \end{equation}
 for all $t\in [0,\rho]$,
which together with \eqref{dqdisttvonega} implies
\begin{equation}
 \dist(tv,\Omega)\ge 
 \min_{y'\in\R^{n-1}} \frac12 |tv_n-f(y')| + \frac12 |tv'-y'|
 \ge\frac14 v_nt,
\end{equation}
which is \eqref{eqlemmaliprand} with $c=4/v_n$.

To prove \eqref{eqlemmaliprand2}, let $y\in\Omega$, $t\in[0,\rho]$. Then $|v|=1$,
\eqref{eqlemmaliprand} and $\dist(x+tv,\Omega)\le |x+tv-y|$  give
\begin{equation}
 |x-y|\le t+|x+tv-y|\le (c+1) |x+tv-y|
\end{equation}
which proves \eqref{eqlemmaliprand2}.
\end{proof}

\begin{proof}[Proof of Lemma~\ref{lemmatransversal}]
\ref{lemmatransversalgeneric}: 
There are finitely many points in $\gamma$ which belong to more than one segment, and for almost all choices of $a$ none of them belongs to the 2-dimensional set $\partial\Omega$. For the rest of the argument we can consider each of the segments composing $\gamma$ individually.
Analogously, since $\partial\Omega$ is covered by finitely many Lipschitz graphs, it suffices to prove the assertion for one of them. Therefore it suffices to show the following: if $v\in S^{2}$, 
 $F:\R^{2}\to\R^3$ is Lipschitz, and $r>0$, then for $\calL^3$-almost every $a\in\R^3$ the following holds: the equation
\begin{equation}\label{eqFypa}
F(y')\in a+v\R: y'\in B_r'
\end{equation}
has finitely many solutions $y'_1,\dots, y'_K$, 
at each $y_k'$ the function $F$ is differentiable, 
with $v$ not contained in the space spanned by $\partial_1F(y'_k)$ and $\partial_2F(y'_k)$.
We remark that one can replace $\calL^3$-almost every $a\in\R^3$ with $\calH^2$-almost every $a\in v^\perp$.

To see this, one considers the Lipschitz function 
$G:B_r'\to v^\perp\subseteq\R^3$ defined by
$G(y'):=P_vF(y')$,  where $P_v:=\Id-v\otimes v$ is the projection onto the space $v^\perp$. 
For $a\in v^\perp$, equation \eqref{eqFypa} can be rewritten as $G(y')=a$.
As $F$ is Lipschitz, it is almost everywhere differentiable, and 
$G(\{y'\in B_r': DF(y') \text{ does not exist}\})$ is an $\calH^2$-null set. 
By Sard's Lemma, $G(\{y'\in B_r': {\rm rank\,} DG(y')\le 1\})$ is also an $\calH^2$-null set. 
We observe that, since $DG=(\Id-v\otimes v)DF$, $DG$ has full rank exactly when $v$ is not contained in the space spanned by $\partial_1F$ and $\partial_2F$.
By the area formula \cite[Th.~2.71]{AmbrosioFP} applied to the function $G$ 
we obtain
\begin{equation}
 \int_{B_r'} \sqrt{\det DG^TDG} d\calL^2 =
 \int_{v^\perp} \calH^0(B_r'\cap G^{-1}(a)) d\calH^2(a),
\end{equation}
so that $B_r'\cap G^{-1}(a)$ is a finite set for $\calH^2$-almost all $a$. This proves the claim.
 
\ref{lemmatransversallimit}: 
Let $\{x_1,\dots, x_M\}=\gamma\cap\partial\Omega$.
By Lemma~\ref{lemmaliprand} 
there are $c'$, $\rho>0$ such that \eqref{eqlemmaliprand2} holds for all of them. 
Possibly reducing $\rho$, $\gamma\cap B_{\rho}(x_i)$ is a diameter of $B_{\rho}(x_i)$  for all $i$. Pick $c_*\in(0,1)$ with $c_*c'<1$.
We claim that for $\eps$ sufficiently small one has
\begin{equation}\label{eqomegagammacstepsgam}
 \Omega\cap (\gamma)_{c_*\eps}\setminus(\Omega\cap\gamma)_\eps=\emptyset.
\end{equation}
We prove \eqref{eqomegagammacstepsgam} by contradiction. If not, there are $\eps_k\to0$, $y_k\in\Omega$, $z_k\in\gamma$ with
$|y_k-z_k|< c_*\eps_k$ and $\dist(y_k, \Omega\cap\gamma)\ge \eps_k$. As $c_*<1$ the last two conditions imply $z_k\in\gamma\setminus\Omega$. By compactness of $\bar\Omega$, passing to a subsequence we can assume that $y_k$ and $z_k$ converge to the same limit, which belongs to $\overline\Omega\cap(\gamma\setminus\Omega)=\partial\Omega\cap\gamma$. Therefore there is $i\in\{1,\dots, M\}$ with  $z_k\to x_i$, which means that  for sufficiently large $k$ we have $z_k\in \gamma\cap B_\rho(x_i)$ and therefore
$z_k=x_i+t_kv_i$ for some $t_k\to0$.  By  \eqref{eqlemmaliprand2} we have
\begin{equation}
 |y_k-x_i|\le c'|y_k-z_k|\le c'c_*\eps_k<\eps_k.
\end{equation}
This contradicts $\eps_k\le\dist(y_k,\Omega\cap\gamma)\le |y_k-x_i|$ and concludes the proof of  \eqref{eqomegagammacstepsgam}.

We now observe that  \eqref{eqomegagammacstepsgam} implies
$ \Omega\setminus(\Omega\cap\gamma)_\eps
\subseteq
 \Omega\setminus(\gamma)_{c_*\eps}$ and therefore
\begin{equation}
 \Omega\cap (\gamma)_{\eps}\setminus(\Omega\cap\gamma)_\eps
\subseteq
 \Omega\cap (\gamma)_{\eps}\setminus(\gamma)_{c_*\eps}
\subseteq
(\gamma)_{\eps}\setminus(\gamma)_{c_*\eps}.
\end{equation}
For any $s\in(0,1)$, if $\eps$ is sufficiently small then
$(c_*\eps)^s\ge\eps$. Using first \eqref{eqdefhatnueps} and then
 \eqref{eqepssnuhat},
\begin{equation}\begin{split}
 \limsup_{\eps\to0}
 \frac1{\log\frac1\eps} \int_{(\gamma)_{\eps}\setminus(\gamma)_{c_*\eps}}|\beta|^2 dx \le &
  \limsup_{\eps\to0}
  \frac{\ln\frac1{c_*\eps}}{\ln\frac1{\eps}} 
  \hat\nu_{c_*\eps}((\gamma)_{(c_*\eps)^s})\\
  =&
  \limsup_{\eps\to0}
    \hat\nu_{\eps}((\gamma)_{\eps^s})\le 
  c(1-s).
  \end{split}
\end{equation}
 As $s\in(0,1)$ was arbitrary, this concludes the proof of \eqref{eqepssnuhat3}.
 \end{proof}

We next come to the key construction for the proof of the upper bound. 
We define, given a pair $(\mu,\beta)\in\mathcal{M}^1_{\eps\calB}(\Omega)\times L^1(\Omega;\R^{3\times3})$,

\begin{equation}\label{eqdefeubeps}
 \Eub_\eps[\mu,\beta]:=\begin{cases}\displaystyle
                    \Elast[\beta, \Omega\setminus(\supp\mu)_{\eps}], 
                   & \text{ if $(\mu,\beta)\in\mathcal{A}^*_\eps$,  }\\
                   \infty, & \text{ otherwise,}
                  \end{cases}
\end{equation}
where $\Elast$ and $\mathcal{A}^*_\eps$ are defined in \eqref{eqdefFepsbetaA} and \eqref{Astar} respectively.

\begin{theorem}\label{theoubinternal}
 Let $\Omega\subseteq\R^3$ be a bounded Lipschitz set. Assume that $W$ obeys \HWFinite\ for some $p\in(1,2]$.
 Let
 $(\mu,\eta,Q)\in \calM_\calB^1(\Omega)\times L^2(\Omega;\R^{3\times 3})\times \SO(3)$ with $\mu=b\otimes t\mathcal{H}^1\LL\gamma$ and $\curl\eta=0$. Then for every sequence $\eps_k\to0$
 and any $(h_{\eps_k},\alpha_{\eps_k})$ satisfying \eqref{eqdefheps} there 
 is $(\mu_k,\beta_k)\in\calA_{\eps_k}^*$, with $\mu_k$  $(h_{\eps_k},\alpha_{\eps_k})$-dilute, such that 
 {\begin{equation}\label{eqweakconvmuepsmuub}
\frac1{\eps_k}\mu_k\weakstarto\mu         \hskip5mm
{\text{ in $\Omega'$, for all $\Omega'\subset\subset\Omega$}}
\end{equation}
and 
\begin{equation}\label{eqetaepstheoremub}
 \frac{Q^T\beta_k-\Id}{\eps_k(\ln\frac1{\eps_k})^{1/2}}
 \weakto \conteta \text{ weakly in } L^{q}_\loc(\Omega;\R^{3\times 3})
 \end{equation}
 for all $q\in [1,2)$.} 
 Moreover
 \begin{equation}\label{eq:upperbound1}
 \limsup_{k\to\infty} \frac{1}{\eps_k^2\ln\frac{1}{\eps_k}}\Eub_{\eps_k}[\mu_k,\beta_k]\le  \int_\gamma {\psi_\C^\rel (b,Qt)} d\calH^1+ \int_\Omega\frac12  {\C_Q}  \conteta\cdot \conteta dx
 \end{equation}
 where $\C:=D^2W(\Id)$, $\psi_\C^\rel$ is defined as in \eqref{psi-rel}, {$\C_Q$ as in \eqref{eqdefCQ}} and 
 \begin{equation}\label{eq:upperbound2}
 \limsup_{k\to\infty} 
 \frac{1}{\eps_k^2\ln\frac{1}{\eps_k}}
 \int_{\Omega\cap (\supp\mu_k)_{\eps_k}} \dist^q(\beta_k,\SO(3)) dx=0
 \end{equation}
 for all $q\in[1,2)$.
 
 Assume $W$ obeys \HWLin\ for some $p\in(1,2]$. Let
 $(\mu,\eta)\in \calM_\calB^1(\Omega)\times L^2(\Omega;\R^{3\times 3})$ with $\mu=b\otimes t\mathcal{H}^1\LL\gamma$ and $\curl\eta=0$. 
 Then the same holds with 
 {
\begin{equation}\label{eqetaepstheoremub2}
 \frac{\beta_k}{\eps_k(\ln\frac1{\eps_k})^{1/2}}
 \weakto \conteta \text{ weakly in } L^{q}_\loc(\Omega;\R^{3\times 3})
 \end{equation}
in place of \eqref{eqetaepstheoremub},}
 $\psiC^\rel(b,t)$ in place of $\psiC^\rel{(b,Qt)}$ {and $\C$ in place of $\C_Q$} in \eqref{eq:upperbound1} and
 $|\beta_k+\beta_k^T|^q$ in place
 of $\dist^q(\beta_k,\SO(3))$
 in \eqref{eq:upperbound2}.
 
\end{theorem}
\begin{proof}We prove the thesis in the finite case as the linear case can be treated similarly.
 As in \cite[Proposition 6.8]{ContiGarroniOrtiz2015}, the proof is divided in two steps. In the first one we provide an explicit construction attaining the unrelaxed limiting energy, and under the additional assumption that $\eta$ is essentially bounded.  
In the second step we conclude by using a diagonal argument.
 
 \medskip
 \textit{Step 1}: {We assume additionally that} $\mu=b\otimes t \mathcal{H}^1\LL\gamma\in\mathcal{M}^1_\calB(\R^3)$ is polyhedral, 
 transversal to $\partial\Omega$ in the sense of Lemma~\ref{lemmatransversal},
which implies $|\mu|(\partial\Omega)=0$,
 and
 that $\eta\in L^\infty(\R^3;\R^{3\times 3})$ with $\curl\eta=0$.
We first sketch the treatment of the rotation $Q$, ignoring $\eta$ at first. Our aim is to define $\beta_k$ 
as $Q(\Id+\eps_k\xi)$ for some matrix-valued field $\xi$, 
and then to expand the energy, using the invariance \eqref{eqWfWFQ} and then a Taylor series,  as
\begin{equation}
\begin{split}
 W(Q+\eps_kQ \xi)&= W((Q+\eps_kQ\xi)Q^T)
 =W(\Id + \eps_k Q\xi Q^T)\\
 &\simeq \frac12 \eps_k^2 \C (Q\xi Q^T) \cdot (Q\xi Q^T)
=\frac12 \eps_k^2 \C_Q \xi \cdot \xi ,
 \end{split}
\end{equation}
with $\C_Q$ defined as in \eqref{eqdefCQ}.
In particular, 
$\xi$ should be chosen to solve the linear problem with coefficients $\C_Q$, and to obey $\curl (Q\xi)=\mu$. 

We now start the construction.  Let $\xi\in L^{3/2}(\R^3;\R^{3\times3})$ be the solution to
 \begin{equation}
  \begin{cases}
   \curl{\xi}= {Q^T\mu},\\ 
  {\Div\C_Q\xi}=0,
  \end{cases}
 \end{equation}
{as in Theorem \ref{theoremsolr3}. Clearly $Q^T\mu=Q^Tb\otimes t \mathcal{H}^1\LL\gamma\in\mathcal{M}^1_\calB(\R^3)$
is also transversal to $\partial\Omega$.}
 We define
 \begin{equation}\label{eq:recovery}
 \mu_k:=\eps_k\mu\LL\Omega \quad\text{ and }\quad
{\beta_k:=Q+\eps_k\log^{1/2}\frac{1}{\eps_k}Q \eta +\eps_k{Q\xi .}
} \end{equation}
Then clearly
{\eqref{eqweakconvmuepsmuub} holds and}
$(\mu_k,\beta_k)\in\calA^*_{\eps_k}$. Since $\mu$ is polyhedral, 
by Theorem \ref{theoremsolr3}\ref{lemmasolr3lp}
for all $q\in[1,2)$ 
we have $\xi\in L^q(\Omega;\R^{3\times 3})$ which implies
\begin{equation}\label{convbeta}
 \frac{Q^T\beta_k-\Id}{\eps_k\log^{1/2}\frac{1}{\eps_k}}=
 \conteta+\frac{{\xi}}{\log^{1/2}\frac{1}{\eps_k}}
 \to\conteta\ \text{ strongly in } L^q(\Omega;\R^{3\times3}),
 \end{equation}
 and in particular 
\eqref{eqetaepstheoremub} holds.
 Similarly, 
 \begin{equation*}
 \begin{split}
 \limsup_{k\to\infty}\frac{1}{\eps_k^2\log\frac{1}{\eps_k}}\int_{\Omega\cap (\gamma)_{\eps_k}}&\dist^q(\beta_k,\SO(3))\dx\\
 &\le \limsup_{k\to\infty}\frac{1}{\eps_k^2\log\frac{1}{\eps_k}}\int_{\Omega\cap (\gamma)_{\eps_k}}|\beta_k-Q|^q\dx\\
 &= \limsup_{k\to\infty}\frac{1}{\ln\frac{1}{\eps_k}}\int_{\Omega\cap (\gamma)_{\eps_k}}|\beta|^q\dx=0.
 \end{split}
 \end{equation*}

Let $s\in(0,1)$ be fixed.
 By the growth condition on $W$ it holds
 \begin{equation}\label{eq:upp-bound}
 \begin{split}
 \frac{1}{\eps_k^2\log\frac{1}{\eps_k}}\Eub_{\eps_k}[\mu_k,\beta_k]\le \frac{1}{\eps_k^2\log\frac{1}{\eps_k}}&\Elast[\beta_k,\Omega\setminus(\gamma)_{\eps_k^s}]
 \\&+
 \frac{c}{\eps_k^2\log\frac{1}{\eps_k}}\int_{\Omega\cap (\gamma)_{\eps_k^s}\setminus {(\gamma\cap\Omega)_{\eps_k}}}|\beta_k-Q|^2\dx.
 \end{split}
 \end{equation}
 By Theorem \ref{theoremsolr3}\ref{lemmasolr3lp} 
and $\eta\in L^\infty(\Omega;\R^{3\times 3})$, for any $x\in\Omega\setminus(\gamma)_{\eps_k^s}$ we have
 \begin{equation*}
{|Q^T\beta_k(x)-\Id|=|\eps_k\log^{1/2}\frac{1}{\eps_k}\eta(x)+\eps_k\xi(x)|\le c(\eps_k\log^{1/2}\frac{1}{\eps_k}+\eps_k^{1-s})=:\delta_k. }
 \end{equation*}
 By frame indifference and Taylor expansion near the identity there is $\omega$   with $\omega(\delta)\to0$ as $\delta\to0$ such that
 \begin{equation*}
  \begin{split}
 W(\beta_k)&=W(\beta_kQ^T)=W(\Id+ (\beta_kQ^T-\Id))
 \\
 \le& \frac12 \C(\beta_kQ^T-\Id)\cdot (\beta_kQ^T-\Id)
  +\omega(\delta_k)|\beta_kQ^T-\Id|^2\\
  =& \eps_k^2\log\frac1{\eps_k}\frac12 \C_Q \eta\cdot \eta 
  +\eps_k^2 \frac12\C_Q \xi \cdot \xi \\
  &+\eps_k \log^{1/2}\frac{1}{\eps_k} \C_Q \eta \cdot \xi 
  +\omega(\delta_k)|Q^T\beta_k-\Id|^2
  \end{split}
\end{equation*}
pointwise, and
 \begin{equation}\label{upperboundpreliminary}
 \begin{split}
 &\frac{1}{\eps_k^2\log\frac{1}{\eps_k}}\Elast[\beta_k,\Omega\setminus(\gamma)_{\eps_k^s}]
\le\int_{\Omega\setminus(\gamma)_{\eps_k^s}}\frac12{\C_Q}\eta\cdot\eta\dx+ \frac{1}{\log\frac{1}{\eps_k}}\int_{\Omega\setminus(\gamma)_{\eps_k^s}}\frac12\C_Q\xi\cdot\xi\dx\\
 &\quad +\frac{c}{\log^{1/2}\frac{1}{\eps_k}}\int_{\Omega\setminus(\gamma)_{\eps_k^s}}
 |\conteta|\, |\xi|\dx+ \int_{\Omega\setminus(\gamma)_{\eps_k^s}}
 \omega(\delta_{k})\frac{|\log^{1/2}\frac{1}{\eps_k}\conteta+{\xi}|^2}{\log\frac{1}{\eps_k}}\dx.
 \end{split}
 \end{equation}
 Now  {\eqref{eqprop61omega} in} Proposition \ref{propupperboundcontunrelax2} yields
 \begin{equation}\label{eq1}
 \begin{split}\limsup_{k\to\infty}
 \frac{1}{\log\frac1{\eps_k^s}}\int_{\Omega\setminus(\gamma)_{\eps_k^s}}\frac12\C_Q\xi\cdot\xi\dx\le& \int_{\Omega\cap \gamma}{\psi_{\C_Q}(Q^Tb,t)}d\mathcal{H}^1\\
 =&
 \int_{\Omega\cap \gamma}{\psiC(b,Q t)}d\mathcal{H}^1
 \end{split}\end{equation}
 where in the last step we used Lemma~\ref{lemmarotatebetacurl}\ref{lemmarotatebetacurlpsic}.
 Since $\eta\in L^\infty(\Omega;\R^{3\times3})$ and $\xi\in L^1(\Omega;\R^{3\times3})$ we derive that 
 \begin{equation}\label{eq2}
 \lim_{k\to\infty} \frac{1}{\log^{1/2}\frac{1}{\eps_k}}\int_{\Omega\setminus(\gamma)_{\eps_k^s}}|\conteta|\,|\xi|\dx=0.
 \end{equation}
Additionally, recalling  \eqref{eqsupepsnuepsbr}
 \begin{equation}\label{eq3}\begin{split}
 &\limsup_{k\to\infty}\int_{\Omega\setminus(\gamma)_{\eps_k^s}}
 \frac{|\log^{1/2}\frac{1}{\eps_k}Q\conteta|^2+|\xi|^2}{\log\frac{1}{\eps_k}}\dx\\
 &=
 \lim_{k\to\infty}\int_{\Omega\setminus(\gamma)_{\eps_k^s}}
 {|\conteta|^2}+\frac{|{\xi}|^2}{\log\frac{1}{\eps_k}}\dx<\infty.
 \end{split}
 \end{equation}
 As $ \omega(\delta_k)\to0$, also the last error term vanishes in the limit.
 
 Finally, again from $\eta\in L^\infty(\Omega;\R^{3\times 3})$ and 
 {\eqref{eqepssnuhat}} in  Proposition~\ref{propupperboundcontunrelax2} we obtain
 \begin{equation}\label{eq4}
 \begin{split}
& \lim_{k\to\infty}\frac{1}{{\eps_k^2}\log\frac{1}{\eps_k}}\int_{\Omega\cap (\gamma)_{\eps_k^s}\setminus (\gamma)_{\eps_k}}|\beta_k-Q|^2\dx\\
 &\le 2
 \lim_{k\to\infty}\int_{\Omega\cap (\gamma)_{\eps_k^s}\setminus (\gamma)_{\eps_k}}|\eta|^2\dx
 +2
 \lim_{k\to\infty}
 \frac{1}{\log\frac{1}{\eps_k}}
 \int_{\Omega\cap (\gamma)_{\eps_k^s}\setminus (\gamma)_{\eps_k}}|{\xi}|^2\dx
 \\
& \le c(1-s).
 \end{split}
 \end{equation}
By Lemma~\ref{lemmatransversal}\ref{lemmatransversallimit} the same holds 
with the integration domain ${\Omega\cap (\gamma)_{\eps_k^s}\setminus (\gamma\cap\Omega)_{\eps_k}}$.
 Gathering together \eqref{eq:upp-bound}--\eqref{eq4},
we infer
 \begin{equation*}
 \limsup_{k\to\infty}\frac{1}{\eps_k^2\log\frac1{\eps_k}}\Eub_{\eps_k}[\mu_k,\beta_k]\le\int_\Omega\frac12\C_Q\conteta\cdot\conteta\dx+ s\int_{\Omega\cap\gamma}\psiC(b,Qt)d\mathcal{H}^1+c(1-s)
 \end{equation*}
for all $s\in(0,1)$, and hence for $s=1$.
 \medskip
 
 \textit{Step 2:} Let $\mu=b\otimes t\mathcal{H}^1\LL\gamma\in\MBCC(\Omega)$, 
 $\eta\in L^2(\Omega;\R^{3\times 3})$ with $\curl\eta=0$. We first approximate $\eta$, and then $\mu$.
 
We first show that we can construct a smooth sequence $\eta_j$ with $\curl\eta_j=0$ such that $\eta_j\to \eta$ strongly in $L^2(\Omega;\R^{3\times3})$.
  To see this, observe that since $\Omega$ is Lipschitz there are an open set $\omega$ with $\partial\Omega\subset\omega$ and a 
  bilipschitz map $\Phi:\omega\to\omega$ such that $\Phi(x)=x$ for $x\in\partial\Omega$ and
 $\Phi(\Omega\cap \omega)=\omega\setminus\overline\Omega$. 
 We extend $\eta$ to $\Omega\cup\omega$ by
 \begin{equation}
  \eta(x):=\eta(\Phi(x))D\Phi(x) \text{ for } x\in \omega\setminus\Omega.
 \end{equation}
Obviously $\eta\in L^2(\Omega\cup\omega;\R^{3\times 3})$. For every ball {$B\subset\omega$} such that $B\cap \Omega$ is simply connected there is $u\in W^{1,2}(B\cap\Omega;\R^3)$ such that $\eta=Du$ in $B\cap\Omega$, and $u\circ\Phi$ gives an extension of $u$ to $W^{1,2}(B\cup\Phi(B);\R^3)$. As $\eta=D(u\circ\Phi)$ on $\Phi(B)\setminus\Omega$ we obtain $\curl\eta=0$ on $\Omega\cup\omega$.
We then define $\eta_j$ as the mollification on scale $1/j$ of the extension. For every $j$ we have $\eta_j\in C^\infty(\overline{\Omega};\R^{3\times 3})$, $\curl\eta_j=0$, and
\begin{equation}\label{eqconvetajl2}
\lim_{j\to\infty} \|\eta_j-\eta\|_{L^2(\Omega)}
=0
\end{equation}
which implies
\begin{equation}
\lim_{j\to\infty}\int_{\Omega} \frac12 \C_Q \eta_j\cdot\eta_j dx = 
 \int_\Omega \frac12 \C_Q \eta\cdot \eta dx.
\end{equation}

We now turn to the measure $\mu$.
   By \cite[Theorem~3.1 and Lemma~2.3]{ContiGarroniMassaccesi2015} (see also \cite[Theorem~6.5 and Lemma~6.2]{ContiGarroniOrtiz2015}) we can find a sequence $\nu_j=\tilde b_j\otimes \tilde t_j\mathcal{H}^1\LL\tilde \gamma_j\in\MBCC(\R^3)$ such that $\nu_j\LL\Omega\weakstarto\mu$,
   \begin{equation}\label{eqnudeltaomega}
    \lim_{\delta\to0} |\nu_j|((\partial\Omega)_\delta)=0 \text{ for all $j$, }
   \end{equation}
   and 
  \begin{equation}\label{relaxed}
\limsup_{j\to\infty}\int_{\tilde \gamma_j\cap\Omega}\psiC(\tilde b_j,Q\tilde t_j)d\mathcal{H}^1\le \int_\gamma\psiC^\rel(b,Qt)d\mathcal{H}^1.
  \end{equation}
  By \cite[Lemma~6.4]{ContiGarroniOrtiz2015},
  which we can apply by
  Lemma~\ref{lemmacellproblemlb}\ref{lemmacellproblemlbtcont},
  for every $j$
  there are a polyhedral measure $\mu_j=b_j\otimes t_j\mathcal{H}^1\LL\gamma_j\in\MBCC(\R^3)$ and a bijective map $f^j\in C^1(\R^3;\R^3)$ such that 
  \begin{equation*}
   |f^j_{\sharp}\mu_j-\nu_j|(\R^3)\le\frac1j,
  \end{equation*}
  \begin{equation*}
   |Df^j(x)-\Id|+|f^j(x)-x|\le\frac1j\quad\text{for all }x\in\R^3,
  \end{equation*}
and, for some $r_j>0$,
   \begin{equation}\label{polyhedral-approx}
\int_{\gamma_j\cap(\Omega)_{r_j}}\psiC(b_j,Qt_j)d\mathcal{H}^1\le(1+c\frac1j) \int_{\tilde\gamma_j\cap\Omega}\psiC(\tilde b_j,Q\tilde t_j)d\mathcal{H}^1+c\frac1j,
  \end{equation}
  where $(\Omega)_{r_j}=\{x:\dist(x,\Omega)<r_j\}$.
Here $f^j_\sharp\mu_j$ denotes the push-forward of the current corresponding to the measure $\mu_j$, in the sense that
\begin{equation}
 f^j_\sharp(b_j\otimes t_j\mathcal{H}^1\LL\gamma_j)
 = \left(b_j\otimes 
 \frac{D_{t_j}f_j}{|D_{t_j}f_j|}\right)
 \circ f_j^{-1}  \calH^1\LL f_j(\gamma_j).
\end{equation}
{By Lemma~\ref{lemmatransversal}\ref{lemmatransversalgeneric} 
we can choose $f_j$ and $\mu_j$ such that $\mu_j$ is transversal to $\partial\Omega$
(possibly replacing $f_j$ by $f_j+a_j$, for some $a_j\to0$).}
In particular it turns out that $\mu_j\weakstarto\mu$ in $\Omega$. Moreover since $\mu_j$ is polyhedral it is $(h_{\eps_k},\alpha_{\eps_k})$-dilute for sufficiently large $k$.  By step 1 for every $j$ there is $(\mu_k^j,\beta_k^j)\in\calA_{\eps_k}^*$ constructed as in \eqref{eq:recovery} such that $\eps_k^{-1}\mu_k^j=\mu_j$ and 
 \begin{equation*}
\frac{Q^T\beta_k^j-\Id}{\eps_k\log^{1/2}\frac{1}{\eps_k}}\to\eta_j\quad\text{ strongly in } L^q(\Omega;\R^{3\times3}),
 \end{equation*}
 for all $q\in[1,2)$.
 Moreover
 \begin{equation}\label{limsupdiagonal1}
   \limsup_{k\to\infty}\frac{1}{\eps_k^2\log\frac{1}{\eps_k}}\Eub_{\eps_k}[\mu_k^j,\beta^j_k]\le\int_\Omega\frac12\C_Q\conteta_j\cdot\conteta_j\dx+ \int_{\Omega\cap\gamma_j}\psiC(b_j,Qt_j)d\mathcal{H}^1,
 \end{equation}
 and
  \begin{equation}\label{eqlimsupepsc}
 \begin{split}
 \limsup_{k\to\infty}\frac{1}{\eps_k^2\log\frac{1}{\eps_k}}\int_{\Omega\cap (\supp\mu_k^j)_{\eps_k}}\dist^q(\beta^j_k,\SO(3))\dx=0.
 \end{split}
 \end{equation}
Further, from \eqref{eqconvetajl2},  \eqref{polyhedral-approx} and \eqref{relaxed} it follows that
 \begin{equation*}
  \limsup_{j\to\infty} 
  \int_\Omega\frac12\C_Q\conteta_j\cdot\conteta_j\dx+ \int_{\Omega\cap\gamma_j}\psiC(b_j,Qt_j)d\mathcal{H}^1
  \le \int_\Omega\frac12\C_Q\conteta\cdot\conteta\dx+
  \int_\gamma\psiC^\rel(b,Qt)d\mathcal{H}^1.
 \end{equation*}
  To conclude it remains to construct a diagonal sequence.  Notice first that for any fixed $j$ there is $k_j\in\N$ such that for $k\ge k_j$ the following properties are satisfied:
 \begin{enumerate}
  \item $\mu_k^j\LL\Omega$ is $(h_{\eps_k},\alpha_{\eps_k})$-dilute;
  \item
 $$\left\|\frac{Q^T\beta_k^j-\Id}{\eps_k\log^{1/2}\frac{1}{\eps_k}}
 -\conteta_j \right\|_{L^q(\Omega;\R^{3\times3})}\le\frac1j
 \quad \text{ for all } q\in[1,2-\frac1j];$$
\item 
 \begin{equation*}
  \frac{1}{\eps_k^2\log\frac{1}{\eps_k}}\int_{\Omega\cap(\supp\mu_k^j)_{\eps_k}}\dist^q(\beta^j_k,\SO(3))\dx\le \frac1j\quad \text{ for all } q\in[1,2-\frac1j];
 \end{equation*}
 \item 
 \begin{equation*}
\frac{1}{\eps_k^2\log\frac{1}{\eps_k}}\Eub_{\eps_k}[\mu_k^j,\beta^j_k]\le \int_\Omega\frac12\C_Q\conteta_j\cdot\conteta_j\dx+ \int_{\Omega\cap\gamma_j}\psiC(b_j,Qt_j)d\mathcal{H}^1+\frac1j.
 \end{equation*}
   \end{enumerate}
 Without loss of generality, $k_j>k_{j-1}$.
 Finally for every $k>0$ we take $\tilde\mu_k:=\mu_k^j$ and $\tilde\beta_k:=\beta_k^j$ if $k\in[k_j,k_{j+1})\cap\mathbb{N}$, and this concludes the proof.
\end{proof}

\begin{remark}\label{remubinternal2}
 The same result as in 
 Theorem~\ref{theoubinternal} holds 
 if $\Eub_\eps$ is replaced by
\begin{equation}\label{eqdefeubeps2}
E^\moll_\eps[\mu,\beta]:=\begin{cases}\displaystyle
  \Elast[\beta, \Omega], 
  & \text{ if $(\mu,\beta)\in\mathcal{A}^\moll_\eps$,
  }\\
  \infty, & \text{ otherwise,}
 \end{cases}
\end{equation}
where $\Elast$ and $\mathcal{A}^\moll_\eps$ are defined in \eqref{eqdefFepsbetaA} and \eqref{eq:Amoll} respectively.
 To see this, it suffices to replace
the definition of $\beta_k$ in \eqref{eq:recovery} with
  \begin{equation}\label{eq:recovery1}
  \hat\beta_k:=Q+\eps_k\log^{1/2}\frac{1}{\eps_k}Q\eta+\eps_k Q\xi*\varphi_\eps.
 \end{equation}
Then one applies {\eqref{eqprop61omega} in}
Proposition~\ref{propupperboundcontunrelax2} 
to a larger set $(\Omega)_{r}$, for some $r>0$. 
{For almost all $r$ we have $|\mu|(\partial((\Omega)_r))=0$, so that one}
obtains instead of  \eqref{eq1}
 \begin{equation}\label{eq1st} 
 \limsup_{k\to\infty}
 \frac{1}{\log\frac1{\eps_k^s}}\int_{(\Omega)_{r}\setminus(\gamma)_{\eps_k^s}}\frac12\C_Q\xi
 \cdot\xi\dx\le \int_{(\Omega)_{r}\cap\gamma}\psiC(b,Qt)d\mathcal{H}^1.
 \end{equation}
 {By convexity this gives
 \begin{equation}\label{eq1stb} 
 \limsup_{k\to\infty}
 \frac{1}{\log\frac1{\eps_k^s}}\int_{\Omega\setminus(\gamma)_{2\eps_k^s}}\frac12\C_Q(\varphi_{\eps_k}\ast\xi)\cdot
 (\varphi_{\eps_k}\ast\xi)\dx\le \int_{(\Omega)_{r}\cap\gamma}\psiC(b,Qt)d\mathcal{H}^1
 \end{equation}
 for almost all $r>0$, and since $|\mu|(\partial\Omega)=0$ 
 also with the integral over $\Omega\cap\gamma$ in the right-hand side. In the rest of the estimates we need to use $2\eps_k^s$ instead of $\eps_k^s$, without any significant change.}
{The estimate in \eqref{eq4} carries over by convexity, if one restricts the domain by $\eps_k$. As the exponent $s$ was generic, this is only relevant on the ``inner'' side, and it remains to deal with $(\gamma)_{2\eps_k}$.}
{This is finally treated by Remark~\ref{remarkbetamollcore}, which shows that
 \begin{equation}
 \limsup_{k\to\infty}
 \frac{1}{\eps_k^2\log\frac1{\eps_k}}\int_{(\gamma)_{2\eps_k}}
 W(\hat\beta_k) dx
 \le 
 \limsup_{k\to\infty}
 \frac{c}{\log\frac1{\eps_k}}\int_{(\gamma)_{2\eps_k}}
 (|\eta|^2+|\xi\ast\varphi_{\eps_k}|^2) dx=0.
 \end{equation}
}
{Step 2} is unchanged.
 \end{remark}

\section{Proofs of the main results}\label{sec-mainproofs}

We start with the proof of our main result, Theorem~\ref{thm:Gamma-limit}.

\begin{proof}[Proof of  Theorem~\ref{thm:Gamma-limit}]
Compactness follows from 
Proposition~\ref{propcptmubdbeta}.
 For the lower bound we use Theorem~\ref{theoremlowerboundsec5} with $\rho_\eps=\eps$
 and  {\eqref{eqelbesubcr}}.
For the upper bound we use Theorem~\ref{theoubinternal}, observing that
\begin{equation*}
 F_\eps^\subcr[\mu,\beta]\le \frac{1}{\eps^2\log\frac{1}{\eps}}\Eub_\eps[\mu,\beta] + \frac{c}{\eps^2\log\frac{1}{\eps}}\ \int_{\Omega\cap (\supp\mu)_\eps} \dist^p(\beta,\SO(3)) dx.
\end{equation*}
The geometrically linear case is proven analogously.
\end{proof}

We now turn to Theorem~\ref{thm:Gamma-limit2}.

\begin{proof}[Proof of Theorem~\ref{thm:Gamma-limit2}]
 (i): As in the proof of Theorem~\ref{thm:Gamma-limit}, the lower bound follows from Theorem~\ref{theoremlowerboundsec5} and the fact that $$\frac{1}{\eps^2\log\frac{1}{\eps}}\Elb_{\rho,\eps}[\mu,\beta]\le F_\eps^*[\mu,\beta]$$ for every $\rho\ge\rho_\eps$ which follows as above from $\calA_\eps^*\subseteq\calA_{\rho,\eps}^\core$ for every $\rho>0$. 

For the upper bound we use Theorem~\ref{theoubinternal}, observing that
{$\rho_\eps\ge\eps$ implies}
$$F_\eps^*[\mu,\beta]\le \frac{1}{\eps^2\log\frac{1}{\eps}}\Eub_\eps[\mu,\beta].$$

 (ii): For the lower bound we observe that 
  $$\frac{1}{\eps^2\log\frac{1}{\eps}}\Elb_{{\rho_\eps,\eps}}[\mu,\beta]\le F_\eps^\core[\mu,\beta].$$
 Hence the thesis follows again by Theorem~\ref{theoremlowerboundsec5}. 
 
 For the upper bound we use Theorem~\ref{theoubinternal}, observing that $$F_\eps^\core[\mu,\beta]\le \frac{1}{\eps^2\log\frac{1}{\eps}}\Eub_\eps[\mu,\beta]$$
 which again is a consequence of  $\calA^*_\eps\subseteq{\calA_{\rho_\eps,\eps}^\core}$
{(see Lemma~\ref{lem:comparison-admissible-pairs}\ref{lem:comparison-admissible-pairsstarcore}).}

 (iii): For lower bound we use Theorem~\ref{theoremlowerboundsec5} and the fact that $$\frac{1}{\eps^2\log\frac{1}{\eps}}\Elb_{\rho,\eps}[\mu,\beta]\le F_\eps^\moll[\mu,\beta]$$ for every $\rho\ge\eps$ which follows 
 from $\calA_\eps^\moll\subseteq\calA_{\rho,\eps}^\core$ (see Lemma~\ref{lem:comparison-admissible-pairs}\ref{lem:comparison-admissible-pairsmollcore}). 
 
 For the upper bound we use the same argument as in Theorem~\ref{theoubinternal}, as discussed in  Remark~\ref{remubinternal2} 
 observing that  \begin{equation*}
  F_\eps^\moll[\mu,\beta]= \frac{1}{\eps^2\log\frac{1}{\eps}}E^\moll_\eps[\mu,\beta] .
 \end{equation*}
 
\end{proof}

\section*{Acknowledgements}

This work  was partially funded by the Deutsche Forschungsgemeinschaft (DFG, German Research Foundation)
{\sl via} project 211504053 - SFB 1060, 
project 390685813 - GZ 2047/1 - HCM3,
and project 2044-390685587 - ``Mathematics M\"unster: Dynamics--Geometry--Structure'' and by  PRIN 2017BTM7SN$\_$004 ``Variational methods for stationary and evolution problems with singularities and interfaces''.


\end{document}